\documentclass[11pt,a4paper]{amsart}
\usepackage[english]{babel}
\usepackage{microtype}
\usepackage{amsmath}
\usepackage{amscd}
\usepackage{amssymb}
\usepackage{amsthm}
\usepackage[table]{xcolor}
\usepackage[matrix, arrow,curve]{xy}
\usepackage{verbatim}

\usepackage{soul}

\usepackage{tikz}
\usepackage{tikz-cd}
\usetikzlibrary{arrows}

\makeatletter
\def\pgf@stroke@inner@line{%
 \let\pgf@temp@save=\pgf@strokecolor@global
 \pgfsys@beginscope%
 {%
  \pgfsys@roundcap%
  \pgfsys@setlinewidth{\pgfinnerlinewidth}%
  \pgfsetstrokecolor{\pgfinnerstrokecolor}%
  \pgfsyssoftpath@invokecurrentpath%
  \pgfsys@stroke%
 }%
 \pgfsys@endscope%
 \global\let\pgf@strokecolor@global=\pgf@temp@save
}
\makeatother

\usepackage[inline]{enumitem}

\usepackage{xr} %
\usepackage{hyperref}
\usepackage[capitalise,nosort,nameinlink]{cleveref}

\definecolor{darkgreen}{rgb}{0.0, 0.2, 0.13}
\newcommand{\annette}[1]{{\color{orange} Annette: XYZ #1}}
\newcommand{\johan}[1]{{\color{red} Johan: XYZ #1}}

\newtheorem{lemma}{Lemma}[section]
\newtheorem*{lemma*}{Lemma}
\newtheorem*{thm*}{Theorem}
\newtheorem*{rem*}{Remark}
\newtheorem{proposition}[lemma]{Proposition}
\newtheorem{thm}[lemma]{Theorem}
\newtheorem{cor}[lemma]{Corollary}
\newtheorem{theorem}[lemma]{Theorem}
\newtheorem{conjecture}[lemma]{Conjecture}

\theoremstyle{definition}
\newtheorem{defn}[lemma]{Definition}

\newtheorem{setting}[lemma]{Setting}

\newtheorem{ex}[lemma]{Example}

\newtheorem{rem}[lemma]{Remark}

\newtheorem{remark}[lemma]{Remark}
\newtheorem{question}[lemma]{Question}

\newcommand{\ssm}{\smallsetminus}

\newcommand{\Mor}{\mathrm{Mor}}

\newcommand{\Hom}{\mathrm{Hom}}

\newcommand{\Spec}{\mathrm{Spec}}
\newcommand{\isom}{\cong}
\newcommand{\id}{\mathrm{id}}
\newcommand{\ohne}{\smallsetminus}
\newcommand{\tensor}{\otimes}
\newcommand{\cone}{\mathrm{Cone}}
\newcommand{\tot}{\mathrm{Tot}}
\newcommand{\Sm}{\mathrm{Sm}}
\newcommand{\SmAff}{\mathrm{SmAff}}
\newcommand{\SmProj}{\mathrm{SmProj}}
\newcommand{\dR}{\mathrm{dR}}
\newcommand{\DR}{\mathrm{DR}}
\newcommand{\RdR}{R\Gamma_\dR}
\newcommand{\an}{\mathrm{an}}
\newcommand{\interior}{\mathrm{int}}
\newcommand{\sing}{\mathrm{sing}}
\newcommand{\rd}{\mathrm{rd}}
\newcommand{\Per}{\mathcal{P}} %
\newcommand{\Pnaive}{\Per_{\mathrm{nv}}}
\newcommand{\Pgnaive}{\Per_{\mathrm{gnv}}}
\newcommand{\Pabs}{\Per_{\mathrm{abs}}}
\newcommand{\Pcoh}{\Per_{\mathrm{coh}}}
\newcommand{\Plog}{\Per_{\mathrm{log}}}
\newcommand{\Psmaff}{\Per_{\SmAff}}
\newcommand{\Pmot}{\Per_{\mathrm{mot}}}

\newcommand{\alg}{\mathrm{alg}}

\newcommand{\vol}{\mathrm{vol}}

\newcommand{\Q}{\mathbb{Q}}
\newcommand{\Qbar}{{\overline{\Q}}}
\newcommand{\Z}{\mathbb{Z}}
\newcommand{\R}{\mathbb{R}}
\def\C{\mathbb{C}}
\newcommand{\Pe}{\mathbb{P}} %
\newcommand{\A}{\mathbb{A}}
\newcommand{\Na}{\mathbb{N}}
\newcommand{\IN}{\Na}
\newcommand{\Oh}{\mathcal{O}}
\newcommand{\Eh}{\mathcal{E}}
\newcommand{\Sh}{\mathcal{S}}

\newcommand{\Def}{\mathrm{Def}}

\newcommand{\Blo}{\mathrm{Blo}}
\newcommand{\OBl}{\mathrm{OBl}}

\newcommand{\trdeg}{\mathrm{trdeg}}

\newcommand{\cc}[1]{\mathrm{cc}({#1})}
\newcommand{\md}{\,\mathrm{d}}

\newcommand{\e}{\mathrm{e}}

\newcommand{\Bcirc}{B^\circ}
\newcommand{\Bsharp}{B^\sharp}
\newcommand{\Ptilde}{{\tilde{\Pe}^1}}

\newcommand{\reg}{\mathrm{Reg}}

\def\xref#1{\cref{#1}}
\def\Aref#1{\cref{#1}}
\def\Bref#1{\cref{#1}}
\def\partAorB#1#2{#2}

\begin{document}
\title{Exponential periods and o-minimality}
\author{Johan Commelin}
\address{Mathematisch Instituut, Utrecht University, Budapestlaan 6, 3584CD Utrecht, The Netherlands}
\email{j.m.commelin@uu.nl}
\author{Philipp Habegger}
\address{Department of Mathematics and Computer Science, University of
Basel, 4051~Basel, Switzerland}
\email{philipp.habegger@unibas.ch}
\author{Annette Huber}
\address{Math. Institut, Universität Freiburg, Ernst-Zermelo-Str. 1, 79102~Freiburg, Germany}
\email{annette.huber@math.uni-freiburg.de}
\date{\today}

\begin{abstract}
 Let $\alpha \in \C$ be an exponential period.
 This paper shows that the real and imaginary part of $\alpha$
 are up to signs volumes of sets definable without parameters
 in the o-minimal structure generated by
 the real exponential function and ${\sin}|_{[0,1]}$.
 This is a weaker analogue of the precise
 characterisation of ordinary periods
 as numbers whose real and imaginary part
 are up to signs volumes of $\Q$-semi-algebraic sets.

 Furthermore, we define a notion of naive exponential period and
 compare it to the existing notions using cohomological methods.   
 In particular, naive exponential periods are the same
 as periods of exponential Nori motives,
 which justifies that the definition of naive exponential periods
 singles out the correct set of complex numbers
 to be called \emph{exponential} periods.
\end{abstract}

\maketitle

\tableofcontents

\section*{Introduction}

Exponential periods are, roughly speaking, complex numbers of the form
\begin{equation}\label{eq:rough} \int_\sigma \e^{-f}\omega\end{equation}
where $\omega$ is an algebraic differential form, $f$ an algebraic function and
$\sigma$ a domain of integration of algebraic
nature.
They have a conceptual interpretation as entries of the period matrix between
twisted de Rham cohomology and rapid decay homology;
more on this later.

The aim of this paper is
to give several definitions that make (\ref{eq:rough}) precise,
and to compare them.
Our main result is the following theorem.

\begin{thm}[\xref{thm:compare_all}]
 Let $k\subset\C$ be a subfield such that $k$ is algebraic over $k_0=k\cap \R$.
 The following subsets of $\C$ agree:
 \begin{enumerate}
  \item naive exponential periods over $k$;
  \item cohomological exponential periods of triples $(X,Y,f)$
   where $X$ is a smooth variety over~$k$,
   $Y \subset X$ is a simple normal crossings divisor
   and $f \in \Oh(X)$ is a regular function;
\item cohomological exponential periods of triples $(X,Y,f)$ where $X$ is an arbitrary variety over~$k$, $Y\subset X$ a closed subvariety and $f\in \Oh(X)$ is a regular function;
  \item periods of effective exponential Nori motives over $k$.
 \end{enumerate}
 Additionally, for every such number its real and imaginary part
 are up to signs volumes of compact subsets of $\R^n$  definable over
 $k_0$ 
 in the o-minimal structure $\R_{\sin,\exp}$.
\end{thm}

The condition on $k$ is natural from the point of view of real semi-algebraic geometry going into the definition of naive exponential periods.
The most interesting case from the number theoretic point of view is $k=\Q$, or $\Qbar$ or $\Qbar\cap\R$,
which produce the same subset of exponential periods of~$\C$.

Let us now explain the notions appearing in this theorem.

\subsection{Naive exponential periods}
We propose the following very explicit definition
as one way of making (\ref{eq:rough}) precise.

\begin{defn}\label{defn:naive}\label{defn:naive2}
 Let $k\subset\C$ be a subfield such that $k$ is algebraic over $k \cap \R$.

 A \emph{naive exponential period} over $k$
 is a complex number of the form
 \[ \int_G \e^{-f} \omega \]
 where $G \subset \C^n$ is a pseudo-oriented (not necessarily compact) closed
 $(k \cap \R)$-semi-algebraic subset,
 $\omega$ is a rational algebraic differential form on $\A^n_k$
 that is regular on $G$
 and $f$ is a rational function on $\A^n_k$ such that $f$ is regular and proper
 on $G$
 and, moreover,
 $f(G)$ is contained in a strip
 \[
  S_{r,s}=\{z\in\C \mid \Re(z)>r, |\Im(z)|<s\},
 \]
 for some real numbers $r, s$.
\end{defn}
A pseudo-orientation on $G$ is the choice of an orientation on a
$(k\cap\R)$-semi-algebraic open subset whose complement has positive
codimension (and hence measure~$0$), see \cref{orientation}.

We check that these integrals converge absolutely.
In the case $f = 0$, we recover
the notion of an (ordinary) naive period
as introduced by Friedrich in \cite{Fr},
see \cite[Definition~12.1.1]{period-buch}\footnote{Unfortunately, the definition is misstated in loc. cit. See the erratum at \url{http://home.mathematik.uni-freiburg.de/arithgeom/preprints/buch-errata/buch.html}.}.
The definition of a naive exponential period
is not identical to the definition given by Kontsevich--Zagier
in~\cite[\S4.3]{kontsevich_zagier}.
See \cref{contrast_with_KZ} for more details about the difference.

\subsection{On o-minimality}
In his ``Esquisse d'un Programme'', Grothendieck set forth the need
for, and the principles of, some form of ``tame'' topology.
O-minimality provides a good theory of ``tame'' subsets of~$\R^n$,
avoiding Cantor sets, fractals, the graph of a space-filling curve and
$\sin(1/x)$.
In recent years,
o-minimality has seen spectacular applications in algebraic geometry,
most notably as an important tool
in the proof of the Andr\'e--Oort conjecture for
$\mathcal{A}_g$
via the Pila--Zannier strategy,
see the survey~\cite{KlinglerAOSurvey}.

The `o' in ``o-minimality'' stands for ``order''.
The concept was first introduced in work of Van den Dries~\cite{Dries_1984_Remarks_Tarski}
and Pillay--Steinhorn~\cite{PillaySteinhorn}
at about the same time that Grothendieck was writing his
``Esquisse d'un Programme''.
We recall the definition and some basic properties
    of o-minimal structures in Section \ref{ssec:omin}.

By work of Wilkie \cite{Wilkie_1996_Model_completeness}
and Van den Dries and Miller \cite{Dries_Miller_1994_real_exponential_field}, 
 the structure of subsets of
$\R^n$ defined using the quantifiers $\forall,\exists,$ the basic operations $+,\cdot,<$, the elements of $k\cap\R$,
the real exponential function $\exp$,
and the restriction of the analytic function $\sin$
to the bounded interval $[0,1]$ is an example of an o-minimal structure.
We denote it by $\R_{\sin,\exp,k}$.

\begin{thm}[See Theorem~\ref{thm:naive_is_volume}]
 Let $\alpha$ be a naive exponential period over~$k$.
 Then its real and imaginary part are up to signs
 volumes of compact subsets of $\R^n$
 definable in $\R_{\sin,\exp,k}$.
\end{thm}

This generalises a result for ordinary periods:
their real and imaginary part are volumes of compact semi-algebraic sets,
see \cite[Proposition~12.1.6]{period-buch} together with \cite{viu-sos}. 
There is a significant difference though:
in the case of ordinary periods, we also have the converse implication.
The volume of a compact $\Q$-semi-algebraic set is by definition a naive period.
This is no longer clear or even expected in the exponential setting.
The definable subsets appearing in the theorem are of a special shape.
For example, we do not need to iterate the functions $\exp$ and ${\sin}|_{[0,1]}$.
The number $\e^\e$ is definable in the o-minimal structure
and hence also appears as a volume.
We do not expect it to be an exponential period.

\begin{question}
  Is there a natural way to characterise definable sets
  whose volumes are naive exponential periods?  
\end{question}

\subsection{Exponential periods and cohomology}
The origins of the theory of exponential periods
lie in a version of Hodge theory for vector bundles with irregular connections.
To our knowledge such a theory was first considered by Deligne, see \cite[p.~17]{irregular}.
A systematic study of the period isomorphism was started by Bloch and Esnault in~\cite{bloch_esnault_00},
and fully developed by Hien~\cite{hien-surfaces}.
He establishes  a period isomorphism
between de Rham cohomology of the connection and a suitable homology theory.
The special and central case of exponential connections
is treated by Hien and Roucairol~\cite{hien-roucairol}.
If $X$ is a smooth variety over a field $k\subset\C$
and $f\in\Oh(X)$ a regular function,
they consider the twisted de Rham complex $\Omega^*_f$
with differential $\omega\mapsto d\omega-df\wedge \omega$.
Its hypercohomology is \emph{twisted de Rham cohomology}.
They define rapid decay homology of $X^\an$ (see Section~\ref{ssec:rd})
taking the role of singular cohomology in the classical case
and a period pairing
\[ H_n^\rd(X,\Q)\times H^n_\dR(X,f)\to\C\]
inducing a perfect pairing after extending scalars to $\C$.
As in the classical case,
the theory can be extended to singular varieties and also relative cohomology.

The numbers in the image of the pairing are the exponential periods.
Their study in their own right
was proposed by Kontsevich and Zagier
in the last paragraph of~\cite{kontsevich_zagier}.

Ordinary periods have an even more conceptual interpretation
 via the torsor $T$ of isomorphisms between the de Rham realisation and the Betti realisation,
two fibre functors on the Tannakian category of mixed (Nori) motives. The period map defines a $\C$-valued point in $T(\C)$. We recover the period algebra as the image of $\Oh(T)\to\C$,
see \cite{period-buch}.
The same picture also applies in the case of exponential periods.
Fres\'an and Jossen have developed
a fully fledged theory of exponential motives in~\cite{fresan-jossen}. 
Their book also contains a very accessible account of the constructions
and the proof of the period isomorphism.
They also give many examples
of interesting numbers that appear as exponential periods.

We prove:

\begin{thm} [Propositions~\ref{prop:coh_implies_naive} and~\ref{naive_is_effective}
]
 A complex number $\alpha$ is a naive exponential period over $k$ if and only if
 there is a smooth variety $X$ over $k$,
 a simple normal crossings divisor $Y$,
 and a regular function $f\in\Oh(X)$ such that $\alpha$ is in the image of the period pairing
 \[ H_n^\rd(X,Y,\Q)\times H^n_\dR(X,Y,f)\to \C.\]
\end{thm}

Again this generalises the result for ordinary periods, see \cite[Theorem~12.2.1]{period-buch}.
Actually, the theorem also holds for general $X$ and~$Y$
or even all periods of effective exponential Nori motives,
see \xref{thm:compare_all}.

The general proof is quite technical,
and we therefore include the arguments in the curve case in \cref{sec:curves}.
This special case is more accessible,
yet already contains the main ideas.

\subsection{Method of proof}
The global strategy is similar to the case of ordinary periods.
Algebraic varieties admit triangulations by semi-algebraic simplices.
This allows us to represent homology classes by semi-algebraic sets.
In the simplest case, the period pairing on cohomology has the shape
\[ (\sigma,\omega)\mapsto \int_\sigma \e^{-f}\omega,\]
suggesting the relation to naive periods.
Conversely, the Zariski closure of a semi-algebraic set $G$
is an algebraic variety $X$,
and the Zariski closure of its boundary is a closed subvariety $Y \subset X$.

The main new tool compared to the classical case
is the real oriented blow-up of a smooth analytic variety at some divisor.
In the simplest case of $\Pe^1$ and the divisor $\infty$,
it is the compactification of $\C$ by a circle at infinity.
The points correspond to the directions of half rays.
Its use is of long standing in the theory of irregular connections.
Hien and Roucairol and also the exposition of Fres\'an--Jossen
use it to establish the period isomorphism in the exponential case.
Indeed, rapid decay homology  of $X$ can be computed
as the homology of a certain partial compactification
$\Bcirc(X^\an,f)$
of $X^\an$ relative to its boundary, see \Cref{prop:fj_rd}.
For details on $\Bcirc(X^\an,f)$
see \cref{defn:our_Bcirc} and \cref{ssec:orbmc}.
It is still semi-algebraic, more precisely,
a semi-algebraic manifold with corners.

However, this is not yet enough to bound the imaginary part of $f(G)$,
something that is crucial in showing
that $\int_G \e^{-f}\omega$ is the volume of a definable set
in the o-minimal structure
$\R_{\exp,\sin,k}$.
Recall that the complex exponential is \emph{not} definable,
only the real exponential and $\sin$ (or $\cos$) restricted to bounded intervals.
We introduce a smaller semi-algebraic subset $\Bsharp(X,f)$ of $\Bcirc(X,f)$.
The actual key step in the proof of our main theorem
is the comparison between the homology of $\Bsharp(X,f)$ and $\Bcirc(X,f)$
in \xref{prop:comp_homol}.
In the simplest case, they agree because a half-circle is contractible to a single point.

There are two reasons for the considerable length of this paper:
on the one hand, we aim 
for readers without a background
in o-minimality and/or in the classical theory of periods
and have chosen to reproduce definitions from the literature
and to give detailed arguments and references.
We have also added a section on the case of curves
that is not needed for the proof of the main theorems,
but should be more accessible and still uses all of the main ideas.

On the other hand, we ran into many technical problems.
\begin{itemize}
 \item For example, we do not know if the real oriented blow-up of a smooth variety
  can be embedded into $\R^n$ preserving both the semi-algebraic and differentiable structure.
  Instead we introduce the notion of a semi-algebraic (or more general: definable) manifold,
  at the price of having to extend some results that are well-known for semi-algebraic subsets of $\R^n$ to the manifold setting.
 \item The standard triangulation results in semi-algebraic geometry
  or for sets definable in an o-minimal structure only give facewise differentiability of the simplices.
  This is not strong enough for a straightforward application of Stokes's theorem---something that we need for a well-defined period pairing depending only on homology classes.
  Our way out is by a result of  Czapla--Paw\l{}ucki \cite{omin-triang},
   who prove the existence of
  o-minimal $C^1$-triangulations.
  We can then use a subtle version of Stokes's theorem proved by Whitney in~\cite{whitney}
  for ``regular'' differentials on $C^1$-manifolds.
 \item Finally,
  the period isomorphism has a simple description only in the case of a smooth affine variety.
  The general case is handled by hypercovers.
  This involves some checking of strict compatibilities between our real oriented blow-ups and their subspaces
  and a  check that the abstract period pairing is still realised by
  integration.
\end{itemize}

\subsection{Structure of the paper}

The following diagram explains the global structure of the paper,
and how the different theorems contribute to the main comparison result.

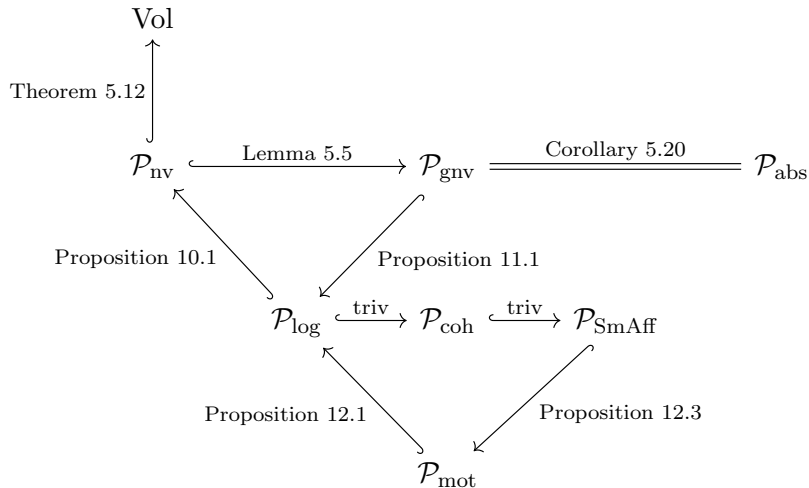
\[
 \begin{tikzcd}
  \text{Vol}
  \\ \\
  \partAorB{\hyperref[defn:naive]{\Pnaive}}{\hyperref[defn:naive2]{\Pnaive}}
  \ar[uu, hook, "\text{\Aref{thm:naive_is_volume}}"]
  \ar[rr, hook, "\text{\Aref{naive-is-generalised}}"]
  &
  & \partAorB{\hyperref[defn:gnaive]{\Pgnaive}}{\hyperref[defn:naive2]{\Pgnaive}}
  \ar[ddl, hook, "\text{\Bref{naive_is_effective}}"]
  \ar[rr, equal, "\text{\Aref{cor:abs=naive}}"]
  &
  & \partAorB{\hyperref[defn:abs]{\Pabs}}{\hyperref[defn:naive2]{\Pabs}}
  \\ \\
  & \partAorB{\hyperref[defn:coh_period]{\Plog}}{\hyperref[defn:coh_period]{\Plog}}
  \ar[uul, hook, "\text{\Bref{prop:coh_implies_naive}}"]
  \ar[r, hook, "\text{triv}"]
  & \partAorB{\hyperref[defn:coh_period]{\Pcoh}}{\hyperref[defn:coh_period]{\Pcoh}}
	 \ar[r, hook, "\text{\hyperref[rem:is_smaff]{Rmk}}", "\labelcref{rem:is_smaff}"']
  & \partAorB{\Psmaff}{\hyperref[defn:smaff]{\Psmaff}}
  \ar[ddl, hook, "\text{\Bref{smaff_sub_mot}}"]
  \\ \\
  &
  & \partAorB{\Pmot}{\hyperref[sec:concl]{\Pmot}}
  \ar[uul, hook, "\text{\Bref{mot_sub_coh}}"]
 \end{tikzcd}
\]

After settling notation in \xref{sec:not}, we review o-minimal structures for those readers not familiar with the theory in \xref{ssec:omin}. In \xref{sec:definable-manifold}, we introduce the notion of a definable $C^p$-manifold with corners because rapid decay homology has a natural description in such terms. We set up the theory of integration of differential forms. The main result is \xref{thm:omin_volume}: certain integrals can be expressed as volumes.

In \xref{realblow}, we review the construction of the real oriented blow-up and show that it is a semi-algebraic manifold with corners.

\xref{sec:naive} discusses naive exponential periods and their variants, in particular the issue of their convergence. By applying the results of \xref{sec:definable-manifold} we establish that they can be expressed as volumes, the first step in our comparison result.

In \xref{sec:exp}, we review the definition of rapid decay homology and twisted de Rham cohomology and the period isomorphism, concentrating on the more accessible smooth affine case. The technical \xref{sec:triangle} discusses existence and properties of triangulations of semi-algebraic manifolds. This is crucial input for the inclusion of the set of cohomological exponential periods into naive exponential periods.

\Cref{sec:curves} proves part of the upper triangle in the case of curves.
In this special case the main ideas of the proof are present,
but several delicate problems are avoided.

For the remainder of the paper, the technical level is notched up.
\Cref{sec:exp_gen} extends
the definition of cohomological exponential periods
to arbitrary pairs $(X, Y)$
of a variety~$X$ and a closed subvariety $Y \subset X$.
\cref{sec:coh_is_naive} is devoted to proving $\Plog(k) \subset \Pnaive(k)$,
whereas
\cref{sec:gnaive_is_coh} shows the inclusion $\Pgnaive(k) \subset \Plog(k)$.
Finally, in \cref{sec:concl} we
prove the remaining parts, which are all very formal,
and glue all the pieces together to obtain the main theorem.

\subsection{Outlook}
Our comparison results point to a deeper relation between periods and o-minimal theory.
While the case of ordinary periods---with their incarnations as entries of periods matrices or as volumes of semi-algebraic sets---might be seen as a coincidence, this second instance suggests that this is not the case.
Bakker, Brunebarbe, Klingler and Tsimerman
have been pursuing a project
of making a systematic use of tame geometry in Hodge theory
and apply it successfully to questions related to the Hodge
conjecture. A central tool was their  GAGA theory merging complex
spaces with o-minimal geometry. We hope that de Rham cohomology and the period isomorphism can also be extended to  a suitable category of definable spaces in a compatible way with the case of algebraic varieties, providing a new point of view on period numbers. 
The note \cite{huber_sa_motives} shows that this cannot be achieved on the naive candidate category  of all definable spaces. The category of definable complex spaces (or a suitable subcategory) is a better candidate.

\subsection*{Acknowledgements}
Many thanks to Amador Martin-Pizarro for teaching two of us (Commelin and 
Huber) not only the formalism but also the intuition of o-minimal theory. 
The review of o-minimal theory in \cref{ssec:omin} owes a lot to his talks.
We also thank Fabrizio Barroero and Reid Barton for discussions on
o-minimality
and Lou van den Dries for carefully explaining aspects
 of $C^p$-cell decomposition and triangulation.

We thank Marco Hien, Ulf Persson and Claus Scheiderer for answering questions on the real oriented blow-up.
Stefan Kebekus shared his insights on blow-ups in algebraic geometry.
Finally, we appreciated the help of Nadine Gro\ss{}e with the theory of integration.

We thank Amador Martin-Pizarro and Javier Fres\'an for their comments on our first draft.

We are thankful to the referee for their careful reading and many helpful comments.

The second-named author does not take credit for the cohomological computations starting in Section~\ref{sec:exp_gen}.

\section{Notation and preliminaries}
\label{sec:not}
\subsection{Fields of definition}
\label{ssec:field}\label{ssec:field2}
If $z$ is a complex number,
we write $\Re(z)$ and $\Im(z)$ for its real and imaginary part.
Let $k\subset\C$ be a subfield.
We denote by $k_0$ the intersection $k \cap \R$,
by $\bar k$ the algebraic closure of~$k$ in~$\C$,
and by $\tilde k$ the real closure of~$k_0$ in~$\R$.
Note that $k$ is not automatically algebraic over~$k_0$.
(For example, let $a,b \in \R$ be such that $\trdeg_\Q(\Q(a,b)) = 2$,
and consider $k = \Q(a + bi)$. Then $k_0 = \Q$.)
The following conditions on $k$ are equivalent:
\[
 \text{$k_0 \subset k$ is alg.} \iff
 \text{$k_0 \subset \bar k$ is alg.} \iff
 \text{$\tilde k \subset \bar k$ is alg.} \iff
 \text{$[\bar k : \tilde k] = 2$}.
\]
If $k$ satisfies these conditions,
then so does every intermediate extension $k \subset L \subset \C$
with $k \subset L$ algebraic.

\subsection{Categories of varieties}
Let $k\subset\C$ be a subfield.
By variety we mean a quasi-projective reduced separated scheme of finite type over $k$.
We denote the associated analytic space on $X(\C)$ by~$X^\an$.

\subsection{Good compactifications}\label{ssec:gc}\label{ssec:gc2}
We say that $(X,Y)$ is a \emph{log-pair}
if $X$ is smooth of pure dimension $d$, and
$Y$ a simple normal crossings divisor.
A \emph{good compactification} of $(X,Y)$
is the choice of an open immersion $X \subset \bar{X}$ such that
$\bar{X}$ is smooth projective, $X$ is dense in $\bar{X}$ and
$\bar{Y}+X_\infty$ is a simple normal crossings divisor where $\bar{Y}$ is
the closure of $Y$ in $\bar{X}$ and $X_\infty = \bar{X} \ssm X$.

If, in addition, we have a morphism $f:X\to\A^1$, we say that
$\bar{X}$ is \emph{a good compactification relative to $f$} \emph{(or a good compactification of $(X,f)$)} if $\bar{X}$ is a good compactification and $f$ extends
to $\bar{f}:\bar{X}\to\Pe^1$. Note that $\bar{f}^{-1}(\infty)$ is a simple normal crossings divisor in this case.
We say that $\bar{X}$ is \emph{a good compactification of the log-pair 
$(X,Y)$ relative to $f$} if it is a good compactification of $X$ relative to $f$ and at the same time a good compactification of $(X,Y)$.

Such compactifications exist, by the following argument.
Let $f:X\to\A^1$ be a morphism. Consider the graph of $f$ in
$X\times\A^1$ and take its Zariski closure $\bar{X}''$ inside
$\bar{X}'\times\Pe^1$, where $\bar{X}'$ is a projective variety
containing $X$ as a Zariski open and dense subset. We may consider $X$
as a Zariski open and dense subset of $\bar{X}''$.
The projection $\bar{X}''\to \Pe^1$ is a morphism that extends
$f$.
By applying resolution of singularities to~$\bar{X}''$ we see that a good compactification
relative to $f$ exists. Suppose that we have in addition a simple normal crossings divisor $Y$ in $X$. Again by resolution of singularities, we can achieve that $\bar{Y}+X_\infty$ is a simple normal crossings divisor.

\subsection{Some semi-algebraic sets} \label{sec:not_B}\label{sec:not_B2}
Let $k$ be as in Section~\ref{ssec:field}.
Let $X$ be a smooth variety,
$\bar{X}$ a good compactification,
$X_\infty = \bar{X} \ohne X$.
We denote by $B_{\bar{X}}(X)$
the oriented real blow-up $\OBl_{X_\infty}(X^\an)$
of $\bar{X}^\an$ in $X_\infty^\an$;
for details see 
Remark~\ref{alt-notn-blow-up} and Definition~\ref{defn:orb}.
It is a $k_0$-semi-algebraic $C^\infty$-manifold with corners, see Proposition~\ref{prop:is_sa}.

In the case $X = \A^1$, $\bar{X} = \Pe^1$,
we write $\tilde{\Pe}^1 = B_{\Pe^1}(\A^1)$.
This is a manifold with boundary:
the compactification of $\C \isom \R^{2}$ by a circle at infinity,
one point for each half ray.
For $s \in \C \ssm \{0\}$,
we write $s\infty$ for the point of $\partial \tilde{\Pe}^1$
corresponding to the half ray $s[0,\infty)$.
We say $\Re(s\infty) > 0$ if $\Re(s)>0$.
We put
\begin{align*}
 \Bcirc &= \Ptilde \ohne
 \{ s\infty \in \partial \Ptilde \mid \Re(s\infty) \leq 0\}
 = \C\cup \{s\infty \mid \Re(s) > 0 \}, \\
\partial \Bcirc&=\Bcirc\ssm \C=\{s\infty \mid \Re(s) > 0 \} ,\\
 \Bsharp &= \Ptilde \ohne
 \{ s\infty \in \partial \Ptilde \mid s\infty \neq 1\infty \}
 = \C\cup \{1\infty\},\\
\partial \Bsharp &=\Bsharp\ssm \C=\{1\infty\}.
\end{align*}

If $G\subset\R^n$ is a $k_0$-semi-algebraic subset,
we will also denote by $\partial G$ the complement $G\ssm G^\interior$
where $G^\interior$ is the interior of $G$ inside~$X(\R)$
where $X$ is the Zariski closure of $G$ in $\A^n_{k_0}$.
If $G$ is of dimension $d$, then $\partial G$ is of dimension at most $d-1$.

This is compatible with notation above if we identify $\Ptilde$ with a closed disc in $\R^2$.
Note that we do not assume that $G$ is closed.
%

%
%
%
%

\subsection{$C^1$-homology}\label{sec:homology}

A subset $\Delta\subset\R^n$ is called  a \emph{linear $m$-simplex} if it is the interior of a convex hull of  $m+1$ affinely independent points. Our most important example 
is the \emph{standard  simplex} $\Delta_n$ (in the normalisation of \cite{warner}), the convex hull of
$0$ and the standard basis vectors of $\R^n$. This means:
\[
\Delta_n= \left\{ (x_1, \dots, x_n) \mid x_i > 0 \text{ and } \sum_i x_i < 1 \right\}
 \subset \R^n.
\]
It is open in the ambient space.
We denote by $\bar \Delta_n$ its closure in~$\R^n$.
We fix the standard orientation.
We define the face maps $\bar{\Delta}_{n-1}\to \bar{\Delta}_n$ as
in \cite[(2)~p.142]{warner}. Moreover, for any topological space $X$ and subspace $Y$,
we let $H_n(X,Y ;R)$ denote $n$-th singular homology with coefficients in
the ring $R$.

A \emph{manifold with corners} is a second countable Hausdorff
topological space such that every point has a neighborhood that is
homeomorphic to an open subset of $\R^{m_1}\times\R^{m_2}_{\geq 0}$.
We will assume that each manifold with corners is equipped
with a set of charts which need not be maximal; later on this set will be finite.
Say $p\ge 1$.
A map defined on a subset $A$ of $\R^n$ with values in
$\R^m$ is called $C^p$ if it extends to a $C^p$ map on an open
neighborhood of $A$ in $\R^n$ with values in $\R^m$. A
\emph{$C^p$-manifold with corners} is a manifold with corners such
that all transition maps betweens charts are $C^p$.
A map between two  $C^p$-manifolds with corners is called $C^p$ if it
is $C^p$ on all charts. 
The \emph{boundary} of a manifold with corners $X$ is the 
subset $\partial X=X\ssm X^{\mathrm{sm}}$ where $X^{\mathrm{sm}}$ is the open subset of regular points of $X$, i.e., their image under the chart maps is in $\R^{m_1}\times \R^{m_2}_{>0}$.

\begin{rem}
  If $X$ is both a manifold with corners and a semi-algebraic set, then we have defined two conflicting notions of boundary $\partial X$ in the same spirit. 
  For the semi-algebraic manifolds with corners constructed in
  Chapter~\ref{realblow} and beyond, the two notions of boundary agree.
\end{rem}

\begin{defn}\label{defn:int_simplex}
Let $X$ be a manifold with corners.
A \emph{$C^1$-simplex} of dimension $n$ on $X$ is a continuous map
\[ \sigma:\bar{\Delta}_n\to X\]
such that for any chart $\phi: U\rightarrow V\subset \R^{m_1}\times\R^{m_2}_{\ge 0}$ with $U$ open in $X$
the composition $\phi\circ \sigma|_{\sigma^{-1}(U)}
:\sigma^{-1}(U)\rightarrow V$
is $C^1$ in the above sense.

Let $S_n(X)$ be the space of formal $\Q$-linear combinations
of $C^1$-simplices of dimension $n$.
For $A\subset X$ closed, we denote by $S_n(A)\subset S_n(X)$ the subspace spanned by simplices with image in $A$.
\end{defn}

The restriction of $\sigma$ to a face is again $C^1$, hence the usual boundary operator $\partial$ turns
$S_*(X)$ into a complex. The barycentric subdivision of a $C^1$-simplex is again $C^1$.

\begin{rem}\label{rem:easy_convergence}If $\omega$ is an $n$-form of class $C^1$, then
 $\sigma^*\omega=g\md t_1\wedge\dots\wedge \md t_n$ for a $C^0$-function $g$ on $\bar{\Delta}_n$.
	Hence $\int_\sigma \omega$ is equal to the Lebesgue integral of $g$ and converges absolutely.
\end{rem}

\begin{thm}\label{thm:compare}Let $X$ be a $C^1$-manifold with corners. Then the complex
$S_*(X)$ of $C^1$-chains computes singular homology of $X$ and the complex
$S_*(X)/S_*(\partial X)$ computes singular homology of $X$ relative to its boundary $\partial X$.
\end{thm}
\begin{proof}
It suffices to show that $S_*(X)$ computes singular homology of $X$ and
$S_*(\partial X)$ computes singular homology of $\partial X$.
It is equivalent to prove the result in cohomology instead.  
The argument for the $C^\infty$-case and smooth manifolds without boundary is given in \cite[Section~5.31]{warner}.
It works without changes for $S_*(X)$ (the $C^1$-case).

We go through the argument to verify that it still applies to manifolds with corners and to their boundaries.
Recall that the boundary $\partial X$ is not a $C^1$-manifold itself, but only a closed subset in
a $C^1$-manifold with corners. We write $Y$ for either $X$ or $\partial X$. 
For every open subset $U\subset Y$ let $S^*(U)$ be the dual complex of
$S_*(U)$. This defines a presheaf of complexes.
We consider its sheafification $\Sh^*$  on $Y$. Warner shows that it is
a fine resolution of the constant sheaf $\Q$. Assuming this, we obtain isomorphisms
\[ H^n(Y,\Q)\isom H^n(\Sh^*(Y)).\]
In \cite[Section~5.32]{warner}, he shows that, moreover,
\[ H^n(\Sh^*(Y))\isom H^n(S^*(Y)).\]
His argument applies verbatim, once we note that the barycentric subdivison
of a $C^1$-simplex in $Y$ is again a $C^1$-simplex in $Y$. Together
\[ H^n(Y,\Q)\isom H^i(S^*(Y))\]
as claimed.

It remains to study $\Sh^*$. Warner shows (see \cite[p.~193-194]{warner}) that the complex
is fine. His argument on p.~193 depends on the construction of a partition of unity.
In the case $Y=\partial X$, we may use a partion of unity on $X$. Otherwise the argument is unchanged.

In order to show that the complex is a resolution of the constant sheaf $\Q$, it suffices to show
that $S^*(U)$ is contractible if $U$ is the intersection of a unit ball
with $Y$ (or more precisely its image under a chart). Warner achieves this by constructing a simplicial homotopy, see \cite[p.~194--196]{warner}.
Depending on the situation, $U$ is an open ball (the manifold case) or the intersection
of an open ball with $\R_{\geq 0}^{n_1}\times\R^{n_2}$ (the case of a manifold with corners) or the boundary of the latter (the case of $\partial X$). In all cases it contains the line segment $[0,x]$ for all $x\in U$.
If $\sigma$ is a simplex
with values in $U$, then so is $\tilde{h}_p(\sigma)$ of Equation~(21) of loc.\ cit. 
\end{proof}
\begin{thm}[{\cite[Chapter~III, \S\S 16-17]{whitney}}]
 \label{thm:stokes}Let $X$ be a $C^2$-manifold with corners.
 Let $\omega$ be an $n$-form of class $C^1$ on $X$ and let
 $\sigma :\bar{\Delta}_{n+1}\to X$ be a
 $C^1$-simplex. Then
 \[ \int_{\sigma}d\omega=\int_{\partial \sigma }\omega.\] 
\end{thm}
\begin{proof}
We put together the relevant statements of \cite{whitney}.  In working with \cite{whitney} note that Whitney's notion of smooth means $C^1$ in modern terms,
see loc.\ cit.\ p.~15.

We first recall the notion of a \emph{regular} differential form in
  a Euclidean space
  introduced by Cartan~\cite[Section III.16]{whitney}. (Warning: this
  notion is unrelated to the usual concept of regularity in algebraic
  geometry.)
  A continuous $r$-form $\omega$ on an open subset $R\subset\R^n$ is regular if 
  there exists a continuous and necessarily unique
  $(r+1)$-form $\omega'$ (then called $d\omega$) such that for every
  oriented linear
	$(r+1)$-simplex $\bar{\Delta} \subset R$ (i.e., the convex hull of $(r+2)$ affinely independent points) we
  have 
\[ \int_{\bar{\Delta}} \omega'=\int_{\partial\bar{\Delta}}\omega.\]
All $C^1$-forms on $R \subset \R^n$ are regular by the criterion of \cite[Lemma~III.16d]{whitney} (based on Stokes's Theorem for $\Delta$ as in \cite[Theorem~III.14A]{whitney}). 

	By \cite[Section~III.17]{whitney} a \emph{regular} $r$-form on a manifold (without boundary) is defined as a continuous $r$-form such that the restriction to all coordinate charts is regular with $d\omega$ independent of the chart. In particular, all $C^1$-forms on manifolds are also regular.
		By \cite[Theorem~III.17B]{whitney} the notion is stable under pull-back via $C^1$-maps. In particular, $\sigma^*\omega$ on $\bar{\Delta}_{n+1}$ is regular. The formula of the theorem holds.

We argue in more detail in order to handle the case of a manifold with corners.
Let $\omega$ and $\sigma$ be as in the hypothesis. 
After passing to
  a barycentric subdivision we may assume that $\sigma$ takes values
  in the domain of a chart $U \rightarrow
  V\subset\R^{m_1}\times\R^{m_2}_{\ge 0}$, with $U\subset X$ open. The form $\omega$ extends to a $C^1$-form on an open neighbourhood $\tilde{V}$ of $V$ in $\R^{m_1+m_2}$.
By definition, $\sigma$ extends to a $C^1$-map 
\[ \tilde{\sigma}:\Omega\to \tilde{V}\]
on an open
  neighbourhood $\Omega\subset\R^{n+1}$ of $\bar{\Delta}_{n+1}$.
 
  Thus $\omega$ is regular on $\tilde{V}$ and $\omega'$ is the usual $d\omega$ by
  Stokes's theorem for linear simplices, see \cite[Theorem
  III.14A]{whitney}. 
   By \cite[Theorem~III.16B]{whitney}, the
  pull-back of a regular form under a $C^1$\nobreakdash-map is again regular and
  pull-back commutes with $d$. So 
  $\tilde{\sigma}^*\omega$ is regular on $\Omega$ with $d\tilde{\sigma}^*\omega=\tilde{\sigma}^*d\omega$. By
  definition of regularity, we have
  \[ 
\int_{\bar{\Delta}_{n+1}} \tilde{\sigma}^*d\omega
  = \int_{\partial\bar{\Delta}_{n+1}} \tilde{\sigma}^*\omega
\]
  when considering $\bar{\Delta}_{n+1}$ as a linear simplex in
  $\R^{n+1}$  with the usual orientation. 
  So the formula of the theorem holds. 
\end{proof}

\begin{rem}The $C^2$-condition is needed in order to make the notion of a $C^1$-form and hence $d \omega$ well-defined. It
could be relaxed to a manifold with corners of class $C^1$ and  $\omega$ regular in the sense of Whitney (see the proof above). The $C^2$-case is enough for our purposes.
\end{rem}

\begin{rem} An alternative description of singular homology of semi-algebraic
spaces (or more generally, spaces definable in some o-minimal structure, see
Chapter~\ref{sec:definable-manifold}) using semi-algebraic (definable) maps without any regularity assumptions is given in \cite{huber-stokes}. Theorem~\ref{thm:stokes} is used to deduce Stokes's Theorem in this setting. 
\end{rem}

\section{O-minimal structures}

\label{ssec:omin}

For the purposes of our paper it is helpful to think of o\nobreakdash-minimal geometry as a generalisation of semi-algebraic geometry.
The canonical reference
for o\nobreakdash-minimality  is~\cite{D:oMin}.
Within the encyclopedia of mathematics,
o\nobreakdash-minimality is firmly rooted in the field of mathematical logic
and more particularly model theory.
In this section we briefly survey the essentials
in a fashion that is geared
towards geometers with no background in model theory.
The reader is warned in advance
that some of the definitions presented below
are severe mutations of more general concepts in model theory.
\begin{defn}[{\cite[Chapter 1, (2.1)]{D:oMin}}]
 \label{structure}
 A \emph{structure} on a non-empty set~$R$ is a sequence
 $\mathcal{S} = (\mathcal{S}_m)_{m \in \Z_{\ge0}}$
 such that for each $m \ge 0$
 \begin{enumerate}
  \item $\mathcal{S}_m$ is a boolean subalgebra of
   the power set~$\mathcal{P}(R^m)$:
   that is, $\varnothing \in \mathcal{S}_m$,
   and $\mathcal{S}_m$ is closed under complements
   and binary unions and intersections;
  \item if $A \in \mathcal{S}_m$,
   then $R \times A$ and $A \times R$ belong to $\mathcal{S}_{m+1}$;
  \item $\{(x_1, \ldots, x_m) \in R^m \mid x_1 = x_m\} \in \mathcal{S}_m$;
  \item if $A \in \mathcal{S}_{m+1}$, then $\pi(A) \in \mathcal{S}_m$,
   where $\pi \colon R^{m+1} \to R^m$
   is the projection onto the first~$m$ coordinates.
 \end{enumerate}
\end{defn}

We are actually only going to need the case $R=\R$, but in this section we will present the definitions in the general setting.

\begin{defn}
 \label{semi-algebraic}
 Let $k\subset\R$ be a subfield.
 A structure that is relevant to the topic of this paper
 is the structure of \emph{$k$-semi-algebraic sets} over~$\R$
 consisting  of those subsets of~$\R^m$ that are finite
   unions of sets of the form
 \[
  \big\{x \in \R^m \mid
  f_1(x) = \ldots = f_k(x) = 0 \text{ and }
  g_1(x) > 0, \ldots, g_l(x) > 0\big\}
 \]
 for some polynomials $f_i,g_j \in k[X_1, \ldots, X_m]$.
\end{defn}

It is a non-trivial fact that the collection of semi-algebraic sets
satisfies the final condition in \cref{structure}.
This result is known as the Tarski--Seidenberg theorem.
The structure does not change when we replace $k$ by an algebraic subextension in $\R$, hence we may assume $k$ to be real closed.

A structure can often be ``generated'' by a smaller collection of sets.
This leads to the following concept
(one that is more faithful to the model-theoretic point of view).
We follow the terminology of~\cite{D:oMin}.

\begin{defn}[{\cite[Chapter 1, (5.2)]{D:oMin}}]
  \label{def:mts}
 A \emph{model theoretic structure}
 $\mathcal{R} = (R, (S_i)_{i \in I}, (f_j)_{j \in J})$
 consists of a non-empty set~$R$, called its \emph{underlying set},
 relations $S_i \subset R^{m(i)}$ (for $i \in I$ and  $m(i) \in \Na_0$),
 and functions $f_j \colon R^{n(j)} \to R$ (for $j \in J$ and $n(j) \in \Na_0$).
 If $n(j) = 0$, we call $f_j$ a \emph{constant}
 and identify it with its unique value.
\end{defn}
If $\mathcal{R} = (R, (S_i)_{i \in I}, (f_j)_{j \in J})$
is a model theoretic structure,
and $C \subset R$ a subset,
then we denote the model theoretic structure
$(R, (S_i)_{i \in I}, (f_j)_{j \in J} \cup (c)_{c \in C})$
by $\mathcal{R}_C$.
The elements of~$C$ are called \emph{parameters}.

\begin{defn}[{\cite[Chapter 1, (5.3)]{D:oMin}}]
 \label{definable}\leavevmode
 \begin{enumerate}
  \item Let $\mathcal{R} = (R, (S_i)_{i \in I}, (f_j)_{j \in J})$
   be a model theoretic structure.
   We denote by $\Def(\mathcal{R})$ the smallest structure on~$R$
   that contains the $S_i$, for $i \in I$,
   and the graphs of the functions $f_j$ (for $j \in J$).
  \item A subset $A \subset R^m$ is called \emph{definable} in $\mathcal{R}$
   if $A \in \Def(\mathcal{R})_m$.
   A function $f \colon R^m \to R^n$ is definable in $\mathcal{R}$
   if its graph
   \[ \Gamma(f) = \{(x,y) \mid y = f(x)\} \subset R^m \times R^n = R^{m+n}\]
   is definable in $\mathcal{R}$.
   A point $x \in R^m$ is definable in $\mathcal{R}$
   if the singleton $\{x\} \subset R^m$ is definable in $\mathcal{R}$.
  \item Let $C \subset R$ be a subset.
   A subset/function/point is
   \emph{definable in $\mathcal{R}$ with parameters from~$C$}
   or \emph{definable over~$C$ in $\mathcal{R}$}
   or \emph{$C$-definable in $\mathcal{R}$}
   if it is definable in~$\mathcal{R}_C$.   
 \end{enumerate}
\end{defn}

The following proposition serves two purposes:
it makes the relation of the previous definitions with logic apparent,
and it is a useful result for showing that certain sets are definable.

\begin{rem}[{\cite[Chapter 1, (5.9) Exercise~1]{D:oMin}}]
 \label{definable_iff}
 If the tuple $\mathcal{R} = (R, (S_i)_{i \in I}, (f_j)_{j \in J})$
 is a model theoretic structure, and $C \subset R$ a subset,
 then a subset $A \subset R^m$ is definable in $\mathcal{R}$
 with parameters from~$C$
 if and only if
 there exists a definable $B\subset\R^{m+n}$  
 and elements $c_1, \ldots, c_n \in C$
 such that
 \[
  A = \big\{(a_1, \ldots, a_m) \in R^m \mid
   (a_1, \ldots, a_m, c_1, \ldots, c_n)\in B\big\}.
 \]
\end{rem}

\newcommand{\rcf}[2]{({#1},{<}, 0, 1, {+}, {\cdot}{#2})}

\begin{ex}
\begin{enumerate}
\item
From now on, we will denote by  $\R_\alg$ the model theoretic structure $\rcf{\R}{}$
and (consistent with \cref{definable}) for every subfield $k\subset\R$
 we denote by $\R_{\alg,k}$ the model theoretic structure obtained from~$\R_\alg$
 by adding elements in $k$  as constants.
This is justified by the fact that the structure $\Def(\R_{\alg,k})$
consists precisely of the
$k$-semi-algebraic sets introduced in
\Cref{semi-algebraic}.
 Indeed, they are defined
 by first-order formulas in the language of~$\R_\alg$
 with parameters from~$k$.
\item
 Let $A \subset \R^m$ be a $k$-semi-algebraic set.
 Using \Cref{definable_iff} it becomes straightforward
 to show that the topological closure $\bar A \subset \R^m$ is semi-algebraic.
 Indeed
 \[
  \bar A =
  \big\{x \in \R^m \mid
  \forall \varepsilon \in \R, \exists y \in A,
  \varepsilon > 0 \to |x - y| < \varepsilon \big\},
 \]
which is clearly a first-order formula.
\end{enumerate}
\end{ex}

\begin{remark}
  For our purposes it is essential to keep track of parameters. For
  example, $\pi$ is $\R$-definable in $\R_\alg$ but not
  $\Q$-definable in $\R_\alg$. When dealing with definable sets we
  usually explicitly mention the scope of our parameters. 
\end{remark}

\begin{defn}
We say that a model theoretic structure
	$\mathcal{R}$ \emph{expands} the structure $\rcf{\R}{}$
	if its underlying set is $\R$, and if it contains the relation  ${<}$, the constants
   $0,1$, and the functions ${+},{\cdot}$ with their usual interpretations. 
\end{defn}

Now we are finally ready for the central notion.

\begin{defn}[{\cite[Chapter 1, (3.2) and (5.7)]{D:oMin}}]  
 \label{o-minimal}
 A model theoretic structure $\mathcal{R}$ expanding
  $\rcf{\R}{}$
   is \emph{o-minimal} %
   if the $\R$-definable subsets of $\R$ 
   are exactly the finite unions of points and
   (possibly unbounded) open intervals in~$\R$.
 \end{defn}
 
	\begin{rem}Note that in this definition, Van den Dries considers  $\R$\nobreakdash-de\-fi\-nable
  subsets of $\R$ in $\mathcal{R}$.
  In particular, it is not required
  that every interval is definable in $\mathcal{R}$ without
  introducing additional parameters.

\end{rem}

\begin{ex}
 Since the $\R$-semi-algebraic subsets of the real line
 are exactly finite unions of points and (possibly unbounded) open intervals,
 we see that $\R_\alg$ is an o-minimal structure.

 Note that $\mathbb{N}$ and $\Z$ are not definable subsets
 in any o-minimal structure,
 because of the finiteness condition in the definition.
 In particular,
 the functions $\sin \colon \R \to \R$ and $\exp \colon \C \to \C$
 (after identifying $\C$ with $\R^2$)
 cannot be definable in any o-minimal structure.
\end{ex}

\begin{defn}
 \label{Rexp}
 The model theoretic structure
 $\rcf{\R}{,\exp}$ %
 will be denoted by $\R_{\exp}$.
 Here $\exp \colon \R \to \R$ is the \emph{real} exponential function
 (and not the complex one, this is important!).
\end{defn}

\begin{defn}\label{defn:Rexpsin}
 Let $\mathcal{F}_{\an}$ be the collection of
 \emph{restricted analytic} functions,
 that is, functions $f \colon \R^m \to \R$ that are zero outside $[0,1]^m$
 and such that $f|_{[0,1]^m}$ can be extended to a real analytic function
 on an open neighbourhood of $[0,1]^m$.

 We denote by $\R_{\an}$ the model theoretic structure
 $\rcf{\R}{,\mathcal{F}_{\an}}$
 and by $\R_{\an,\exp}$
 the model theoretic structure
 $\rcf{\R}{,\mathcal{F}_{\an},\exp}$.
Finally, we denote by $\R_{\sin,\exp}$ the model theoretic structure
 $\rcf{\R}{, {\sin}|_{[0,1]}, \exp}$.
 For every subfield $k\subset\R$, we denote by $\R_{\sin,\exp,k}$ the
 model theoretic structure where we adjoin all elements in $k$ as
 parameters. %
\end{defn}
This is one of the protagonists in this paper.

\begin{rem}
 \label{Rsinexp_cos_pi}
 The model theoretic structure $\R_{\sin,\exp}$ will be of most interest to us.
 Note that 
 if the {bounded} interval $I \subset \R$ is definable with parameters in~$C$,
 then the functions ${\sin}|_I$ and ${\cos}|_I$
 are definable in $\R_{\sin,\exp}$ with parameters in~$C$.
 Indeed, one may use the identity $\sin^2(\theta) + \cos^2(\theta) = 1$
 to define $\cos(\theta)$ for $\theta \in [0,1]$.
 After that, $\cos(\theta)$ can be arbitrarily extended
 using $\cos(-\theta) = \cos(\theta)$
 and $\cos(2\theta) = 2\cos^2(\theta) - 1$.
 This allows one to define $\pi$:
 it is twice the smallest positive zero of $\cos$. Finally, one can
 define $\sin$ on
 arbitrary bounded definable intervals
 by translating $\cos$ by $\pi/2$.
\end{rem}

\begin{thm}
 \label{o-minimal_eg}
 The model theoretic structures $\R_{\exp}$, $\R_{\an}$, $\R_{\an,\exp}$, and
 $\R_{\sin,\exp}$
 are o-minimal.
 \begin{proof}
  For $\R_{\an}$, the result
was proven
  by Van den Dries
  in~\cite{Dries_1986_generalization_Tarski_Seidenberg} using
     a theorem of Gabrielov.
  Wilkie proved that $\R_{\exp}$ is o-minimal
  in~\cite{Wilkie_1996_Model_completeness}.
  Building on Wilkie's result (that was already announced in 1991),
  Van den Dries and Miller~\cite{Dries_Miller_1994_real_exponential_field}
  showed that $\R_{\an,\exp}$ is o-minimal.
Finally, $\R_{\sin,\exp}$ is o-minimal because its definable
    sets are definable in the o-minimal structure $\R_{\an,\exp}$
 and it expands~$\R_{\alg}$.
 \end{proof}
\end{thm}

\begin{remark}
  \label{o-min-dimension}
   A fundamental fact about o-minimal structures is that each 
    definable set is a finite disjoint union of basic building blocks
    called \textit{cells}. If the set is defined over a subfield $k\subset\R$, then so are the cells.
	This follows from 
        the Cell Decomposition Theorem~\cite[Chapter 3, (2.11)]{D:oMin}. 
Indeed, 
        suppose the set is definable in a model
          theoretic structure that expands
          $(\R,<,0,1,+,\cdot)$.
        Then the cells arising from the Cell Decomposition Theorem are
      defined in the same model theoretic structure. In particular, no
    additional constants are required to define the cells. For
    example, if a set $A$ is definable in $\R_{\alg,k}$ then the cells
  in a cell decomposition can be chosen to be definable in
  $\R_{\alg,k}$. In other words, these cells arise from polynomial
	equalities and inequalities with coefficients in $k$, see also 
\cite[Chapter~3, (2.19) Exercise 4]{D:oMin}.
    Using this theorem one can introduce 
    a good notion of
    dimension of definable sets
    that behaves as one expects intuitively.    
 For example, if $X$ is a non-empty definable set,
 then $\dim(\bar X \ssm X) < \dim(X)$.
 See \cite[Chapter 4]{D:oMin} %
 for details and other properties
 of the dimension. The cells in a decomposition of $A$ are in
   general not uniquely determined by $A$.
\end{remark}

\begin{remark}
  For the reader well versed in o-minimality we remark that for
  the remainder of this text, our o-minimal structures will always
  expand $\rcf{\R}{}$. In particular, a definable subset of $\R^n$ is
  \emph{connected}  if and only if it is \emph{definably connected}.
  Moreover, the word \emph{compact}
  retains its meaning from point set topology.
\end{remark}

\section{Definable manifolds}
\label{sec:definable-manifold}
Fix an arbitrary o-minimal structure $\mathcal{S}$ expanding
$\rcf{\R}{}$.
and a subfield $k \subset \R$.
In the remainder of this section,
all definable sets are understood to be definable
in~$\mathcal{S}$
with parameters from~$k$
unless otherwise specified.

\begin{defn}\label{defn:manifold}
Let $0\leq p\leq \infty$.
\begin{enumerate}
\item
  A \emph{definable $C^p$-manifold with corners} $M$ is a
  $C^p$-manifold with corners together with the choice of a finite atlas
  $(\phi_i:U_i\to V_i\subset\R^{n_i}\times\R_{\ge 0}^{m_i})_{i\in I}$ such that
  the $V_i$ are open in $\R^{n_i}\times\R^{m_i}_{\ge 0}$ and definable and the transition
  maps $\phi_{ij}=\phi_j\circ\phi_i^{-1}$ are definable and of class
  $C^p$ on their domains.
 Its boundary $\partial M$ is the union of the preimages of the boundaries
of $\R^{n_i}\times \R^{m_i}_{\geq 0}\subset\R^{n_i+m_i}$ under the~$\phi_i$.
\item
  A subset $G\subset M$ is called \emph{definable} if $\phi_i(G\cap U_i)$ is definable in $\R^{n_i + m_i}$ for all $i$.
\item
  A subset $N$ of a definable $C^p$-manifold with corners $M$ is called a
  \emph{submanifold} if there is a $C^p$-manifold $N'$ and a $C^p$-immersion
  $N'\rightarrow M$ that is a homeomorphism onto $N$
		(where $N$ carries the subspace topology).
		That is, our submanifolds are embedded and have no
                corners or boundary.
\item \label{it:map}Let $(M,\phi_i)$ and $(N,\psi_j)$ be definable $C^p$-manifolds with corners.
  A map $f \colon M \to N$ of definable $C^p$-manifolds with corners is called
  a \emph{definable $C^p$} map
  if all $\psi_j\circ f\circ \phi_i^{-1}$
  are definable and $C^p$ on their domains.
\item
  A definable subset $G$ of a $C^p$-manifold with corners $M$ is
  called \emph{affine} if there is a definable
  open neighbourhood of $G$ in $M$ that
  is definable $C^p$-isomorphic to an open subset of $\R^n\times \R^m_{\ge 0}$. 
\end{enumerate}
\end{defn}

\begin{rem}
 The definition of a definable manifold includes the choice of a finite atlas.
 The finiteness condition is important,
 as, for example, we do not want manifolds with infinitely
 many connected components.
 So we cannot work with a maximal atlas.
 However, we could work with an equivalence class of finite atlases.
 Alternatively, one may rephrase the definition in the language of
 locally ringed sites, using the Grothendieck topology of definable
 open subsets and finite covers.
 The definition of a definable manifold is inspired by and related to
 the semi-algebraic spaces of Delfs and Knebusch~\cite{DK81}
 and the complex analytic definable spaces of
 Bakker--Brunebarbe--Tsimerman~\cite{omingaga}.
 See \cite[Chapter  10,~\S 1]{D:oMin} for an introduction to general
 definable spaces.
\end{rem}

\begin{rem}
 \label{affine-semialg-spaces}
 Robson (see \cite{Robs83}) showed that all semi-algebraic spaces
 (the $C^0$-case of the above definition)
 are actually affine, i.e., can be embedded into $\R^n$.
 This was extended to $C^p$-manifolds without boundary by Kawakami, see \cite{kawakami}.
 It is not clear to us if the result extends to manifolds with corners.
 The above notion is general enough for our needs.
\end{rem}

\begin{ex}
 Let $\bar{\Delta}\subset\R^n$ be the closed linear simplex spanned by $v_0,\dots,v_m\in k^n$.
 Then $\bar{\Delta}$ is a definable $C^p$-manifold with corners for all $p\geq 0$.
 As this example shows,
 the boundary of a manifold with corners
 does not necessarily have a natural structure of $C^p$-submanifold for $p\neq 0$. Indeed, the boundary of the $2$-simplex (i.e., a triangle) is homeomorphic to a $1$-sphere, but it is not a $C^1$-submanifold of $\R^2$.
 We are particularly interested in the case $p=1$
 because every affine definable set $G$ has a triangulation such that
 the maps $\bar{\Delta}\to G$ are maps of definable $C^1$-manifolds in the above sense.
 See \cite{ohmoto-shiota-triangulation} and~\cite{omin-triang},
 and also \cref{prop:triangle_manifold}, where we quote this result.
\end{ex}

Another well-known example of definable manifolds without boundary are cells.
We refer to Chapter~3
of~\cite{D:oMin} for the definition and basic properties of $C^0$-cells.
Moreover, \cite[Chapter~7,~\S 3]{D:oMin} introduces $C^p$-cells and proves the
decomposition theorem for $p=1$; the general case is similar.

\begin{ex}\label{cell_Cp}
 Let $C\subset\R^n$ be a definable $C^p$-cell of dimension~$d$.
  By \cite[(2.7) in Chapter 3]{D:oMin} there is a choice 
  of $d$
  coordinates $\{x_{i_1},\dots,x_{i_d}\}$ on $\R^n$
   inducing a definable homeomorphism
   $\phi=(x_{i_1},\dots,x_{i_d}):C\to \phi(C)$
   with $\phi(C)$ an open cell in $\R^d$.
  We give $C$ the structure of an affine definable $C^p$-manifold
   using the chart $\phi$.
   Then the inclusion $C\to \R^n$
   is a definable $C^p$-map of definable $C^p$-manifolds.
   In other words, cells are definable $C^p$\nobreakdash-submanifolds of $\R^n$.
\end{ex}

\begin{defn}
  \label{def:Reg}
  Fix an integer $p\ge 1$, let $d\geq 0$ be an integer,
  and let $M$ be a definable $C^p$-manifold with corners
  and $G\subset M$ a definable subset.
  We define $\reg_d(G)$ to be the
  set of $x\in G$ that admit an open neighbourhood $U$ in $M$
  such that $G\cap U$ is a $C^p$-submanifold of $M$ of dimension $d$.
\end{defn}

\begin{rem}\label{rem:reg}
  The set $\reg_d(G)$ is open in $G$, it is empty if $\dim G<d$.
 If $\dim(G)=d$, it is the maximal subset of $G$ that is
 a submanifold of $M$ having connected components of dimension $d$.
 If $G$ and~$H$ are disjoint definable subsets of $M$,
 then in general there is no inclusion between the two sets $\reg_d(G\cup H)$ and
 $\reg_d(G)\cup \reg_d(H)$.

	For example, let $M$ be the plane,
	let $G$ be the coordinate axis $\{(x,y) \mid y = 0\}$,
	and let $H$ be the other coordinate axis minus the origin $\{(x,y) \mid x = 0 \text{ and } y \ne 0\}$.
	Then $\reg_d(G) \cup \reg_d(H)$ contains the origin, but $\reg_d(G \cup H)$ does not.

	In the other direction, let $M$ be the line,
	let $G = \{x \mid x \le 0\}$, and let $H = \{x \mid x > 0\}$.
	Then $\reg_d(G) \cup \reg_d(H) = M \ssm \{0\}$, whereas $\reg_d(G \cup H) = M$ contains the origin.
\end{rem}

The following lemma adapts to our situation
the fact that the $p$-regular
points of given dimension of a definable set constitute a definable set.
\begin{lemma}
 \label{Reg-definable}
 Let $M,G,$ and $\reg_d(G)$ be as in \Cref{def:Reg}.
 Then $\reg_d(G)$ is a definable subset of $M$ and
   $\dim G \ssm \reg_d(G) < d$ if $\dim G = d$.
\end{lemma}
\begin{proof}
  Assuming the first claim we begin by proving the last claim by
  contradiction.
  Suppose $H = G\ssm
  \reg_d(G)$ has dimension $\dim G=d$.
Let $U\subset M$ be the domain of a chart $\phi:U\to
V\subset\R^n\times\R^m_{\geq 0}$ which has non-empty intersection with
$H$. We can even arrange that $\dim H\cap U=d$. So $\phi(H\cap U)$ is a definable set of dimension $d$. We replace $G$ and $H$ by $\phi(U\cap G)$ and $\phi(H\cap U)$, respectively.
  So we may assume $H\subset G\subset \R^n\times\R^m_{\ge 0}$.
  We fix a $C^p$-cell decomposition of
  $G$ partitioning $H$ and $G\ssm H$.
  One cell in $H$ must have top dimension $\dim H=\dim G$ and this
  cell has a point not contained in the closure of any other cell.
  This point lies in $\reg_d(G)$, which is a contradiction.

  To show that $\reg_d(G)$ is definable it suffices to work in a single chart.
  So without loss of generality
  $G$ is a definable subset of $\R^n\times\R^m_{\ge 0}$ of dimension~$d$.
  We use the classical theory of differential
  manifolds to characterize submanifolds locally as graphs of functions: $\reg_d(G)$ is the set of points of~$G$
	that have an open neighbourhood in~$\R^{n+m}$ in which $G$ is the graph
  of a $C^p$ map defined on an open subset of a projection of~$\R^{n+m}$ to
  $d$ different coordinates.
  The argument laid out in \cite[B.9]{vdDMiller:96} applies directly
  to our slightly more general situation,
  and implies the definability of $\reg_d(G)$.
 \end{proof}

\begin{lemma}
  \label{lem:Gdecomposition}
  Let $G\subset\R^n$ be a definable subset of dimension $d$.
  Let $\pi\colon\R^n\rightarrow\R^d$ denote
  the projection to the first~$d$ coordinates.
  There are pairwise disjoint definable open subsets
  $G_0,G_1,\ldots,G_N$ of $\reg_d(G)$
  with $\dim G\ssm (G_0\cup \cdots\cup  G_N)< d$
  such that all fibres of $\pi|_{G_0}$ have positive dimension
  and such that $\pi|_{G_i} \colon G_i \to \pi(G_i)$ is a chart for
  all $i\in \{1,\ldots,N\}$.
\end{lemma}

\begin{proof}
 Without loss of generality $G = \reg_d(G)$.
 Let $G'$ be the set of points of~$G$ that are isolated in their fibre of~$\pi|_G$.
 It is definable, see %
   \cite[Chapter 4, (1.6) Corollary]{D:oMin}.
  Each fibre of $\pi|_{G'}$ is discrete
  and thus finite with
  uniformly bounded
  cardinality, see \cite[Chapter~3, (3.7) Corollary]{D:oMin}. Let $N$
  be the largest cardinality of a fibre.

  By definable choice \cite[Chapter 6, (1.2) Proposition part (i)]{D:oMin},
  applied to the graph of $\pi|_{G'}$  there is a
  definable section
  $\psi_1 \colon \pi(G') \rightarrow G'$ of $\pi|_{G'}$, i.e. $\pi \circ
  \psi_1$ is the identity. The image $\psi_1(\pi(G'))$ is a definable subset of $G'$.
  The complement $G'_1=G'\ssm \psi_1(\pi(G'))$ is
  also definable. Now $\pi|_{G'_1}$ certainly still has finite fibres,
  but the maximal fibre count dropped to $N-1$. We repeat this step
  and find a section $\psi_2\colon \pi(G'_1) \rightarrow G'_1$ and again
  the fibre count of $\pi$ on $G'_2 = G'_1\ssm \psi_2(\pi(G'_1))$ drops
  by one.

  After  $N$ steps, all
  fibres are exhausted. We obtain definable maps $\psi_1,\ldots,\psi_N$
  defined on subsets of $\pi(G')$ whose images cover $G'$ and are
  pairwise disjoint.

  But the $\psi_i$
  may fail to be continuous. By the  Cell Decomposition Theorem,
  \cite[Chapter~3, (2.11) Theorem]{D:oMin}
  applied to the domain of each
  $\psi_i$, we get, after adjusting $N$ and renaming, finitely
  many continuous definable maps $\psi_i:C_i\rightarrow G'$ on cells $C_i\subset\R^d$
  with $\bigcup_i \psi_i(C_i) = G'$ and with $\pi\circ \psi_i$ the identity
  for all $1\le   i\le N$. Observe that the $\psi_i(C_i)$ remain
  pairwise distinct.

  Suppose $\dim C_i = d$; then $C_i$ is open in $\R^d$. As $G$ is a
  manifold, invariance of domain implies that $\psi_i(C_i)$ is open in
  $G$ and $\psi_i:C_i\rightarrow \psi_i(C_i)$ is a homeomorphism.
  Thus $\pi|_{\psi_i(C_i)} :\psi_i(C_i)\rightarrow C_i$ is a chart.
   We  can safely ignore cells $C_i$ with $\dim C_i<d$; the union
  $H = \bigcup_{i:\dim C_i < d} \psi_i(C_i)$ is definable of dimension
  at most $d-1$.
  Fix a cell
  decomposition of $G\ssm G'$ and let $G_0$ be the union of all
  $d$-dimensional cells; then $G_0$ is open, and possibly empty, in the submanifold
  $G$.   We add the remaining cells to $H$. We retain $\dim H
  < d$ and
  the lemma follows from $G = G_0 \cup \bigcup_{i : \dim C_i=d} \psi(C_i)
  \cup H$.
\end{proof}

\begin{lemma}
 Let $p\geq 1$ and let $(M,\phi_i)$ and $(N,\psi_j)$ be definable $C^p$-manifolds with corners.
 Then the bundles $TM$ and $T^*M$ and their exterior powers
 have a natural structure of a definable $C^{p-1}$-manifold with corners.
 Moreover, a definable $C^p$-map $f \colon M\to N$ induces
 definable $C^{p-1}$-maps $df \colon TM\to TN$ and $d^*f \colon T^*N\to T^*M$.
\end{lemma}
\begin{proof}
We explain the case of $TM$. Let $\phi_i:U_i\to V_i\subset\R^n\times\R^m_{\geq 0}$ for $i\in I$ be a definable atlas of $M$. We glue $TM$ from 
\[ \tilde{U}_i=\phi^{-1}_i (T\R^{n+m}).\] 
It is a $\R^{n+m}$-bundle over $M$.
The induced maps 
\[ \tilde{\phi}_i:\tilde{U}_i\to \tilde{V}_i=V_i\times_{\R^{n+m}}T\R^{n+m}\subset\R^n\times\R^{m}_{\geq 0}\times\R^{n+m}
\]
 define charts for $TM$. By definition the transition maps for the charts of $M$ are $C^p$, meaning they extend to open neighbourhoods in $\R^{n+m}$ and are $C^p$ as such. The transition maps for the charts of $TM$ are $C^{p-1}$ as derivatives of $C^p$-maps.

It remains  to verify definability of the transition maps for the charts of $TM$.
 This holds because the derivative
 of a definable differentiable function is definable. 
 Indeed, in the $1$-dimensional case
 the graph $\Gamma(f')$ of the derivative is given by the formula
 \[
  \left\{ (x,y) \:\middle\vert\:
  \forall \varepsilon > 0, \exists \delta > 0, \forall x', |x' - x| < \delta \to
  \left|\frac{f(x') - f(x)}{x' - x} - y \right| < \varepsilon \right\}.
 \]
See \cite[\S 2.1]{vdDMiller:96} for a treatment on partial derivatives of more general definable maps.
\end{proof}

Many properties of affine definable sets extend immediately to the non-affine case. This is in particular the case for the notion of dimension and the stratification by submanifolds. We want to use these facts in order to integrate differential forms.

\begin{defn}
 \label{defn:diff_form}
 Let $p\geq 1$.
 Let $(M,\phi_i)$ be a definable $C^p$-manifold with corners and $G\subset M$ a definable subset.
A \emph{differential form $\omega$ of degree $d$ on $G$} is a continuous section
\[ \omega \colon G\to \Lambda^dT^*M.\]
It is called \emph{definable} if it is definable as a map in
  the sense of \cref{defn:manifold}~(\ref{it:map}).
\end{defn}

In the affine case, we can give an explicit description:
Let $x_1,\dots,x_n$ be the standard coordinates on $\R^n$. For
$I=\{i_1,\dots,i_d\}\subset\{1,\dots,n\}$ a subset with $i_1<i_2<\dots<i_d$ we write as usual
\[ \md x_I=\md x_{i_1}\wedge\dots\wedge \md x_{i_d}.\]
A differential form on $G$ can be written uniquely as
\[ \omega=\sum_{I}a_I\md x_I\]
with $a_I:G\to\R$ continuous.
It is definable if and only if the $a_I$ are definable.

\begin{rem}
Note that we do not put differentiability conditions or require that $\omega$ extends to a neighbourhood of $G$.
\end{rem}

\begin{lemma}
 Let $p\geq 1$, and $f \colon M\to N$ be a definable $C^p$-map of definable manifolds with corners.
 Let $G\subset M$ and $H\subset N$ be definable subsets with $f(G)\subset H$.
 Then the pull-back of a differential form on $H$ defines a differential form on $G$.
 If $\omega$ is definable, then so is $f^*\omega|_G$.
\end{lemma}
\begin{proof}
By definition, $f^*\omega|_G \colon G\to \Lambda^d T^*M$ is the composition
\[ G\to H\to \Lambda^dT^*N\to \Lambda^dT^*M\]
of continuous   maps. Hence it is definable if $\omega$ is definable.
\end{proof}

As usual, we can only expect a well-defined integration theory for differential forms on oriented domains. 
\begin{defn}
  \label{orientation}
  Fix an integer $p\ge 1$, let $d\geq 0$ be an integer,
  and let  $M$ be a definable $C^p$-manifold with corners
  with $G\subset M$ a definable subset of dimension $d$.
\begin{enumerate}
\item\label{it:pseudo}
  A \emph{pseudo-orientation} on $G$ is the choice of
  an equivalence class of a definable open subset $U\subset\reg_d(G)$
  such that $\dim(G\ohne U)<d$ and an orientation on $U$.
  Two such pairs are equivalent if they agree on the intersection.
  We thereby obtain an equivalence relation.
\item  A \emph{pseudo-oriented} definable set  is a definable set together with the choice of a pseudo-orientation.
   \item Given a pseudo-orientation on~$G$
    with $U$ as in (\ref{it:pseudo})
    and a differential form $\omega$ of degree $d$ on $G$,
    we define
    \[ \int_G\omega:=\int_{U}\omega\]
    if the integral on the right converges absolutely.
  \end{enumerate}
\end{defn}

\begin{rem}
 The same definition also allows us to integrate a $d$-form $\omega$
 over a $G$ of dimension smaller than $d$:
 in this case $\reg_d(G)=\emptyset$ and the integral is set to $0$.
 Such integrals occur in our formulas and are to be read in this way.
\end{rem}

\begin{lemma}\label{lem:pseudo} Let $p\geq 1$.
 Let $G$ be a definable subset of a definable $C^p$-manifold with corners $M$.
 \begin{enumerate}
  \item The integral is well-defined,
   i.e., independent of the choice of representative for the pseudo-orientation.
  \item By restriction a pseudo-orientation on $G$ also induces
   a pseudo-orientation on every definable subset $G'\subset G$ with $\dim G'=\dim G$.
  \item A pseudo-orientation on $G$ induces a
   pseudo-orientation on every definable superset $G\subset G''$ such
   that $\dim(G''\ohne G)<d$, in particular on~$\bar{G}$.
  \item \label{it:modif} Let $\pi \colon G'\to G$ be a definable modification,
   i.e., there is an open definable subset $U\subset \reg_d(G)$ with $\dim(G\ohne U)<d$
   such that $\pi|_{U'} \colon U'=\pi^{-1}(U)\to U$ is an isomorphism
   of definable  $C^p$-manifolds and $\dim(G'\ohne U')<d$.
   Then a pseudo-orientation on $G$ induces a pseudo-orientation on~$G'$.
 \end{enumerate}
\end{lemma}
\begin{proof} We start with (1). Suppose we are given two representatives of the same pseudo-orientation on definable open subsets $U_1,U_2\subset\reg_d(G)$
such that $\dim(G\ssm U_i)<d$. Then the same is true on $U_1\cap U_2$.
 Hence it suffices to consider the 
 case $U_1\subset U_2$.
		By assumption the orientation on $U_2$ restricts to
 $U_1$. We have
 \[ \int_{U_2}\omega=\int_{U_1}\omega \]
 because $U_2\ohne U_1$ has measure~$0$.
 The left hand side converges absolutely if and only if the right hand side does.

For (2)-(4), we fix a pseudo-orientation on $G$, i.e., an orientation on 
some $U\subset\reg_d(G)$ such that $\dim(G\ohne U)<d$. 

Let $G'\subset G$ and $U'=U\cap\reg_d(G')$.
The orientation on $U$ restricts to an orientation on~$U'$.
We have $\dim(G'\ohne U')<d$,
hence this data defines the pseudo-orientation on~$G'$.

Let $G\subset G''$ and $U''=U\cap\reg_d(G'')$.
The orientation on $U$ restricts to an orientation on~$U''$.
As $\dim(G''\ohne G)<d$, we also have $\dim(G''\ohne U'')<d$,
hence again this data defines a pseudo-orientation on~$G''$. 

The case of a modification combines the two operations: 
the pseudo-orientation on $G$ defines one on the subset $U$, and hence the isomorphic $U'$. It extends to the superset $G'$.
\end{proof}

\begin{cor}\label{cor:int_formel}
 Let $G,H\subset M$ be definable subsets of dimension at most $d$ of a definable $C^p$-manifold with corners,
 equipped with a pseudo-orientation on $G\cup H$.
 Let $\omega$ be a definable differential form of degree  $d$ on $G\cup H$.
 Then with the restricted pseudo-orientations
\[ \int_{G\cup H}\omega=\int_G\omega+\int_H\omega-\int_{G\cap H}\omega\]
where the left hand side is absolutely convergent if and only if all terms on the right are.
\end{cor}
\begin{proof}We may assume $\dim G = \dim H=d$. We can decompose $G\cup H$ into the disjoint subsets
$G\cap H, G\ohne H, H\ohne G$. Hence it suffices to check the formula in the case where the two sets are disjoint. 

We start with an orientation on a definable open subset $U\subset \reg_d(G\cup
H)$ with $\dim (G\cup H)\ohne U < d$. 
The pseudo-orientations on $G$ and $H$ are represented by the
restricted orientations
on  $V=U\cap \reg_d(G)$ and $W=U\cap\reg_d(H)$, respectively.
Then $V\cup W$ represents our pseudo-orientation on $G\cup H$. 
By definition
and by the standard computation rules for integration on manifolds,
we find
\[ \int_{G\cup H}\omega = \int_{V\cup W}\omega =
 \int_V\omega+\int_W\omega = \int_G\omega+\int_H\omega. \qedhere\]
\end{proof}

\begin{rem}
\begin{enumerate}
\item
As in the case of ordinary orientations,
the value of the integral depends on the choice of
pseudo-orientation. Note that even a simple definable set like an
interval $G$
admits infinitely many different pseudo-orientations.
If $U\subset G$ is the complement of finitely many points such that it has $n$ connected components,
there are $2^n$ possible orientations  giving rise to $2^n$ different pseudo-orientations of $G$. Even for the simple differential form $\md x$, this gives up to $2^n$ different values of $\int_G\md x$. By varying $U$ this leads to infinitely many pseudo-orientations and values of the integral.		
\item
  For each $G$  the choice $U = \reg_d(G)$ is canonical if it is possible
  to choose an orientation on this set. 
However, the behaviour of $\reg_d(G)$ under
standard topological operations is complicated.
It is not true that the choice of an orientation on $\reg_d(G)$
also induces an orientation on $\reg_d(\bar{G})$ (take $G=\R\ohne\{0\}$).
Neither is it true that $\reg_d(G')\subset \reg_d(G)$ if $G'\subset G$
(take the $x$-axis in the union of the coordinate axes in~$\R^2$).
Our more flexible notion sidesteps these problems.
\item
Note also that every non-empty  definable set~$G$
admits a pseudo-orientation because open cells are orientable and $G$ admits a cell decomposition.
\item 
  The restriction operation described in the proof of
  Lemma~\ref{lem:pseudo}(2) is well-defined in
  the following sense. Two representatives of a pseudo-orientation on $G$
  restrict to representatives of the same pseudo-orientation on $G'$.
   Moreover, the same holds true for the extension
  operation described in the proof of part (3).
  Finally,  extending a pseudo-orientation from $G$ to $G''$ and
  then restricting it back to $G$  recovers the original pseudo-orientation.
  So the extension in part (3) of the lemma is unique. 
\end{enumerate}
\end{rem}

Let $\mathrm{vol}(\cdot)$ denote the Lebesgue measure on $\R^n$.
\begin{remark}\label{rem:jordan}
  Let us recall some basic measure-theoretic properties of a
  definable subset $X\subset\R^n$.
  By the Cell Decomposition Theorem, $X$ is a finite union of cells.
  As cells are locally closed, $X$ is a Borel set and
  in particular Lebesgue measurable. 
  The topological boundary
  $\mathrm{bd}(X)$ is definable of dimension $\le n-1$.
  The Hausdorff dimension of $\mathrm{bd}(X)$ equals $\dim
  \mathrm{bd}(X)<n$,
  see the last paragraph on page 177 of~\cite{vdD:limitsets}.
	In particular, the $n$-dimensional Hausdorff measure of
  $\mathrm{bd}(X)$ vanishes. It is well-known that the $n$-dimensional
  Hausdorff measure coincides with the Lebesgue measure, so 
  $\mathrm{vol}(\mathrm{bd}(X))=0$. It follows from \cite[Theorem 8.50]{LaczkovichSos} that if $X$ is bounded, then
  $\mathrm{bd}(X)$ 
  is  Jordan
  measurable with Jordan measure zero.
  (See
  \cite[Chapter~3]{LaczkovichSos} for the definition and properties of
  the Jordan measure.) In particular, any bounded definable subset of
  $\R^n$ is
  Jordan measurable, with Jordan measure equal to $\mathrm{vol}(X)$,
  and has the same Jordan measure as its closure in $\R^n$
   \cite[Theorems~3.7, 3.9, and Section 8.4]{LaczkovichSos}.
\end{remark}

\begin{rem}
  \label{rem:integralexamples}
    If $G\subset\R^n$ is a definable open subset with the standard orientation
    and $\omega=\md x_1\wedge\dots \wedge \md x_n$,
    then $\int_G\omega=\vol(G)$.
    This number is always finite if $G$ is bounded.
\end{rem}

We will
 see that the example of the volume form is really the general case,
but before that we need to establish a technical lemma.

\begin{lemma}\label{lem:choose_nhd}
 Let $(M,\phi_i)$ be a definable manifold with corners and $x\in M$. Then there
 is a definable open neighourhood $U_x\subset M$ with compact closure and such that
 $\bar{U}_x\subset U_i$ for some $i$.
\end{lemma}
\begin{proof}
 We fix $i$ such that $x\in U_i$. Recall that $V_i=\phi_i(U_i)$ is open in
 $\mathbb{H}:=\R^n\times\R^m_{\geq 0}$.
 Hence there is a definable $0<r<\infty$ such that the open
 ball $\mathbb{H}\cap B_r(\phi_i(x))$ is contained in $V_i$. Let $a\in \mathbb{H}$ be definable with
 distance at most $r/4$ from $\phi_i(x)$. Put
 $V_x=B_{r/2}(a)\cap\mathbb{H}$.
 Then $\bar{V}_x\subset B_r(\phi_i(x))\cap\mathbb{H}$
 is a compact %
set contained in~$V_i$.
 We put $U_x=\phi_i^{-1}(V_x)$.
\end{proof}

\begin{lemma}
 \label{lem:Zlinearvolume}
 A finite $\Z$-linear combination of volumes
 of definable bounded open subsets of~$\R^d$
 is up to sign the volume of a definable bounded open subset of~$\R^{d}$.
\end{lemma}
\begin{proof}
 All contributions with a positive coefficient
 can be combined into a single one by taking the disjoint union of
 translates of the definable sets. In the same way all contributions
 with a negative coefficient can be combined into a single one. So it
 suffices to prove that the difference of the volumina of two
 definable bounded open subsets of $\R^{d}$ is up to sign the volume
 of a definable bounded open subset of $\R^d$.
 
The argument of Viu-Sos, see \cite[Section~3]{viu-sos} in
 the semi-algebraic setting works identically in the definable case
 and provides what we want.
We recap his argument: Let $B_+$ and $B_-$ be open bounded
definable subsets of $\R^d$ such that (without loss of
generality) $\vol(B_+)>\vol(B_-)$. We cover $B_+$ and $B_-$ by a mesh
of sidelength $\epsilon$. Let $N_+$ be the number of closed
cubes fully contained in $B_+$ and $N_-$ the number of cubes meeting
the closure of $B_-$. Then 
\[ \vol(B_+)\geq N_+\epsilon^d,\quad N_-\epsilon^d\geq \vol (B_-).\]
By making $\epsilon$ smaller and smaller, we approximate the volumes
from below and above, respectively; 
indeed, see \cite[Theorem 3.4]{LaczkovichSos}
and use that $B_{+}$ and $B_-$ are
Jordan measurable, see Remark~\ref{rem:jordan}.
If the approximation is good enough, this implies $N_+>N_-$. For every
cube meeting the closure of $B_-$ we choose a cube contained in $B_+$.
We can then remove a copy of one from the other. The result is a definable subset with total volume $\vol(B_+)-\vol(B_-)$.
\end{proof}

Recall that we work with  $k$-definable sets in a fixed o-minimal structure $\mathcal{S}$ expanding $\rcf{\R}{}$.
\begin{thm}
 \label{thm:omin_volume}
  Let $p\geq 1$, and $(M,\phi_i)$ be a definable $C^p$-manifold with corners,
  $G\subset M$ a pseudo-oriented  compact 
definable subset of dimension $d$.
  Let $\omega$ be a differential form of degree $d$ on $G$
  as in \cref{defn:diff_form}.
  Then
  \[
   \int_G \omega
  \]
  converges absolutely.
  If $\omega$ is definable, then the value is up to a sign
  the volume of a definable bounded open subset of $\R^{d+1}$.
\end{thm}

\begin{proof}
 We are going to rewrite our integral
 as a finite $\Z$-linear combination of other integrals.
 Eventually these summands will be absolutely convergent,
 proving absolute convergence of the original integral.
 In the definable case,
 every summand will be written as a difference
 of volumes of bounded definable open subsets of $\R^{d+1}$.
 By Lemma~\ref{lem:Zlinearvolume} this will imply
 that the original volume is up to a sign
 the volume of a single definable bounded open subset of $\R^{d+1}$
 and hence finish the proof of the theorem.

 We begin by showing how to reduce to the case $M=\R^n$.
 By \cref{lem:choose_nhd} 
each point $x\in G$
 has a definable open neighbourhood $U_x$ in $M$
 such that $\bar{U}_x$ is compact
 and contained in one of the finitely many charts of $M$.
 By hypothesis $G$  is compact,
 so it is covered by finitely many such neighbourhoods;
 let us call them $U_1,\ldots,U_a$.
The $\bar{U}_i$ and their multiple intersections inherit a pseudo-orientation from $G$. 
By the inclusion-exclusion principle from \cref{cor:int_formel},
we have
 \[ \int_{G}\omega=
  \sum_{i=1}^a\int_{
   \bar{U}_{i}}\omega-\sum_{i<j}\int_{ \bar{U}_{i}\cap
   \bar{U}_{j}}\omega \pm\dots .\]
 We now replace $G$ by one of the $\bar{U}_i$ 
 (or multiple intersections),  making it affine.
From now on we assume $M=\R^n$.

 We have $\omega=\sum_I a_I\md x_I$.
 Again, it suffices to treat the summands separately.
 After a coordinate permutation
 we may assume without loss of generality
 that $\omega =a \md x_1\wedge\dots\wedge \md x_d$ where $a$ is continuous on~$G$.
 Recall that $\pi \colon \R^n \to \R^d$ denotes
 the projection onto the first $d$ coordinates.
 Let $y_1,\dots,y_d$ be the coordinates on $\R^d$.
 Hence
 \[
  \pi^*(\md y_1\wedge\dots\wedge \md y_d)=\md x_1\wedge\dots \wedge \md x_d.
 \]

 We let $G_0,G_1,\ldots,G_N$ be pairwise disjoint as
 in \cref{lem:Gdecomposition} applied to~$G$.
 In particular, $G_0\cup G_1\cup\cdots \cup G_N$ equals $G$
 up to a subset of dimension at most $d-1$.
 All $G_i$ inherit a pseudo-orientation from $G$ and all
   $\pi|_{G_i}$ with $i\ge 1$ are charts. We may
   replace each such
   $G_i$ by a finite union of open subsets, again up to a subset
   of dimension $d-1$, and assume that all $G_1,\ldots,G_N$
   carry an orientation in
   the classical sense and that $\pi|_{G_i}\colon
   G_i\rightarrow \pi(G_i)$ is orientation preserving. Thus
 \[ \int_G \omega = \sum_{i=0}^N \int_{G_i} \omega\]
 by \cref{cor:int_formel}
 if all integrals on the right
 converge absolutely.
 Thus it suffices again   to treat a single $\int_{G_i}\omega$.
 
  We begin with the easy case $i=0$. By assumption, all fibres of
  $\pi|_{G_0}$ have positive dimension, hence the dimension of $\pi(G_0)$ is strictly smaller than $d$. Thus the restriction of any $d$-form to $\pi(G_0)$ vanishes, and so $\pi|_{G_0}^*=0$ on differential forms of degree $d$.
  In particular the restriction of $\omega=a\md x_1\wedge\cdots\wedge \md x_d$ to~$G_0$
  vanishes.
 Hence $\int_{G_0}\omega$ converges absolutely with value $0$, the volume of $\emptyset$.

  Now we treat $G_i$ with $i\ge 1$.
    Then $\pi|_{G_i} : G_i \rightarrow
  \pi(G_i) \subset\R^d$  is a chart and thus has an
  inverse $\psi_i:\pi(G_i) \rightarrow G_i$.
  Note that $\psi_i$ is of $C^p$-class. %
  The integral
  \[ \int_{\pi(G_i)} a\circ \psi_i \md y_1\wedge \cdots \wedge \md y_d\]
  converges absolutely as $a$ is continuous on the compact~$G$
  and thus in particular bounded on $G_i$.
  Finally,
\[ \psi_i^* (\md x_1\wedge
  \cdots \wedge \md x_d) = \md y_1\wedge \cdots \wedge \md y_d\]
 as
  $\pi\circ\psi_i$ is the identity. Thus
  \[ \int_{G_i} a \md x_1\wedge\cdots \wedge \md x_d = \int_{\pi(G_i)} \psi_i^*(a
  \md x_1\wedge \cdots \wedge \md x_d) = \int_{\pi(G_i)} a\circ\psi_i
  \md y_1\wedge\cdots\wedge  \md y_d\]
 converges absolutely.

  Suppose that $\omega$ is definable; then $a$ is definable.
  It remains to show that $\int_{\pi(G_i)} a\circ\psi
  \md y_1\wedge\cdots\wedge \md y_d$
   is the volume of a definable bounded
  open subset of $\R^{d+1}$. This integral equals
  \begin{equation*}
    \int_{C_+} a\circ\psi_i \md y_1\wedge\cdots \wedge  \md y_d -
        \int_{C_-} |a\circ\psi_i| \md y_1\wedge\cdots \wedge \md y_d
  \end{equation*}
  with $C_{\pm} = \{y\in \psi_i(G_i) \mid {\pm a}(\psi_i(y)) > 0\}$ both
  definable bounded and open in $\R^d$.   Hence
  it equals $\vol(U_+) - \vol(U_-)$ with
  $U_{\pm } =  \{(y,z)\in C_{\pm} \times \R : 0<z < |a(\psi_i(y))|
  \}$.
  Note that $U_{\pm}$ are both definable bounded and open in
  $\R^{d+1}$. This difference is the volume of a definable bounded
  open subset of $\R^{d+1}$ by \cref{lem:Zlinearvolume}.
\end{proof}

\begin{rem}
  Let us explain why we cannot replace $\R^{d+1}$ by $\R^d$ in the theorem above.
  Consider the half-circle $G = \{(x,y) \in \R^2\mid y\geq 0, x^2+y^2=1\}$.
  It is compact, semi-algebraic and definable without parameters.
Then $\reg_1(G)$ is the interior of $G$,
independent of $p$. We orient it anti-clockwise as usual.
  Now   $\int_G y\md x  = -\int_{-1}^1 \sqrt{1-x^2}\md x=- \pi/2 $.
  As $\pi$ is transcendental, $\int_G y\md y$ cannot be the volume of
  $\Q$-semi-algebraic subset of $\R$.
\end{rem}
\begin{rem}The natural way of computing the integral is to pull the differential form back to a chart (via the inverse of the chart map) and evaluate there. However, this pull-back involves a Jacobian matrix. Its entries are not bounded in general, hence convergence is not automatic.

Here is an explicit example:
Let $M=\R^2$, $G=\{(y^2,y)| y\in [0,1]\}$, $\omega=a\md x_1+b\md x_2$ for continuous $a,b$ on $G$. We have $\reg_1(G)=\{(y^2,y)|y\in (0,1)\}$. It is a submanifold. We can use the projections $\pi_1$ and $\pi_2$ to the first or second coordinate as a chart. In both cases the image in $\R$ is the open interval $I=(0,1)$. The inverse $\psi_1:I\to G$  of $\pi_1$ is $t\mapsto (t,\sqrt{t})$. Its Jacobian matrix is
\[ \left(1,\frac{1}{2\sqrt{t}}\right).\]
The second entry is unbounded on $I$. We have
\[ \psi_1^*\omega= (a\circ\phi_1) \md t + (b\circ\phi_1) \frac{1}{2\sqrt{t}} \md t.\]
The coefficient function is unbounded. (Note that $a\circ \phi_1$ and
$b\circ \phi_1$ are bounded because $a$ and $b$ are. Note also that
differentiability of $a$ and $b$ does not come into play. It suffices that they are continuous.)
The solution is to treat the summands
$a\md x_1$ and $b\md x_2$ separately and use the projection $\pi_1$ for the first summand and $\pi_2$ for the second summand. We then interpret
\[ a\md x_1=\pi_1^* ( (a\circ\phi_1)\md t), \quad b\md x_2=\pi_2^*( (b\circ \phi_2)\md t)\]
and the convergence issue disappears.
\end{rem}

\begin{rem}
 A similar convergence argument for integrals
 can also be found in \cite{hanamura_et_al_I}. They treat explicitly the case of $C^\infty$-forms,
 but actually this assumption is not needed.

 Alternatively, convergence also follows from the existence
 of triangulations that are strictly of class~$C^1$,
 shown in \cite{omin-triang}.
\end{rem}

\section{Oriented real blow-up}

\label{realblow}

The oriented real blow-up is a natural construction
in the context of semi-algebraic geometry.
Nevertheless,
it seems that little is written about it from this point of view.
The construction is discussed in
\S{}I.3 of~\cite{Maj84}, \S3.4 of~\cite{fresan-jossen} and~\cite{gillam}.
One of the main purposes of this section 
is to argue
that the oriented real blow-up is semi-algebraic
(in other words, definable in $\R_\alg$) with suitable parameters.
For a general discussion we refer the reader to the aforementioned sources.

Let $X$ be a topological space,
let $\pi \colon L \to X$ be a complex (topological) line bundle on~$X$,
and let $s \colon X \to L$ be a section.
Let $L^*$ be the complement of the zero section, which carries an action of $\C^*$.
We put
\[
 B^*_{L,s} = \{l \in L^* \mid s(\pi(l)) \in \R_{\geq 0}l\}.
\]
Let $x\in X$. If $s(x) = 0$, then $B^*_{L,s}$ contains $L_x \ssm \{0\}$;
otherwise, it contains the unique open half-ray generated by~$s(x)$.
In particular, $B^*_{L,s}$ is stable under the fibrewise action of~$\R_{> 0}$.

Following~\cite{gillam} and
  \cite[Section~3.4.2]{fresan-jossen},
we call the quotient the \emph{simple oriented real blow-up}:
\[
 \Blo_{L,s}(X) = B^*_{L,s}/\R_{>0}.
\]
It is equipped with the quotient topology.
The simple oriented real blow-up comes equiped with a natural projection map
$\pi \colon \Blo_{L,s}(X) \to X$
that is a  homeomorphism outside the zero locus of~$s$.
Multiplying $s$ by a continuous function $X\to\C^*$ produces a blow-up that is canonically isomorphic to $\Blo_{L,s}$.

If $X$ is a complex analytic space, $D \subset X$ an effective Cartier divisor, 
$L=\Oh(D)$ with the system $(U_i,s_i)_{i\in I}$ of tautological sections, then the $\Blo_{L|_{U_i},s_i}$ glue to a globally defined real oriented blow-up
$B_D$ or $\Blo_D(X)$.

\begin{ex}\label{ex:Ptilde}
 The oriented real blow-up $\tilde{\Pe}^1 := \Blo_{\infty}(\mathbb{P}^1_\C)$
 is a compactification of~$\C$ by a circle at infinity.
 The details of the following picture will be explained
 as we describe the general situation in local coordinates.
 
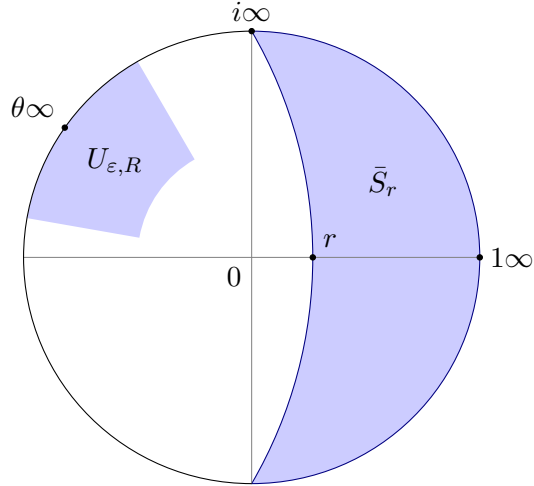
\begin{center}
  \begin{tikzpicture}
   \fill[fill=blue!20!white]
   (120:3) arc (120:170:3) -- (170:1.5) arc (170:120:1.5) -- cycle;
   \filldraw[fill=blue!20!white, draw=blue!50!black]
   (-90:3) arc (-90:90:3) arc (30:-30:6) -- cycle;
   \draw[help lines] (-3,0)--(3,0);
   \draw[help lines] (0,-3)--(0,3);
   \draw (90:3) arc (90:270:3);
   \node[anchor=north east] at (0,0) {$0$};
   \draw[fill] (.804,0) circle (1pt) node[anchor=south west] {$r$};
   \draw[fill] (0:3) circle (1pt) node[anchor=west] {$1\infty$};
   \draw[fill] (90:3) circle (1pt) node[anchor=south] {$i\infty$};
   \draw[fill] (145:3) circle (1pt) node[anchor=south east] {$\theta\infty$};
   \node at (145:2.2) {$U_{\varepsilon,R}$};
   \node at (30:2) {$\bar S_r$};
  \end{tikzpicture}
 \end{center}
 For every $z \in S^1 = \{z \in \C \mid \, |z| = 1\}$
 there is a point $z\infty$ on the boundary:
 it is the point of intersection of the boundary
 and the half-ray $z \cdot \R_{\ge 0}$.
 A system of open neighbourhoods around $z\infty$ is given by the sets
 \begin{multline*}
  U_{\varepsilon,R} =
  \{w \in \C \mid \, |w| > R \text{ and } |{\arg}(w) - \arg(z)| < \varepsilon \}
 \\
 \cup \{w\infty \mid \, |\arg(w) - \arg(z)| < \varepsilon\}
 \end{multline*}
 for small~$\epsilon$ and positive real~$R$.

 The closure of the set $S_r = \{z \in \C \mid \Re(z) \ge r\}$
 is given by the union of $S_r$ and the half-circle
 $\{z\infty \mid \Re(z) \ge 0 \}$.
\end{ex}

Suppose that $L_1, \ldots, L_n$ are line bundles on~$X$
with respective sections $s_1, \ldots, s_n$,
and put $L = L_1 \otimes \cdots \otimes L_n$
with section $s_1 \otimes \cdots \otimes s_n$.
We may then form the fibre product
\[
 \Blo_{L_1,s_1}(X) \times_X \cdots \times_X \Blo_{L_n,s_n}(X)
\]
which naturally maps to $\Blo_{L,s}(X)$.

\begin{defn}\label{defn:orb}
 Let $X$ be a smooth analytic space,
 and let $D \subset X$ be a simple normal crossings divisor.
 Denote the (smooth) irreducible components of~$D$ by $D_1, \dots, D_n$.
 The \emph{oriented real blow-up} of~$X$ in~$D$, denoted by $\OBl_D(X)$
 is the fibre product
 \[
  \Blo_{D_1}(X) \times_X \cdots \times_X \Blo_{D_n}(X).
 \]
 Note that $\OBl_D(X)$ comes with a natural projection map to~$X$.
\end{defn}

One topological intuition for $\OBl_D(X)$
is the complement of a tubular neighbourhood of~$D$ in~$X$.
We now make this picture precise
by a description in local coordinates.

Consider a domain $U$ in~$\C^n$ and $D = D_1 \cup \ldots \cup D_m$
the union of the first~$m$ coordinate hyperplanes (intersected with~$U$).
In that case we have the following explicit description
of $\OBl_{D}(U)$ 
\begin{multline}
 \label{orb-local-coords}
 \{
  (z_1, \ldots, z_n, w_1, \ldots, w_m) \in \C^n \times (S^1)^m \mid\\
 (z_1,\dots,z_n)\in U, 
  z_iw_i^{-1} \in \R_{\ge0} \text{ for $1 \le i \le m$}
 \}
\end{multline}
and $\pi$ is the projection
$(z_1, \ldots, z_n, w_1, \ldots, w_m) \mapsto (z_1, \ldots, z_n)$.
In particular, it is a $C^\infty$-manifold with corners.
Local coordinates are defined by 
\begin{align}\label{eq:sa_chart}
 \OBl_{D}(U)&\to \R_{\geq 0}^{m}\times (S^1)^m\times \C^{n-m}\\
(z_1,\dots,z_n,w_1,\dots,w_m)&\mapsto \left(\frac{z_1}{w_1},\dots,\frac{z_m}{w_m}, w_1,\dots,w_m,z_{m+1},\dots,z_n\right).
\notag
\end{align}

In particular, this gives $\OBl_D(X)$  the structure of a manifold with corners.
As a consequence, we obtain the following result.

\begin{proposition}\label{prop:is_sa}
 Let $k\subset\C$ be a field which is algebraic over $k_0=k\cap\R$.
 Let $X$ be a smooth algebraic variety over~$k$
 and let $D \subset X$ be a simple normal crossings divisor.
 Then the oriented real blow-up $\OBl_D(X^\an)$
 can naturally be endowed with a structure of $k_0$-semi-algebraic
 $C^\infty$-manifold with corners (see Definition~\ref{defn:manifold})
 in such a way that the natural projection map
 $\pi \colon \OBl_D(X^\an) \to X^\an$
 is morphism of $k_0$-semi-algebraic $C^\infty$-manifolds with corners.
 \begin{proof}
  Without loss of generality $k_0=\tilde{k}$ is real closed
  and $k=\bar{k}$ algebraically closed.
  Let $(\bar{X},\bar{D})$ be a good compactification of the log pair $(X,D)$.
  It suffices to prove the proposition for $(\bar{X},\bar{D})$
  because $\OBl_D(X^\an)$ is the preimage of $X^\an$ in $\OBl_{\bar{D}}(\bar{X}^\an)$.
  In other words, without loss of generality we assume that $X^\an$ is compact.

  Without loss of generality we assume that $X$ is connected.
  By definition, for every point $x\in X$,
  there is a Zariski-open neighbourhood~$U_x$
  and an \'etale map $p \colon U_x\to \A^d$ (with $d=\dim X$)
  such that $p(x)=0$ and $D\cap U_x=p^{-1}(\{z_1\cdots z_m=0\})$.  
  By the semi-algebraic inverse function theorem \cite[Proposition~2.9.7]{BCR}
 or \cite[Chapter~7~(2.11)]{D:oMin}
  the map $p^\an$ is invertible on an open ball $B_x$ around $0$ in $\C^d$.
  Let $V_x=p^{-1}(B_x)\subset X^\an$.
  The coordinate functions $z_1,\dots,z_m$ are both holomorphic
  and $k_0$-semi-algebraic.
  Hence the preimage
  \[ \pi^{-1}(V_x)\subset \OBl_D(X^\an)\]
  has the shape described after Definition~\ref{defn:orb}. The map
  (\ref{eq:sa_chart}) defines a chart. More precisely, we also need to cover $S^1\subset\R^2$ by finitely many semi-algebraic charts.
As $X^\an$ is compact, finitely many of
  the $V_x$ suffice to cover $X^\an$. The transition maps are $C^\infty$ and
  $k_0$-semi-algebraic because the transition maps between the $p(V_x)$
  are holomorphic and $k_0$-semi-algebraic.
 \end{proof}
\end{proposition}

\begin{lemma}\label{lem:blow-up_functorial}
 The construction of the oriented real blow-up is functorial:
 Let $X_1$ and~$X_2$ be smooth analytic spaces,
 and let $D_i \subset X_i$ be  simple normal crossings divisors.
 Let $f \colon X_1 \to X_2$ be a morphism
 such that $f^{-1}(D_2) \subset D_1$.
 Then there is a natural morphism $\tilde f$
 such that the following diagram commutes:
 \[\begin{tikzcd}
  \OBl_{D_1}(X_1) \rar{\tilde f} \dar & \OBl_{D_2}(X_2) \dar \\
  X_1 \rar{f} & X_2\ .
 \end{tikzcd}\]
If $f$ is a morphism of smooth algebraic varieties, then $\tilde{f}$ is
a $C^\infty$-morphism of $k_0$-semi-algebraic manifolds with
corners.%
 \begin{proof}
  Compute in local coordinates. Let $x_1 \in X_1$ with image point $x_2$. We choose local holomorphic coordinates $(z_1,\dots,z_{n_1})$ centered at $x_1$ and $(s_1,\dots,s_{n_2})$ centered at
$x_2$, and write $f=(f_1, \dots,f_{n_2})$ where all $f_i$ are holomorphic functions
	 in $z_1,\dots,z_{n_1}$. Without loss of generality, the divisors are of the form
$D_1=\{z_1\cdots z_{m_1}=0\}$ and $D_2=\{s_1\cdots s_{m_2}=0\}$. The local coordinates for the oriented real blow-up are
$(r_1,\phi_1,\dots,r_{m_1},\phi_{m_1},z_{m_1+1},\dots z_{n_1})$ in
$(\R_{\geq 0}\times S^1)^{m_1}\times\C^{n_1-m_1}$ and
$(\rho_1,\psi_1,\dots,\rho_{m_2},\psi_{m_2},s_{m_2+1},\dots s_{n_2})$ in
$(\R_{\geq 0}\times S^1)^{m_2}\times\C^{n_2-m_2}$. In these coordinates, the extension
$\tilde{f}$ of $f$ is of the form
\[ (\tilde{f}_1,\dots,\tilde{f}_{m_2},f_{m_2+1},\dots,f_{n_2})\]
where
\begin{equation}\label{eq:ftilde}
 \tilde{f}_i=\left(|f_i|,\frac{f_i}{|f_i|}\right)\in\R_{\geq 0}\times S^1
\end{equation}
for $1\leq i\leq m_2$. The issue is well-definedness of \eqref{eq:ftilde}.
We fix $i$ and drop the index. By assumption the vanishing locus of
$f$ is contained in $D_1$. Hence $f$ is of the form %
\[ f(z_1,\dots,z_{n_1})=z_1^{v_1}\dots z_{m_1}^{v_{m_1}}h(z_1,\dots,z_{n_1})\]
with $v_j\geq 0$ and $h(0)\neq 0$. From this
\[ \tilde{f}=\left( r_1^{v_1}\dots r_{m_1}^{v_{m_1}}|h(\underline{z})|,\phi_1^{v_1}\dots,\phi_{m_1}^{v_{m_1}}h(\underline{z})/|h(\underline{z})|\right)\]
and the expression is well-defined near $0$.

In the algebraic situation, we choose the $z_i$ and $s_j$ as regular parameters.
The resulting maps are $C^\infty$ and
$k_0$-semi-algebraic. Finitely many charts suffice, as
each $\mathrm{OBl}_{D_i}(X_i)$ is covered by finitely many charts. 
 \end{proof}
\end{lemma}

\begin{rem}
 \label{alt-notn-blow-up}
 In the future, it will often be the case
 that we start with a variety~$X$ that is not proper,
 and consider the oriented real blow-up of the boundary divisor~$X_\infty$
 of a compactification~$\bar X$ of~$X$.
 In such a situation, we will write $B_{\bar X}(X)$
 instead of $\OBl_{X_\infty}(\bar X)$.
\end{rem}

\begin{rem}
 It is not clear to us whether $\OBl_D(X)$
 is affine as semi-algebraic $C^1$-manifold with corners.
 In other words, does there exist
 a semi-algebraic $C^1$-embedding of $\OBl_D(X)$ into~$\R^n$?
 Compare with \cref{affine-semialg-spaces}.
\end{rem}

\section{Naive exponential periods}\label{sec:naive}

Let $k\subset\C$, $k_0=k\cap\R$ and assume that
$k$ is algebraic over $k_0$, see the discussion in Section~\ref{ssec:field}.
Recall from Definition~\ref{defn:naive} the notion of a naive
exponential period. We denote
by $\Pnaive(k)$ the set of naive exponential periods.
Let $\Ptilde$ denote the real oriented blow-up
of $\Pe^1$ at the point at infinity, see Example~\ref{ex:Ptilde}.

\subsection{Examples of integrals}
\label{eg-naive}
We first consider some instructive examples.

\begin{ex}
 Let $G = [1,\infty) \subset \C$, $f = \frac{1}{z}$, $\omega = \md z$.
 Consider
 \[ \int_G \e^{-f}\omega = \int_1^\infty \e^{-\frac{1}{t}}\md t
 = \int_1^0 -\e^{-s}\frac{1}{s^2}\md s .\]
 It does not converge.
 Note that the image $f(G) = (0,1]$ is not closed,
 hence $f \colon G \to \C$ is not proper.
 The properness condition in the definition of a naive exponential period
 was added to exclude cases like this.
\end{ex}

\begin{ex}
 Once again let $G = [1,\infty) \subset \C$, $f=\frac 1z$,
 but $\omega = \frac{1}{z^2}\md z$.
 As in the previous example,
 the data do not satisfy the definition of a naive exponential period
 because $f \colon G \to \C$ is not proper.
 However, this time
 \[ \int_G \e^{-f}\omega = \int_1^\infty \e^{-\frac{1}{t}}\frac{1}{t^2}\md t
 = -\int_1^0\e^{-s}\md s\]
 converges.
 It can be understood as a naive exponential period with $G'=[0,1]$, $f'=z$, $\omega'=\md z$.
\end{ex}

\begin{ex}\label{ex:3}
 Let $s\in S^1$ with $\Re(s)>0$.
 Consider the half ray $G_s=\{ r s \mid r\geq 0\}$, $f=z$, $\omega=\md z$.
 If $s \neq 1$, this data do not satisfy the definition of a naive
 exponential
 period
 because $f(G_s)=G_s$ does not have bounded imaginary part.
 Nevertheless,
 \[ \int_{G_s}\e^{-f}\md z
  = \int_0^\infty \e^{-rs} s \md r
  = -\e^{-rs} \Bigr\rvert^\infty_0=1\]
 converges and is obviously an exponential period.  Note that it is independent of~$s$.
 Actually, $G_s$ defines a class in $H_1^\rd(\A^1,\{0\};\Z)$, see Section~\ref{ssec:rd} below,
 because its closure in $\Ptilde$
 is contained in $\Bcirc=\Bcirc_{\Pe^1}(\A^1,\id)$.
 The homology class is independent of $s$
 (fill in the triangle between $G_1$ and $G_s$,
 the third edge is in $\partial \Bcirc$).
 The period integral only depends on the homology class,
 hence the independence follows from the abstract theory as well.
 We do not allow $G_s$ in our definition of a naive
 exponential period,
 but the same number can be obtained as a naive exponential period for $G_1$.
 This is a general feature, see \xref{prop:comp_homol}.
 In Definition~\ref{defn:gnaive}
 we will introduce the notion of a generalised naive exponential period
 which allows all $G_s$.
\end{ex}

\subsection{General properties}

\begin{defn}\label{defn:gnaive}
A \emph{generalised naive exponential period} over $k$
is a complex number of the form
\[ \int_G\e^{-f}\omega\]
where $G\subset \C^n$ is a pseudo-oriented closed $k_0$-semi-algebraic
subset, $\omega$ is a rational algebraic differential form on $\A^n_k$ that
is regular
on $G$ and $f$ is a rational function on $\A^n_k$ such that
$f$ is regular and proper on $G$ and, moreover, the closure of $f(G)$ in $\Ptilde$
is contained in $\Bcirc=\C\cup \{s\infty \mid s\in S^1, \Re(s)>0\}$.
We denote by
$\Pgnaive(k)$ the set of generalised naive exponential periods.
\end{defn}

We are going to show in Corollary~\ref{cor:nv_converges} that these
generalised naive exponential periods converge absolutely. For
the rest of this section we assume absolute convergence.
The assumption $G\subset\C^n$ is surprising when comparing to the literature on ordinary periods, see Remark~\ref{rem:surprise} below.

\begin{lemma}\label{naive-is-generalised}
 Naive exponential periods are generalised  naive exponential periods.
\end{lemma}
\begin{proof}
The condition $f(G)\subset S_{r,s}$ from Definition~\ref{defn:naive} implies $\overline{f(G)}\subset \Bcirc$.
\end{proof}
\begin{lemma}
 \label{lem:change_field}
 The sets $\Pnaive(k)$ and $\Pgnaive(k)$ are $\bar{k}$-algebras.
 Moreover, $\Pnaive(k)=\Pnaive(\bar{k})$ and $\Pgnaive(k)=\Pgnaive(\bar{k})$.
\end{lemma}
\begin{proof}
 The arguments are the same for both notions. We formulate it for
 naive exponential periods.

 For the first statement we use the same argument as for $f=0$,
 see \cite[Proposition~12.1.5]{period-buch}.

 We give the argument for the second.
 Let $L/k$ be a finite subextension of $\bar{k}/k$.
 Since $k$ is algebraic over $k_0$, the extension $L/L_0$ with
 $L_0=L\cap \R$ is also algebraic.
 Hence, $\Pnaive(\bar{k})=\bigcup_{L/k}\Pnaive(L)$ where $L$ runs through
 all finite subextensions of $\bar{k}/k$.
 Thus it suffices to show that
 $\Pnaive(k)=\Pnaive(L)$ for $L/k$ finite.

As $L/k$ is separable, we have $L\isom k[T]/P$ for an irreducible polynomial $P$. This equation defines $\Spec(L)\subset\A^1$ over $k$. More generally,
$\A^n_L=\A^n\times_{\Spec(k)}\Spec(L)$ can be understood as an affine $k$-varity contained in $\A^{n+1}_k$.
 We call it $\tilde{A}$.
 Then
 \[ \tilde{A}\times_k\C =\bigcup_{\sigma:L\to \C} \A^n_\C\]
 where $\sigma$ runs through all embeddings of $L$ into $\C$ fixing $k$.
 If $\int_G\e^{-f}\omega$ is a naive exponential
 period over $L$, then $f$ and $\omega$
 are defined over $k$ when viewed on $\tilde{A}\subset \A^{n+1}_k$.
 The extension $L_0/k_0$ is algebraic,
 hence every $L_0$-semi-algebraic set is also
 $k_0$-semi-algebraic.
\end{proof}

In particular, we can move between $k$, $\bar{k}$, $\bar{k}\cap\R$ and
$k_0=k\cap\R$ without changing the set of naive exponential
or generalised
exponential
naive periods.

\begin{lemma}\label{lem:real_version}
  Let $k=k_0\subset\R$. The following are equivalent for $\alpha\in\C$:
\begin{enumerate}
\item The number $\alpha$ is a naive exponential period over~$k$.
\item
It can be written as
\[ \alpha=\int_G\e^{-f}\omega\]
with $G\subset \R^n$ a pseudo-oriented
closed $k$-semi-algebraic subset of dimension~$d$,
$f\in k(i)(z_1,\dots,z_n)$ regular on $G$ such that $f|_G \colon G\to\C$ is
proper with image contained in $S_{r,s}$ and
$\omega\in\Omega^d_{k(i)(z_1,\dots,z_n)/k(i)}$ is regular on $G$.
\item Its real and imaginary part can be written as
\begin{align*}
\Re(\alpha)&= \int_G\left(\phantom{-}\cos(f_2)\e^{-f_1}\omega_1 + \sin(f_2)\e^{-f_1}\omega_2\right)\\
   \Im(\alpha) &=
   \int_G\left(-\sin(f_2)\e^{-f_1}\omega_1 + \cos(f_2)\e^{-f_1}\omega_2\right)
\end{align*}
with $G\subset\R^n$ a pseudo-oriented closed $k$-semi-algebraic subset of dimension~$d$, $f_1,f_2\in k(z_1,\dots,z_n)
$ regular on $G$ such that the map $f_1|_G$ is proper,
$f_1(G)$ is bounded from below, $f_2(G)$ is bounded, and
$\omega_1,\omega_2\in\Omega^d_{k(z_1,\dots,z_n)/k}$ regular on $G$.
\end{enumerate}
Moreover, $f_1,f_2$ in (3) are the real and imaginary parts of $f$ in
(2), respectively, and similarly for $\omega_1,\omega_2$.
Finally, $\alpha$ is a generalised naive exponential period if and only if it can be written
as in (2) with $\overline{f(G)}\subset \Bcirc$ instead of 
$f(G)\subset S_{r,s}$.
\end{lemma}
\begin{proof}
 Let
  $G,f,\omega$ be as in the definition of a naive exponential period. By
  definition $G\subset\C^n$ with coordinates $z_1,\dots,z_n$.
  By sending a complex number to its real and imaginary part
  we view $G$ as a real subset $G'$ of~$\C^{2n}$
  with coordinates $x_1,y_1,x_2,y_2,\dots,x_n,y_n$.
  Let $\Sigma \colon \C^{2n}\to \C^n$ be given by
  $(x_1,y_1,\dots,x_n,y_n)\mapsto (x_1+iy_1,\dots,x_n+iy_n)$.
  By definition $\Sigma(G')=G$, compatible with the pseudo-orientation.
  Put $f'=\Sigma^*(f)$ and $\omega'=\Sigma^*(\omega)$.
  Then by the transformation rule
  \[ \int_{G'}\e^{-f'}\omega'=\int_G\e^{-f}\omega.\]
  Note that $f'$ and $\omega'$ are defined over $k(i)$.
  This shows that (1) implies~(2).
  Conversely, a number of the form in~(2)
  is by definition a naive exponential period over $k(i)$.
  By \cref{lem:change_field} it is also a naive exponential period over $k$,
  so (2) implies~(1).
  Let $G,f,\omega$ be as in~(1).
  We put $f=f_1+if_2$ and $\omega=\omega_1+i\omega_2$ and compute $\e^{-f}\omega$. The regularity and boundedness conditions on $f$ and $\omega$ are equivalent to the condition on $f_1,f_2$ and $\omega_1,\omega_2$. The map
$(f_1,f_2)$ is proper if and only if the preimages of boxes $I\times J$ are compact for all bounded closed intervals $I,J\subset\R$. As $f_2$ is bounded, say by $s$, it suffices to consider $J=[-s,s]$. We have
\[ (f_1,f_2)^{-1}(I\times [-s,s])=f_1^{-1}(I).\]
It is compact for all $I$ if and only if $f_1$ is proper.
  So properties (2) and~(3) are equivalent.

  The final claim follows as the equivalence proof of (1) and~(2).
\end{proof}

\begin{cor}\label{cor:conjugate}
The subset $\Pnaive(k)\subset\C$ is stable under complex conjugation. A complex number is a naive exponential period over $k$ if and only if its real an imaginary part are.
\end{cor}
\begin{proof}The second claim follows from the first because $\Pnaive(k)$ is a
$\Q$-vector space, see Lemma~\ref{lem:change_field}.

Let $\alpha$ be a naive exponential period. We write it as in Lemma~\ref{lem:real_version} (2) as
\[ \alpha=\int_G\e^{-f}\omega\]
with $G$, $f=f_1+if_2$, $\omega=\omega_1+i\omega_2$ as in Lemma~\ref{lem:real_version} (3). In particular, $G\subset\R^n$. It suffices to consider $\omega_1$ and $i\omega_2$ separately. 
Without loss of generality, $\omega$ is real. The period for the function
$f_1-if_2$ is complex conjugate to the one for $f$.
\end{proof}

\begin{rem}\label{rem:surprise}
Most
references on ordinary periods work with semi-algebraic $G\subset\R^n$. %
A complex number is a period over $k\subset\R$ if and only its real and imaginary part 
are of the form 
\[ \int_G\omega\]
with $G\subset\R^n$ compact semi-algebraic and $\omega$ a rational differential form on $\A^n_k$, regular on $G$. 

The case of naive exponential periods is a little more subtle. It is probably \emph{not}
enough to allow real and imaginary parts of the form
\[ \int_G\e^{-f}\omega\]
	with $G\subset \R^n$ a closed semi-algebraic subset, $\omega$
        a rational differential form over $k\subset\R$, $f$ a rational
        function over $k$ both regular on $G$ and satisfying the decay
        conditions on $f$. As the formulas in
        Lemma~\ref{lem:real_version} (3) show factors $\cos(g)$ and
        $\sin(g)$ for a second rational function over $k$ have to be
        allowed. As the referee pointed out, the problem already appears for numbers of the form $\e^{i\alpha}=\cos(\alpha)+i\sin(\alpha)$ for $\alpha\in\Qbar$.               
\end{rem}

\subsection{Convergence and definability}
The conditions on our domain of integration can be reformulated.
\begin{lemma}\label{lem:GB}
Let $f:\A^n\to\Pe^1$ be a rational function over $k$ and let $G\subset
\C^n$ be a closed semi-algebraic set such that
$f$ is regular and proper on $G$.
Let $\omega$ be a rational differential form on $\A^n$ over $k$.
Let $X\subset \Pe^n$ be the complement of the polar loci of $f$ and $\omega$,
$\bar{X}$ a good compactification of $X$ relative to $f$, see Section~\ref{ssec:gc}.
Let $\bar{G}$ be the closure of $G$ in the real oriented blow-up $B_{\bar{X}}(X)$
of $\bar{X}$ at the divisor at infinity,
see \cref{alt-notn-blow-up}, and $G_\infty=\bar{G}\ssm G$.
(The case $G_\infty=\emptyset$ can occur.)

Then $f$ extends to a semi-algebraic $C^\infty$-map
$\tilde{f}:B_{\bar{X}}(X)\to\Ptilde$ of compact semi-algebraic
$C^\infty$-manifolds with corners with boundary, mapping
$G_\infty$ to $\partial\Ptilde$.
Moreover,
\begin{enumerate}
\item(Naive exponential periods) $f(G)\subset S_{r,s}$ for some $r,s$ if and only if
$\tilde{f}(G_\infty)\subset\{1\infty\}$.
\item(Generalised naive exponential periods)
$\overline{f(G)}\subset \Bcirc$ if and only if
$\tilde{f}(G_\infty)\subset \partial \Bcirc$.
\end{enumerate}
\end{lemma}
\begin{proof}By definition of $X$, we have $\bar{f}^{-1}(\infty)\subset \bar{X}\ssm X$.
 By \cref{lem:blow-up_functorial} we get an induced $C^\infty$-morphism
 of semi-algebraic manifolds with corners~$\tilde{f}$.

Let $(g_i)_{i\geq 1}$ be a sequence in $G$ converging to $g\in\bar{G}$.
Assume $g\in G_\infty$.
We have $g\in\partial B_{\bar{X}}(X)$ because $G\subset X^\an$ is closed.
In particular, the image
of $g$ in $\bar{X}^\an$ is in the complement of $X^\an$.

We claim that $\tilde{f}(g)\notin\C$.
Assume $\tilde{f}(g)\in\C\subset \Ptilde$.
Note that $\lim \tilde{f}(g_i)=\tilde{f}(g)$ by continuity.
As $f$ is proper, $f(G)\subset \C$ is closed. All $\tilde{f}(g_i)$ are in
$f(G)$, hence so is $\tilde{f}(g)$. Let $D\subset \C$ be a closed disk
around $\tilde{f}(g)$. It is compact, hence so is its preimage
$E:=\tilde{f}|_G^{-1}(D)\subset G$ .
There is $N\geq 1$ such that $\tilde{f}(g_i)\in D$ for all $i\geq N$.
Hence their preimages $g_i$ are in $E$. As $E$ is compact, the limit
point $g$ is in $E$, in particular in $G$. This is a contradiction.
We have shown that $\tilde{f}(G_\infty)\subset\partial\tilde{\Pe}^1$.

In order to follow the proofs of (1) and (2), it may helpful to consult
the figure in Example~\ref{ex:Ptilde}.
Note that $\overline{f(G)}=\tilde{f}(\bar{G})$. 
Hence (2) is obvious.
For (1) note that $\bar{S}_{r,s}\cap\partial\Ptilde=\{1\infty\}$.
 Hence $f(G)\subset S_{r,s}$ implies $\tilde{f}(G_\infty)\subset\{1\infty\}$.
Conversely, consider a small open neighbourhood $U$ of $1\infty$ in $\Ptilde$. It
intersects $\C$ inside some strip of the form $S_{r,s}$. As
$\bar{G}$ is compact, so is $G'=\bar{G}\ssm \tilde{f}^{-1}(U)$. The image
$f(G')$ is compact, so bounded in $\C$. By enlarging $r$ and $s$, we ensure
that both $f(G')$ and $f(G)\cap U$ are contained in the same strip.
\end{proof}

\begin{lemma}\label{lem:decay}Let $f$ and $G$ be as in the definition
  of a generalised naive exponential period. Let $\bar{G}$ be
  the compactification of $G$ as in Lemma~\ref{lem:GB} and
  $G_\infty=\bar{G}\ssm G$.
  Let $c$ be a rational function on $\A^n$ which is regular on $G$. The
  extension of $\e^{-f}c$ by $0$ on $G_\infty$ yields a continuous function on
  $\bar{G}$.
\end{lemma}
\begin{proof}
  Let $(g_i)_{i\geq 1}$ be a sequence in $G$ converging to $g\in G_\infty$.
  Then
  \[  |\e^{-f(g_i)}|=\e^{-\Re(f(g_i))}\to 0\]
  because $f(g_i)$ tends to $\tilde{f}(g)\in \partial \Bcirc$. The function
  $c$ has at worst a pole in $g$, but the exponential factors decays faster than
  $|c(g_i)|$ grows. In total
  \[ \lim_{i\to\infty}|\e^{-f(g_i)}c(g_i)|=0. \qedhere \]
\end{proof}

\begin{cor}\label{cor:nv_converges}Assume that $G,f,\omega$ define a
  generalised naive exponential period.
Then
\[ \int_G\e^{-f}\omega\]
converges absolutely.
\end{cor}
\begin{proof}
We apply Theorem~\ref{thm:omin_volume} to $\bar{G}\subset B_{\bar{X}}(X)$
as in Lemma~\ref{lem:GB}. It is compact. By Lemma~\ref{lem:decay}, the
$C^\infty$-form $\e^{-f}\omega$ on $G$ extends to a continuous form
on $\bar{G}$. This is enough.
\end{proof}

\begin{thm}\label{thm:naive_is_volume}
 If a number $\alpha\in\C$ is a naive exponential period over $k$,
 then its real and imaginary part are up to signs volumes of
	compact subsets $S\subset \R^n$ definable in the o-minimal structure
 $\R_{\sin,\exp} = \rcf{\R}{, {\sin}|_{[0,1]},\exp}$
 with parameters from $k_0$.
 \begin{proof}By Lemma~\ref{lem:change_field}, we may assume $k=k_0$.
  We use the characterisation of naive exponential periods given in
  parts~(2) and~(3) of \cref{lem:real_version}.
  Thus $\alpha = \int_G e^{-f}\omega$ with
  \begin{align*}
   \Re(\alpha) &=\int_G\left(
    \phantom{-}\cos(f_2)\e^{-f_1}\omega_1 + \sin(f_2)\e^{-f_1}\omega_2\right),\\
   \Im(\alpha) &=\int_G\left(
    -\sin(f_2)\e^{-f_1}\omega_1 + \cos(f_2)\e^{-f_1}\omega_2\right)
  \end{align*}
  where  $G\subset \R^n$ is closed and $k$-semi-algebraic of dimension $d$
  carrying a pseudo-orientation, $f_1,f_2\in k(z_1,\dots,z_n)$ are regular
  and proper on $G$, $f_1(G)$ is bounded from below, $f_2(G)$ is bounded, and
  $\omega_1,\omega_2\in \Omega^d_{k(z_1,\dots,z_n)/k}$ regular on $G$.

  We want to apply Theorem~\ref{thm:omin_volume}.
  Again we apply it to the compact
  $k_0$-semi-algebraic $C^\infty$-manifold with corners $B_{\bar{X}}(X)$
  of Lemma~\ref{lem:GB} and the closure $\bar{G}$ of $G$ in $B_{\bar{X}}(X)$.
  It is compact and a semi-algebraic subset of $B_{\bar{X}}(X)$,
  hence definable in~$\R_{\sin,\exp}$.
  The forms $\Re(\e^{-f}\omega)$ and $\Im(\e^{-f}\omega)$ are definable
  on $G_\infty$ because they vanish identically. Hence it remains
  to verify the definability on the affine $G$ itself. The forms $\omega_1$ and
  $\omega_2$ are algebraic, in particular definable.
  By assumption $f_2$ is bounded on $G$,
  hence using \Cref{Rsinexp_cos_pi} the function
  $\sin(f_2)$ is definable in our o-minimal structure. The same
  is true for $\cos(f_2)$ because  $\cos(f) = \sin(f + \pi/2)$,
  and $\pi$ is definable in the o-minimal structure $\R_{\sin,\exp}$.
 \end{proof}
\end{thm}

\begin{rem}
 The above argument does not work for generalised naive exponential periods.
 It is essential that the imaginary part of $f$ is bounded on~$G$.
 However, we are going to show (see \xref{thm:compare_all})
 that every generalised naive exponential period is actually
 a naive exponential period, hence the consequence still applies.
\end{rem}

\begin{rem}
In contrast to the case of ordinary periods, we do not expect that all volumes of definable
sets in this o-minimal structure are naive exponential periods. The above argument only produces very special definable sets: there is no need of nesting
	$\exp$ or ${\sin}|_{[0,1]}$. The Euler number $\e$ is definable
(as $\exp(1)$), hence also $\e^\e$ (as $\exp(\e)$).
The number $\e$ is known to be an exponential period (e.g., $\int_0^1(\e^s+1)\md s$).
However,  we do not see an  obvious way to write $\e^\e$ as an exponential period.
	It would be very interesting to give a characterisation of the
        definable sets whose volumes are exponential periods.
\end{rem}

\subsection{The definition of Kontsevich and Zagier}
\label{contrast_with_KZ}

In \S4.3 of~\cite{kontsevich_zagier},
Kontsevich and Zagier give the following definition.
An \emph{exponential period in the sense of Kontsevich--Zagier}
is ``an absolutely convergent integral
of the product of an algebraic function
with the exponent of an algebraic function,
over a real semi-algebraic set,
where all polynomials entering the definition have algebraic coefficients''.
We take this to mean numbers of the form
\[ \int_G\e^{-f}\omega\]
where $G\subset\R^n$ is semi-algebraic over $\tilde{\Q}=\Qbar\cap \R$,
$f\in\Qbar(z_1,\dots,z_n)$, and $\omega$ a rational algebraic differential form
defined over $\Qbar$ such that the integral converges absolutely.
It is not clear to us if they want $\dim(G)=n$.
In this case, there is a prefered orientation from the orientation of $\R^n$,
in the general case we have to orient $G$.

We have shown that naive and generalised naive exponential periods over $\Qbar$
are absolutely convergent.
In particular, a generalised naive exponential period over~$\Qbar$
is an exponential period in the sense of Kontsevich--Zagier.

What about the converse?

\begin{ex}\label{ex:not_naiv}Let $G=[1,\infty)\subset\R$, $f=iz$, $\omega=\frac{1}{z^2}\md z$. Then
\[ \int_G\e^{-f}\omega=\int_1^\infty \frac{1}{t^2}\e^{-it}\md t=\int_1^\infty \frac{1}{t^2}\cos(t)\md t-i\int_1^\infty \frac{1}{t^2}\sin(t)\md t\]
converges absolutely because $\sin$ and $\cos$ are bounded by $1$.  However,
the data do not define a generalised naive exponential period. The interval
$G$ is \emph{not} a cycle for rapid decay homology of $(\A^1,\{1\})$.
We do
not have $\lim_{t\to \infty}\Re(f(t))\to \infty$ on $G$. In more detail: as we will see in Proposition~\ref{prop:fj_rd}, rapid decay $1$-cycles 
for  $(\A^1,f)$ are represented by $1$-chains
in $\Bcirc(\A^1,f)$ with boundary in $\partial \Bcirc(\A^1,f)$. In the present case
$\partial\Bcirc(\A^1,f)\subset\Ptilde$ is 
\[ \partial\Bcirc(\A^1,f)=\{s\infty\in\partial\Ptilde\mid \Re(f(s))>0\}=
\{s\infty\in\partial\Ptilde\mid \Im(s)<0\}\]
whereas the end point of $G$ is $1\infty$ with $\Im(s)=0$.

This example led us to conjecture that not all exponential periods in the sense of Kontsevich--Zagier are (generalised) naive exponential periods. However, Jossen~\cite{jossen-brief} communicated the following argument to us:  For $\delta\geq 0$ let $\gamma_\delta:[0,\infty)\to\C$ be given by $t\mapsto \delta t +it$ and consider
\[ I(\delta)=\int_{\gamma_\delta}e^{-f}\omega.\]
For $\delta=0$ we get back the integral considered above.
For $\delta>0$ it is a generalised naive exponential period. Moreover, its value is independent of $\delta$ (use Cauchy's theorem). Finally, absolute convergence of all integrals implies that $
\lim_{\delta\to 0}I(\delta)=I(0)$ and even $I(\delta)=I(0)$ for some/all $\delta>0$. This makes $I(0)$ a (generalised) naive exponential period.
\end{ex}
This leads us to expect:
\begin{conjecture}[Jossen]
All periods in the sense of Kontsevich--Zagier
are (generalised) naive exponential periods.
\end{conjecture}

We propose the following modification:

\begin{defn}\label{defn:abs}
An \emph{absolutely convergent exponential period} over $k$
is a complex number obtained as the value of an absolutely convergent integral of the form
\begin{equation}
  \label{eqn:absconexpper}
 \int_G\e^{-f}\omega 
\end{equation}
where $G\subset\C^n$ is a pseudo-oriented (not necessarily closed) $k_0$-semi-algebraic subset,
$\omega$ is a rational algebraic differential form on
$\A_k^n$  that is regular on $G$, $f$ a rational function on $\A^n_k$ regular on $G$ and the closure of $f(G)$ in $\Ptilde$ is contained in $\Bcirc$.

We denote by  $\Pabs(k)$ the set of all absolutely convergent exponential periods over $k$.
\end{defn}
\begin{rem}
The regularity condition for $f$ and $\omega$ on $G$ is harmless. We may replace
$G$ by the open subset $G'$ of points in which $f$ and $\omega$ are regular.
The value of the integral only changes if $\dim(G\ssm G')=\dim(G)$, i.e., if there is an open
$U\subset G$ on which $f$ or $\omega$ are have poles. The integral $\int_U\e^{-f}\omega$
does not make sense in this case, so we %
definitely want to exclude it.
Note that the condition on $f(G)$ excludes \cref{ex:not_naiv}
where we have $\overline{f(G)}=[i,i\infty]$ and $i\infty\notin \Bcirc$.
\end{rem}

We are going to show that every absolutely convergent exponential period is a generalised naive exponential period. Also for later use, let us be more precise.

\begin{proposition}\label{prop:geometry}
 Let $\alpha$ be an absolutely convergent exponential period over
 $k\subset\R$ with domain of integration as in
 \eqref{eqn:absconexpper} of dimension $d$.
 Then there are:
 \begin{itemize}
  \item a smooth affine variety $X$ over $k$ of dimension $d$,
  \item a simple normal crossings divisor $Y\subset X$,
\item
    a closed $k$-semi-algebraic subset $G\subset X(\R)$ of dimension $d$
   such that $\partial G = G \ssm G^\interior$ is contained in~$Y$,
\item a pseudo-orientation on $G$,
 \item a morphism $f \colon X_{k(i)} \to \A^1_{k(i)}$ such that $f|_G \colon G \to\C$
   is proper and such that the closure
   $\overline{f(G)}\subset\Ptilde$ is contained in $\Bcirc$,
  \item  a regular algebraic $d$-form $\omega$ on $X_{k(i)}$,
 \end{itemize}
 such that
 \[ \alpha=\int_{G}\e^{-f}\omega.\] %
\end{proposition}

\begin{proof}
 We start with a presentation
 \[ \alpha=\int_G\e^{-f}\omega\]
 with $G$ of dimension $d$, $f$ and $\omega$
 as in the definition of an absolutely convergent exponential period
 and modify the data without changing the value. This means
$G\subset\C^n$ is a pseudo-oriented $k(i)$-semi-algebraic subset (not necessarily closed), $\omega$ a rational algebraic differential form on $\A^n_{k(i)}$ regular on $G$, and $f$ a regular function on $\A^n_{k(i)}$ regular on $G$ with closure of
$f(G)$ in $\Ptilde$ contained in $\Bcirc$.
 
With the same trick as in Lemma~\ref{lem:real_version}, we may
 assume that $G\subset\R^n=\A^n_k(\R)$ is $k$-semi-algebraic with
 $f,\omega$ algebraic over $k(i)$.

 Let $X_0\subset\Pe_k^{n}$ be the Zariski-closure of $G$.
 It is an algebraic variety defined over $k$ of dimension~$d$, see the
 characterisation of dimension in \cite[Definition~2.8.1]{BCR}.
 Moreover, $\dim X_0(\R) =d$ as a real algebraic set.
 By assumption, $f$ is a rational map on $X_{0,k(i)}$. After replacing
 $X_0$ by a blow-up
 centered in the smallest subvariety of $X_0$ defined over $k$ containing the locus of indeterminancy of $f$,
 it extends to a morphism $f_0 \colon X_{0,k(i)}\to\Pe^1_{k(i)}$.
  By construction, $G\subset X_0(\R)$.
 Let $G_0:=\bar{G}\subset X_0^\an$ be the closure.
 It is contained in $X_0(\R)$ and compact because so is $X_0^\an$.
 It inherits a pseudo-orientation from $G$.
 Let $Y_0 \subset X_0$ be the union of~$X_{0,\sing}$
 and the Zariski closure of~$\partial G_0$,
 where the boundary is taken inside $X_0(\R)$.
 It has dimension less than $d$.

As the next step, let $\pi \colon X_1\to X_0$ be a resolution of singularities
at which the preimage $Y_1$ of $Y_0$ is  a divisor with normal crossings.
The map $\pi$ is an isomorphism outside $Y_0$.
As $Y_0\subset X_0$ has codimension at least $1$,
the intersection $G_0 \cap  Y_0(\R)$ has real codimension at least~$1$ in~$G_0$.
Let $G_1$ be the ``strict transform'' of $G_0$ in $X_1^\an$,
i.e., the closure of the preimage of $U=G_0\ssm (G_0\cap Y_0(\R))$.
By construction $\partial G_1\subset G_1\ohne \pi^{-1}(U)\subset Y_1$.
Let $\omega_1=\pi^*\omega$. Let $f_1=f\circ\pi$. By \cref{lem:pseudo}~(\ref{it:modif}), the set $G_1$ inherits a pseudo-orientation.
Moreover,
\[ \int_{G_1}\e^{-f_1}\omega_1=\int_G\e^{-f}\omega,\]
where the left hand side converges absolutely because the right hand
side does.

We claim that after further blow-ups, we can reach $\pi_2:X_2\to X_1$  preserving the properties of $X_1$, $Y_1$, and $G_1$ such that, in addition, points of $G_2$ in the polar locus of $\omega_2$ are contained in the polar locus of
$f_2= f_1\circ \pi_2$.

We first prove the claim.
Let $X_{1,\infty}$ be the polar locus of $f$ and $X_{1,\omega}$ the polar locus of $\omega$, i.e., the smallest closed subvarieties over $k$ such that
their base change to $k(i)$ contains the poles of $f$ and $\omega$, respectively. Note that $G_1$ is disjoint from $X_{1,\infty}$ because $f_1$ is regular on
$G_1$ and $G_1$ is contained in the real points of $X_1$.

Let $x\in G_1$ be a point at which $f_1$ is regular, but $\omega_1$ has a pole.
Let $U_1$ be a small compact neighbourhood of $x$ in $G_1$ in which $f_1$ is regular. By assumption,
\[\int_{U_1}\e^{-f_1}\omega_1\]
converges absolutely.
As $f_1$ is regular on $U_1$, the factor $\e^{-f_1}$ and its inverse are bounded.
Hence the absolute convergence of the integral is equivalent to the absolute convergence of the integral
\[ \int_{U'_1}\omega_1.\]
This case already shows up in the case of ordinary periods, see the proof of
\cite[Lemma~12.2.4]{period-buch}. The following argument is due to Belkale and Brosnan
in \cite{belkale-brosnan}: After a blow-up
$X_2\to X_1$ we find holomorphic coordinates
such that the pull-back $\omega_2$ of $\omega_1$ has the shape
\[ \text{unit}\times \prod_{j=1}^n z_j^{e_j}\md z_1\wedge\dots\wedge
  \md z_n\]
with $e_j\in\Z$ and such that $G_2$ contains a full coordinate quadrant.
Absolute convergence is only possible if $e_j\geq 0$ for all $j$, i.e., if
$\omega_2$ is regular on $U_2$.
This finishes the proof of the claim.

Let $X$ be the complement of the polar
loci of $f_2$ and $\omega_2$, $Y=X\cap Y_2$,
$f$ and $\omega$ the restrictions of $f_2$ and $\omega_2$ to $X$,
and $G=X^\an \cap G_2$.
The map $f_2\colon G_2\to (\Pe^1)^\an$ is proper, and hence so is $f\colon G\to\C$. The data satisfy all properties stated in the proposition, 
with the exception that $X$ is only quasi-projective rather than affine. 
We have $X\subset\Pe^N_k$ for some $N$.
Let $H$ be the hypersurface defined by the equation $X_0^2+\dots+X_N^2=0$.
Then $\Pe^N_k\ohne H$ is affine.

Note that $G\cap H^\an=\emptyset$ because $G\subset \Pe^N(\R)$ and $H(\R)=\emptyset$.
Hence we may replace $X$ by $X\ohne (X\cap H)\subset \Pe^N\ssm H$, making it quasi-affine. This means that it is open in a closed affine subvariety $X'\subset \Pe^N\ssm H$. The complement $X'\ssm X$ is closed, hence of
the form $V(s_1,\dots,s_m)$ for finitely many $s_i\in\Oh(X')$.
Let $H'=V(s_1^2+\dots +s_m^2)$. Then $H'(\R)\subset V(s_1,\dots,s_m)(\R)$ is disjoint of $X(\R)$, so we have $X(\R)=(X\ohne H')(\R)$.
Hence we may replace $X$ by its open subset $X'\ohne H'$, making it affine as a  the complement of a hypersurface in an affine variety.
\end{proof}

\begin{cor}\label{cor:abs=naive}
  The set of absolutely convergent exponential
  period equals the set of generalised naive exponential periods:
  \[ \Pgnaive(k)=\Pabs(k).\]
\end{cor}
\begin{proof}
 By Corollary~\ref{cor:nv_converges},
 every generalised naive exponential period
 is an absolutely convergent exponential period.

 Let $\alpha$ be an absolutely convergent exponential period.
 By the same argument as for naive exponential periods (see Lemma~\ref{lem:change_field}),
 we may replace $k$ by $k\cap\R$.
 We apply Proposition~\ref{prop:geometry}. Let $X$, $G$, etc. be as specified there. In particular, $X$ is affine and we fix an embedding $X\subset\A^n$. The morphism $f \colon X_{k(i)}\to\A^1_{k(i)}$
 extends to a rational morphism $\A^n_{k(i)}\to \A^1_{k(i)}$.
 The differential form~$\omega$ on~$X$
 extends to a rational differential form on $\A^n_{k(i)}$.
 This data satisfy the assumptions
 of the definition of a generalised naive exponential period.
\end{proof}

\begin{rem} In the definition of absolutely convergent exponential periods, we may restrict to $G\subset\R^n$, by the analogue of Lemma~\ref{lem:real_version}. 
 However, it is not clear to us if we may additionally require
$n=\dim(G)$.
 We tend to expect that it fails to be true.

 The analogous statement for ordinary periods holds true
 because they turn out to be volumes of bounded semi-algebraic sets
 (see \cite[Section~12.2]{period-buch}, also \cite{viu-sos}).
 We have replaced this by our Theorem~\ref{thm:omin_volume}.
 Close inspection of the proof only shows that every naive
 exponential period
 (and hence by \xref{thm:compare_all}
 also all absolutely convergent exponential periods)
 can be written as a $\Z$-linear combination of numbers of the form
 \[ \int_G \e^{-f}\md x_1 \wedge \dots  \wedge \md x_n\]
 for $G\subset\R^n$ of dimension $n$,
 $f \colon G\to\C$ continuous with \emph{semi-algebraic} real and imaginary part.
\end{rem}

\begin{rem}We pick up again on \cref{ex:not_naiv}.
As explained previously, the integral $\int_1^\infty \e^{-it}\frac{\md
  t}{t^2}$ converges absolutely, but does not obviously define a generalised naive period.
We concentrate on the real part. Integration by parts gives
\begin{align*}
 \int_1^\infty\frac{\cos(t)}{t^2}\md t&=\cos(1)-\int_1^\infty\frac{\sin(t)}{t}\md t\\
     &=\cos(1)- \frac{\pi}{2}+\int_0^1\frac{\sin(t)}{t}\md t
\end{align*}
because of the classical identity
$\int_0^\infty\frac{\sin(t)}{t}\md t=\frac{\pi}{2}$.
Note that the function $\frac{\sin(t)}{t}$ is entire, so there are no convergence issues with the last integral.
The numbers $\cos(1)$ and $\pi/2$ are definable
in the o-minimal structure $\R_{\sin,\exp}$ of \cref{defn:Rexpsin}.
The same is true for the function $\frac{\sin(t)}{t}$.
Hence we have written our number as the volume of a set that is definable in $\R_{\sin,\exp}$.
Note, however, that the formula does not give a presentation as an absolutely convergent exponential period. We have
\[ \int_0^1\frac{\sin(t)}{t}\md t=\int_0^1\Im\left(\frac{\e^{it}}{t}\right)\md t,\]
but the real part does not
converge for the choice  $G=(0,1)$, $f=iz$, and $\omega=\frac{\md z}{z}$.
\end{rem}

\section{Review of cohomological exponential periods}\label{sec:exp}

Throughout this section let $k\subset\C$ be a subfield. All varieties are defined over $k$.

We give the definition of exponential periods following
Fres\'an and Jossen in \cite{fresan-jossen} concentrating on the smooth affine case at the moment. The general case will be described in detail in Section~\ref{sec:exp_gen}.

\subsection{Rapid decay homology}\label{ssec:rd}

We follow \cite[1.1.1]{fresan-jossen} 
and set
\[ S_r=\{ z\in\C \mid \Re(z)\geq r\}\]
for $r\in\R$.

\begin{defn}\label{defn:rd_original}
Let $X$ be a complex algebraic variety, $Y\subset X$ a closed subvariety, $f\in\Oh(X)$.
The \emph{rapid decay homology} of $(X,Y,f)$ is defined as
\[ H_n^\rd(X,Y,f)=\lim_{r\to\infty}H_n(X^\an,Y^\an\cup f^{-1}(S_r);\Q).\]
\end{defn}

For $r'\geq r$, there is a projection map on relative homology, so
this really is a projective limit.
A direct limit construction using singular cohomology yields 
rapid decay cohomology $H^n_\rd(X,Y,f)$. It is dual to rapid decay homology.
By \cite[3.1.2]{fresan-jossen}, these limits stabilise, so it suffices to work with a single, big enough $r$. Indeed:

\begin{theorem}[Verdier~{\cite[Corollaire~5.1]{Verdier:76}}]\label{thm:verdier}
  There is a finite set $\Sigma\subset \C$ such that
$f|_{f^{-1}(\C\ssm\Sigma)} \colon f^{-1}(\C\ssm\Sigma)\rightarrow\C\ssm\Sigma$
is a  fibre bundle.
\end{theorem}

As $S_r$ is contractible,
this implies that all $f^{-1}(S_r)$ with $r$ sufficiently large
are homotopy equivalent to a  fibre of $f$.

There is an alternative description of $H_n^\rd(X,f)$ which is better
suited to the computation of periods. It is originally due to Hien and
Roucairol, see \cite{hien-roucairol}. We follow the presentation of
Fres\'an and Jossen in \cite[Section~3.5]{fresan-jossen}.

We fix a smooth variety $X$ and $f\in\Oh(X)$.
Let $\bar{X}$ be a good compactification relative to $f$,
i.e., such that $\bar{X}$ is smooth projective,
$X_\infty=\bar{X}\ssm X$ is a divisor with normal crossing
and $f$ extends to $\bar{f}\colon\bar{X}\to\Pe^1$.
We decompose $X_\infty=D_0\cup D_\infty$ into simple normal crossings divisors
such that $\bar{f}(D_\infty)=\{\infty\}$
and $\bar{f} \colon D_0\to \Pe^1$ is dominant on all components,
i.e., into vertical and horizontal components.

\begin{defn}\label{defn:our_Bcirc}
 We denote by $\pi \colon B_{\bar X}(X) \to \bar X^\an$
 the real oriented blow-up $\OBl_{X_\infty}(\bar{X})$, see \Cref{defn:orb}.
 Let $\tilde{f}:B_{\bar{X}}(X)\to\Ptilde$ be the induced map,
 see Lemma~\ref{lem:blow-up_functorial}.
 We also define
 \begin{align*}
  B_{\bar{X}}^\circ(X,f)
  &= B_{\bar{X}}(X)\ssm \left(\pi^{-1}(D_0^\an)\cup \tilde{f}^{-1}(\{s\infty\in\Ptilde\mid \Re(s)\leq 0\})\right),\\
  \partial \Bcirc_{\bar{X}}(X,f)
  &= \Bcirc_{\bar{X}}(X,f)\ssm X^\an
  = \Bcirc_{\bar{X}}(X,f)\cap \tilde{f}^{-1}(\{s\infty\in\Ptilde\mid \Re(s)>0\}).
 \end{align*}
\end{defn}
We are going to omit the subscript $\bar{X}$ as long as it does not cause confusion.

At this point we only consider them as topological spaces.
In fact $\Bcirc_{\bar{X}}(X,f)$ is
a semi-algebraic $C^\infty$-manifold with corners.

\begin{proposition}[{\cite[Proposition~3.5.2]{fresan-jossen}}]\label{prop:fj_rd}
Let $X$ be a smooth variety over $k$.
For sufficiently large~$r$,
the inclusion induces natural isomorphisms
 \begin{align*}
 H_n(X^\an,f^{-1}(S_r);\Q)&\isom H_n(B(X),\tilde{f}^{-1}(S_r);\Q) \\
                         &\isom H_n(\Bcirc(X,f),\partial \Bcirc(X,f);\Q).
\end{align*}
In particular,
\[ H_n^\rd(X,f)\isom H_n(\Bcirc(X,f),\partial \Bcirc(X,f);\Q).\]
\end{proposition}

\begin{rem}\label{rem:comment}
We will extend this (and more) to relative rapid decay homology
\[ H_n^\rd(X,Y,f)\isom H_n(\Bcirc(X,f), Y^\an\cup \partial \Bcirc(X,f);\Q)\]
in Section~\ref{sec:relative}. This is particularly easy to see if $Y$ is disjoint from $\bar{X}\ohne X$ by comparing the long exact sequences for relative homology.
\end{rem}

Recall from Definition~\ref{defn:int_simplex} that $S_*(M)$ denotes the complex of $\Q$-linear combinations of
$C^1$\nobreakdash-simplices for a $C^p$-manifold with corners $M$ or a closed subset thereof.
It computes singular cohomology by Theorem~\ref{thm:compare}.

\begin{defn}\label{defn:S^rd}
 Let $X$ be a smooth variety, $f\in\Oh(X)$.
	Choose a good compactification $\bar{X}$ relative to $f$.
 We put
 \[ S_*^\rd(X,f)=S_*(\Bcirc_{\bar{X}}(X,f))/S_*(\partial \Bcirc_{\bar{X}}(X,f)).  \]
\end{defn}

\begin{rem}\label{rem:fj}
Fres\'an and Jossen work with piecewise $C^\infty$-simplices instead,
see \cite[Section~7.2.4]{fresan-jossen}.
We opt for the slightly more complicated notion of $C^1$-simplices as opposed to $C^\infty$-simplices
because they are well-suited for working with our semi-algebraic sets.
\end{rem}

\subsection{Twisted de Rham cohomology: the smooth case}\label{ssec:twist}
Let $X/k$ be a smooth variety, $f\in\Oh(X)$. We define a vector bundle
with connection $\Eh^f=(\Oh_X,d_f)$ with $d_f(1)=-df$. The de Rham complex $\DR(\Eh^f)$ has
the same entries as the standard de Rham complex for $X$, but with differential
$\Omega^p\to\Omega^{p+1}$ given by $d\omega-df\wedge \omega$.

\begin{defn}Let $(X,f)$ be as above. We define \emph{algebraic de Rham cohmology} $H^*_\dR(X,f)$ of $(X,f)$ as the hypercohomology of $\DR(\Eh^f)$.
\end{defn}
If $X$ is affine, this is nothing but the cohomology of the complex
\[ \RdR(X):=[\Oh(X)\xrightarrow{d_f}\Omega^1(X)\xrightarrow{d_f}\dots].\]
The definition needs to be extended to the relative cohomology of singular varieties. We first consider a special case.
Let $X$ be smooth and $Y\subset X$ a divisor with simple normal crossings.
Let $Y_\bullet\to Y$ be the \v{C}ech-nerve of the cover
of $Y$ by the disjoint union of its irreducible components, see Section~\ref{ssec:hypercover} for more details.
It is a smooth proper hypercover.
In particular, $H_n(Y_\bullet^\an,\Z)=H_n(Y^\an,\Z)$.

\begin{defn}\label{defn:dR-snc}
Let $X$ be a smooth variety, $Y\subset X$ a divisor with simple normal crossings. We define
\emph{algebraic de Rham cohomology} $H^*_\dR(X,Y,f)$ of $(X,Y,f)$ as the hypercohomology
of the complex of sheaves on $X$ 
\[ \cone\left(\pi_*\DR(\Eh^f|_{Y_\bullet})\to \DR(\Eh^f)\right)[-1]\]
where $\pi:Y_\bullet\to Y\to X$ is the natural map.
\end{defn}

\subsection{The period isomorphism}\label{ssec:pairing_smooth}

 Hien and Roucairol established the existence of a canonical isomorphism
 \[  H^n_{\rd}(X,f)\tensor_\Q\C\to H^n_\dR(X,f)\tensor_k\C
 \]
 for smooth affine varieties $X$,
see \cite[Theorem~2.7]{hien-roucairol}. 
It is also explained and extended to the relative case for any variety $X$ and closed subvariety $Y$ by Fres\'an and Jossen, see
\cite[Theorem~7.6.1]{fresan-jossen}. We refer to it as the  \emph{period isomorphism}. It induces a \emph{period pairing}
\begin{equation}\label{eq:pairing} \langle-,-\rangle \colon H^n_\dR(X,Y,f)\times H_n^\rd(X,Y,f)\to\C.\end{equation}
We will go through the details in Section~\ref{sec:exp_gen}.
\begin{defn}\label{defn:coh_period}
Let $X$ be a variety, $f\in\Oh(X)$ a regular function, $Y\subset X$ a closed subvariety, $n\in\Na_0$. The elements in the image of the period pairing (\ref{eq:pairing}) are
called the \emph{(cohomological) exponential periods} of $(X,Y,f,n)$.

We denote by $\Pcoh(k)$ the set of cohomological exponential periods
for varying $(X,Y,f,n)$ over $k$.
We denote by $\Plog(k)$ the subset of cohomological exponential periods
for varying $(X,Y,f,n)$ such that $(X,Y)$ is a log-pair.
\end{defn}

We record for later use: 
\begin{lemma}\label{lem:change_field_coh}
 Let $K/k$ be an algebraic extension. Then
 \[\Pcoh(K)=\Pcoh(k).\]
\end{lemma}
\begin{proof}
 The same argument as in the classical case, 
 \cite[Corollary~11.3.5]{period-buch},
also applies in the exponential case.
\end{proof}

The construction of the period map is non-trivial.
At this point, we only need its explicit description in a special case.
See Definition~\ref{defn:relative_deRham} for the general case.

\begin{defn}\label{defn:period_pairing_smooth}
Let $X$ be smooth affine. We define a pairing
\[ \Omega^n(X)\times S_n^\rd(X^\an,f)\to\C\]
by mapping $(\omega, \sigma)$ to
\[ \int_{\sigma}\e^{-f}\omega^\an.\]
\end{defn}

\begin{lemma}\label{lem:period_map_smooth}The pairing is well-defined and induces a morphism of
complexes
\[ \Omega^*(X,f)\to \Hom(S_*^\rd(X^\an,f),\C).\]
On cohomology it induces the pairing (\ref{eq:pairing}).
\end{lemma}
\begin{proof}
Let $\omega\in\Omega^n(X)$ and $\sigma$ an $n$-dimensional $C^1$-simplex in
$S_n^\rd(X^\an,f)$.
The smooth form
$\omega^\an$ on $X^\an$ defines a smooth form $\e^{-f}\omega^\an$ on $\Bcirc(X,f)$.
(Note that $\e^{-f}\omega^\an$ vanishes to any order on $\partial \Bcirc(X,f)$.
So it can be extended by $0$ to a neighbourhood of the boundary).
Hence the integral is well-defined.

The compatibility with the boundary map translates as
\[ \int_{\sigma}\e^{-f}d_f\omega^\an=\int_{\partial \sigma}\e^{-f}\omega^\an\]
which holds by Stokes's formula (see Theorem~\ref{thm:stokes})
because $d_f\omega = d\omega - df \wedge \omega$.

The construction is the one of \cite[Chapter~7.2.7]{fresan-jossen}, only with
our $S_*(X)$ instead of their complex, see Remark~\ref{rem:fj}.
\end{proof}

By taking double complexes, this extends to general $X$ and $Y$.
We will discuss this in detail in \xref{sec:exp_gen}.
At this point, we handle the simplest case.

\begin{ex}\label{ex:compute_period}
Let $X$ be a smooth affine variety,
$Y\subset X$ a smooth closed subvariety, $f\in\Oh(X)$.
Then relative twisted de Rham cohomology
is computed by the complex
\begin{align*}
\RdR(X,Y,f)&=\cone\left(\Omega^*(X)\to\Omega^*(Y)\right)[-1]\\
 &=\left[\Omega^0(X)\to \Omega^1(X)\oplus\Omega^0(Y)\to \Omega^2(X)\oplus\Omega^1(Y)\to\dots\right]
\end{align*}
with differential induced by $d_f$ and restriction.
Its rapid decay homology is computed by the
complex 
\[ S_*^\rd(X,Y,f)=\cone\left(S_*^\rd(Y,f)\to S_*^\rd(X,f)\right).\]
Explicitly:
let $\bar{X}$ be a good compactification of $X$ relative to $f$ in the sense of
Section~\ref{ssec:gc}
such that the closure $\bar{Y}$ of~$Y$ in~$\bar{X}$
is a good compactification as well.
Then
\begin{multline*}
 \cone(S^\rd_*(Y,f)\to S^\rd_*(X,f)) =\\
 [S^\rd_0(X,f)\leftarrow S^\rd_1(X,f)\oplus S^\rd_0(Y,f)\leftarrow S^\rd_2(X,f)\oplus S^\rd_1(Y,f)\leftarrow\dots].
\end{multline*}
Let $\sigma$ be a cycle in $S_n^\rd(X,Y,f)$, i.e., a chain $\sigma_X$ on
$X$ such that $\partial \sigma_X=\sigma_Y$ is supported on $Y$. In the second incarnation, we identify it with 
\[ (\sigma_X,-\sigma_Y)\in S_n^\rd(X,f)\oplus S_{n-1}^\rd(X,f).\]
Let $\omega$ be a cocycle in $\RdR(X,Y,f)$, i.e., a pair of differential
forms 
\[ (\omega_X,\omega_Y)\in\Omega^n(X)\oplus\Omega^{n-1}(Y)\]
 such that
$d\omega_X=0$ and $d\omega_Y=\omega_X|_Y$. Their period is
\[\langle [\omega],[\sigma]\rangle=\int_{\sigma_X}\omega_X-\int_{\sigma_Y}\omega_Y.\]
\end{ex}

\section{Triangulations}\label{sec:triangle}

We fix a real closed field $\tilde{k}\subset \R$ and work with semi-algebraic sets of $\R^N$ defined over $\tilde{k}$. We expect that everything holds in general for o-minimal structures, but we do not need this for our application.
We use the set-up of \cite[Chapter~8]{D:oMin}
for complexes. It is not completely standard, but very convenient for us.

Let $n\in\IN_0$. Let $a_0,\dots,a_n\in \tilde{k}^N$ be affinely independent, i.e,  the vectors $a_1-a_0,\dots,a_n-a_0$ are linearly independent.
The \emph{open $n$-simplex} defined by these vectors is the set
\begin{multline*}
\sigma = (a_0,\ldots,a_n)\\
=\left\{ \sum_{i=0}^n \lambda_i a_i \in\R^N:  \text{for all $i$ we have
$\lambda_i>0$ and }\lambda_0+\cdots + \lambda_n =
1\right\}.
\end{multline*}
 We fix the orientation given by $d\lambda_1\wedge\dots \wedge d\lambda_n$.
The closure of $\sigma$ is denoted by $[a_0,\ldots,a_n]$ and
obtained by relaxing to $\lambda_i\ge 0$ in the definition above.
We call $[a_0,\ldots,a_n]$ a \emph{closed $n$-simplex}.
The points $a_0,\ldots,a_n$ are uniquely determined by
$[a_0,\ldots,a_n]$ and thus by $\sigma$.
As usual, a face of $\sigma$ is a simplex spanned by a non-empty
subset of $\{a_0,\ldots,a_n\}$.
Then $[a_0,\ldots,a_n]$ is a disjoint union of faces of $\sigma$.
We write $\tau < \sigma$ if $\tau$ is a face of $\sigma$ and
$\tau\not=\sigma$.

A finite set $K$ of simplices in $\R^N$ is called a \emph{complex} if
for all $\sigma_1,\sigma_2\in K$ the intersection
$\overline{\sigma_1}\cap\overline{\sigma_2}$ is either empty or the
closure of a common face $\tau$ of $\sigma_1$ and $\sigma_2$.
Van den Dries's definition does not ask for $\tau$ to lie in $K$. So
the \emph{polyhedron} spanned by $K$
$$|K|  = \bigcup_{\sigma \in K} \sigma$$
may not be a closed subset of $\R^N$.
We call $K$ a closed complex if $|K|$ is closed or equivalently,
if for all $\sigma\in K$ and all faces $\tau$ of $\sigma$, we have $\tau\in K$.
Note that $\bigcup_{\sigma\in K}\sigma$ is a disjoint union; this is
an advantage of working with ``open'' simplices.
We write $\overline K$ for the complex obtained by taking all faces of
all simplices in $K$. Note that $\overline K$ is $\tilde{k}$-semi-algebraic.

\begin{defn}\label{defn:triangulation}
Let $M$ be a $\tilde{k}$-semi-algebraic $C^1$-manifold with corners,  $A\subset M$ be a $\tilde{k}$-semi-algebraic subset. A \emph{semi-algebraic triangulation} of $A$ is a pair
$(h,K)$ where $K$ is a complex and where
$h:|K|\rightarrow A$ is a $\tilde{k}$-semi-algebraic  homeomorphism.
We say that it is \emph{globally of class $C^1$}
if $h$ extends to a $C^1$-map on an open neighbourhood of $|K|$.

Let $B\subset A$ be a $\tilde{k}$-semi-algebraic subset. We say that $(h,K)$ is compatible
with $B$ if $h(B)$ is a union of elements of~$K$.
\end{defn}

\begin{rem}
 Note that there are other definitions
 of $C^p$-triangulations (for $p\geq 1$) in the literature, where regularity is required only on the \emph{interior} of all faces,
 see for example
 \cite[Remark~9.2.3~(a)]{BCR} or \cite[Chapter~II]{shiota-book}. These are not enough to deduce Stokes's
formula in the semi-algebraic setting, as needed in order to get a well-defined period pairing.
Our solution is the following $C^1$-triangulation theorem by
Czapla--Paw\l{}ucki. Later on we will extend it from semi-algebraic
subsets of $\R^d$ to semi-algebraic manifolds with corners. 
\end{rem}

\begin{proposition}[{Czapla--Paw\l{}ucki, \cite[Main Theorem]{omin-triang}}]
  \label{prop:czaplapawlucki}
    Let $X$ be a compact  $\tilde k$-semi-algebraic
    subset of $\R^d$ and let $A_1,\ldots,A_M$ be $\tilde
    k$-semi-algebraic subsets of $X$. Then there exists a
    $\tilde{k}$-semi-algebraic triangulation of $X$ that is globally of class
    $C^1$ and compatible with each $A_j$.
\end{proposition}
  
\begin{rem}\label{rem:correction}

Ohmoto and Shiota formulate the above
in the locally semi-algebraic setting, see \cite{ohmoto-shiota-triangulation}. (Note that in their convention, ``semi-algebraic'' is used as a short-hand for ``locally semi-algebraic''.) As pointed out by 
Brackenhofer in \cite{brackenhofer}, their proof has a gap. In their
Lemma~3.4 they claim that a certain map is (locally) semi-algebraic.
In particular, it has to be semi-algebraic on closed simplices. No
argument is given, and indeed, the construction of the map using a
locally semi-algebraic partition of unity would not produce a
semi-algebraic map in general.

Meanwhile, Paw\l ucki~\cite{Pawlucki:Cp} proved a $C^p$-triangulation theorem in the
o-minimal setting for all $p\ge 0$. 
\end{rem}

\subsection{Existence of triangulations}

Our aim is to triangulate %
semi-algebraic manifolds with corners, see \cref{defn:manifold}.

\begin{lemma}\label{lem:functions}Let $X$ be a compact $d$-dimensional
  $\tilde{k}$-semi-algebraic $C^1$-manifold with corners. Then there exist 
$\tilde{k}$-semi-algebraic maps 
\[ g_1,\dots,g_m:X\to \R^d\]
of class $C^1$ with the following property:
Let $Y$ be another  $\tilde{k}$-semi-algebraic $C^1$-manifold with corners and $h:Y\to X$ a continuous $\tilde{k}$-semi-algebraic map. Then $h$ is $C^1$ if and only if all $g_j\circ h$ are $C^1$.
\end{lemma}
\begin{proof}%
 Pick an atlas
  $(\phi_i:U_i\to V_i|i=1,\dots,N)$ of $X$. We claim that there are
  $\tilde{k}$-semi-algebraic functions of class $C^1$
  \[ f_1,\dots,f_m:X\to [0,1]\] such that for every $j$ there is an $i(j)$ such that
  \begin{itemize}
  \item the support of $f_j$ is contained in $U_{i(j)}$,
  \item there is an open subset $W_j\subset U_{i(j)}$ on which $f_j$ is
    identically $1$,
  \end{itemize} and, moreover, the $W_j$ are a cover of $X$.

For each $P\in X$ we fix a chart $U_i$ containing $P$.
  Each $V_i$ is an open subset of some $\R^{n_i}\times\R_{\ge
    0}^{m_i}$. There is an open ball $B$ in $\R^{n_i+m_i}$ centered at
  $\phi_i(P)$
  of radius $r>0$
  such that
  $\phi_i(P) \in B\cap  (\R^{n_i}\times\R_{\ge 0}^{m_i})
  \subset V_i$.
  Let $f_P'\colon \R^{n_i+m_i}\rightarrow [0,1]$ be a $\tilde{k}$-semi
  algebraic $C^1$-function that
  is identically $1$ on the open ball of radius $r/2$ centered at
  $\phi_i(P)$ and with support contained completely in $B$. 
  We denote by $W_P$ the preimage in $U_i$ of the said ball of radius
  $r/2$ and by $f_P$ the composition $f_P'\circ \phi_i$ extended by
  zero on $X\ohne U_i$. As $X$ is compact, there are finitely many
  $P_1,\ldots,P_m$ such that $W_{P_1}\cup\cdots\cup W_{P_m}=X$.
  The claim follows with  $W_j = W_{P_{j}}$ and $f_j = f_{P_j}$. 

Now let
\[ g_j=f_j\phi_{i(j)}:X\to \R^d\]
be the product  and consider $h:Y\to X$ as in the statement. If $h$ is $C^1$, then so are all compositions $g_j\circ h$. Conversely, assume that all $g_j\circ h$ are $C^1$.  As $W_1,\dots,W_m$ cover $X$, it suffices to check the
claim after restricting to the preimage of some $W_j$. 
By definition, a map is $C^1$ if its composition with all $\phi_{i(j)}$ is. This is the case
because $g_j=f_j\phi_{i(j)}=\phi_{i(j)}$ on $W_j$ and $g_j\circ h$ is $C^1$.
\end{proof}

\begin{proposition}\label{prop:triangle_manifold}
 Let $X$ be a compact $\tilde{k}$-semi-algebraic $C^1$-manifold with corners,
 $A_1,\dots,A_M$ semi-algebraic subsets of $X$.
 Then there is a $\tilde{k}$-semi-algebraic triangulation of $X$
 compatible with $A_1,\dots,A_M$
 that is globally of class~$C^1$.
\end{proposition}
\begin{proof}
	By \cite[Theorem~1]{Robs83},
	there exists a $\tilde{k}$-semi-algebraic set $X' \subset \R^n$
	and a $\tilde{k}$-semi-algebraic homeomorphism $\psi \colon X \to X'$.
	We stress that $\psi$ is not $C^1$ in general.

	Now let $g_1,\dots,g_m \colon X \to \R^d$ be as in \cref{lem:functions}.
	We consider the graph
	\[ \Gamma = \{(\psi(x), g_1(x), \dots, g_m(x)) \mid x \in X\} \subset \R^n \times (\R^d)^m. \]
	We denote by $\pi_0 \colon \Gamma \to X'$ the first projection,
	and $\pi_j \colon \Gamma \to \R^d$, for $j = 1,\dots,m$ the other projections.
	Let $\iota =(\psi,g_1,\dots,g_m)\colon X \to \Gamma$ be the canonical map. It is a homeomorphism with
	inverse $\psi^{-1} \circ \pi_0$.

	Let $\Phi \colon |K| \to \Gamma$ be a $\tilde{k}$-semi-algebraic triangulation globally of class $C^1$ and compatible with $\iota(A_1),\dots,\iota(A_M)$. Such a triangulation 
	exists by \xref{prop:czaplapawlucki}.
	Let $h \colon |K| \to X$ be the composition
        $\psi^{-1} \circ \pi_0 \circ \Phi$.  Hence the diagram
	\begin{center}
		\begin{tikzcd}
			X \ar[rr, shift left=1, "\iota"] && \Gamma \ar[ll, shift left=1, "\iota^{-1}=\psi^{-1} \circ \pi_0"] \\
			& {|K|} \ar[ur, "\Phi"'] \ar[ul, "h"]
		\end{tikzcd}
	\end{center}
commutes.
	Clearly, $h$ is a homeomorphism, as it is the composition of two homeomorphisms.
	We claim that it is also $C^1$ (but not necessarily a diffeomorphism).
	By \cref{lem:functions} it suffices to check that the compositions $g_j \circ h$ are $C^1$.
		We have the equalities
	\[ g_j \circ h = (\pi_j\circ\iota)\circ  (\iota^{-1}\circ\Phi)=
   \pi_j \circ \Phi, \]
 and the latter is a composition of $C^1$ maps.
 So $h$ is a $C^1$-triangulation of $X$. It is compatible
   with each $A_j$. We are done.
\end{proof}

The existence of these triangulation is used to relate cohomological exponential periods to naive exponential periods. In the curve case this happens in Step~3 of the proof of Proposition~\ref{curve-log-sub-nv}. In the general case it is the input for Proposition~\ref{prop:rd_simplicial}.

\subsection{A deformation retract}
We are going to show that, up to deformation,
a complex $K$ can be identified with a closed complex.
The arguments are similar to the ones in \cite[Chapter 8, (3.5)]{D:oMin}. %
Compare also with Friedrich's \cite[Proposition~2.6.9]{period-buch} and its proof, \cite[p~69]{period-buch}.

If $\sigma$ is a simplex in $\R^N$, then $b(\sigma)$ denotes its barycenter.
Let $K\subset \R^N$ be a complex.
We denote by $\beta(K)$ its barycentric subdivision as defined in
\cite[Chapter 8, (1.8)]{D:oMin}.
Note that $|K| = |\beta(K)|$.

We define the \emph{closed core} of a complex $K$ as
\begin{equation*}
 \cc{K} = \{ \sigma \in K \mid K \text{ contains all faces of $\sigma$}\}.
\end{equation*}
Then $\cc{K}$ is a subcomplex of $K$.
It is a closed complex by definition.
But it can be empty:
consider a complex consisting of a single simplex of positive dimension.
This problem is remedied by passing to the barycentric subdivision.
More precisely, if $K$ is non-empty, then $\cc{\beta(K)}$ is non-empty.
Indeed, the barycenter $b(\sigma)$ of $\sigma \in K$
defines a face $(b(\sigma))$ of $\beta(K)$, which lies in
$\cc{\beta(K)}$.

Finally, note that if $L$ is a subcomplex of $K$, then $\beta(L)\subset \beta(K)$ and
$\cc{L}\subset \cc{K}$, so we have $\cc{\beta(L)}\subset
\cc{\beta(K)}$.

\begin{proposition}
 \label{lem:cc}
 Let $K$ be a complex.
 There exists a  $\tilde k$-semi-algebraic
 retraction $r \colon |K|=|\beta(K)|\rightarrow |\cc{\beta(K)}|$ with the
 following properties.
 \begin{enumerate}
 \item [(i)] For each $x\in |K|$ the half open line segment
   $[x,r(x))$ is contained in the simplex of $\beta(K)$
   containing $x$.
  \item[(ii)] The map
    \begin{equation*}
      H(x,t) = (1-t)x + tr(x)
    \end{equation*}
    is a $\tilde k$-semi-algebraic strong deformation retraction
    $H \colon |K|\times[0,1]\rightarrow |K|$ onto $|\cc{\beta(K)}|$.
 \end{enumerate}
\end{proposition}
We use a variation of the arguments found in 
  \cite[Chapter 8, \S 3]{D:oMin}.
\begin{proof}
Let $b=b(\sigma)$ be a vertex of $\beta(K)$.
As in loc.\ cit.\ we define a continuous semi-algebraic function
\[ \lambda_\sigma \colon |K| = |\beta(K)| \to [0,1]\]
which vanishes on $|\tau|$ if $b$ is not a vertex of $\tau\in\beta(K)$ and
equals the barycentric coordinate with respect to $b$ if it is.

Let us define  furthermore
\[ \Lambda(x)= \sum_{\sigma \in  K} \lambda_\sigma(x).\]

We claim that $\Lambda(x)>0$ for all $x\in |K|$.
Indeed, $x$ is contained in a simplex
$(b(\sigma_0),\ldots,b(\sigma_n))$ of $\beta(K)$; here
$\sigma_0< \cdots < \sigma_n$ are open simplices of $\overline K$ and
$\sigma_n\in K$.
In particular,  $\lambda_{\sigma_n}(x)>0$.
Thus the contribution coming from $\sigma_n$ to the sum $\Lambda(x)$
is strictly positive. As all other contributions are non-negative we
find $\Lambda(x)>0$, as desired.

We are ready to define $r(x)$ for $x\in |K|$ as
$$
r(x) = \frac{\sum_{\sigma \in  K}
\lambda_\sigma(x) b(\sigma)}{\Lambda(x)}.
$$
Thus $r \colon |K| \to \R^N$ is $\tilde k$-semi-algebraic and continuous.

Let us verify that $r(|K|)\subset |\cc{\beta(K)}|$.
Say $x\in |K|$ and let $\sigma_0, \ldots,\sigma_n$ be as before.
Say $\sigma \in K$.
We recall that $\lambda_\sigma(x)>0$ if and only
if $\sigma$ is among $\{\sigma_0,\ldots,\sigma_n\}$.
Let $\sigma_{i_0}<\cdots < \sigma_{i_k}=\sigma_n$ be those among the
$\sigma_0,\ldots,\sigma_n$ that lie in $K$.
So $r(x) = \sum_{j=0}^k \alpha_j b(\sigma_{i_j})$
with coefficients $\alpha_j\in [0,1]$ such that
 $\sum_{j=0}^k \alpha_j=1$. Observe that $\alpha_j>0$ since $x\in
(b(\sigma_0),\ldots,b(\sigma_n))$. Thus $r(x) \in
(b(\sigma_{i_0}),\ldots,b(\sigma_{i_k}))$.
Finally,
$b(\sigma_{i_j})\in \sigma_{i_j}\in K$ for all $j$.
Therefore,  $\beta(K)$ contains all faces of the simplex
$(b(\sigma_{i_0}),\ldots,b(\sigma_{i_k}))$ which must  thus
be an element of $\cc{\beta(K)}$. We conclude $r(x)\in
|\cc{\beta(K)}|$. So the target of $r$ is $\cc{\beta(K)}$, as claimed.

Moreover,  
  $(b(\sigma_{i_0}),\ldots,b(\sigma_{i_k}))$ is a face of
$(b(\sigma_0),\dots,b(\sigma_n)) \in \beta(K)$, hence by convexity the ray
$[x,r(x))$  is
in the simplex of $\beta(K)$ containing $x$.

We now verify that $r$ is a retraction. We still assume $x\in
  (b(\sigma_0),\ldots,b(\sigma_n))$ as above.
 Note that $x = \sum_{\sigma\in\overline K}\lambda_\sigma(x) b(\sigma)$
and $\sum_{\sigma\in\overline K}\lambda_\sigma(x) =1$.
If $\lambda_\sigma(x)>0$ for some $\sigma\in \overline K$, then $\sigma$ is among
$\sigma_0,\ldots,\sigma_n$. Hence
\begin{equation*}
  \sum_{i=0}^n \lambda_{\sigma_i}(x) b(\sigma_i) = x
  \quad\text{and}\quad
  \sum_{i=0}^n \lambda_{\sigma_i}(x) =1.
\end{equation*}

Now suppose $x\in |\cc{\beta(K)}|$. By definition, $\beta(K)$ contains
all faces of $(b(\sigma_0),\ldots,b(\sigma_n))$. In particular,
$b(\sigma_i)\in |K|$ and hence $\sigma_i \in K$ for all $i$. So
$\Lambda(x)=1$ and $r(x)=x$. In particular, $r$ is a
  retraction.

For all $x \in (b(\sigma_0), \ldots, b(\sigma_n)) \in \beta(K)$
and all $t$,
 we find that
\[
  (1-t)x + tr(x) =  \sum_{i=0}^n \frac{\left((1-t)\Lambda(x)
  + t w_{\sigma_i}(x)\right)\lambda_{\sigma_i}(x)}{\Lambda(x)} b(\sigma_i),
\]
where $w_{\sigma_i}$ is constant $1$ if $\sigma_i\in K$ and constant
$0$ else. If $t\in [0,1)$ each factor in the sum on the right is strictly
positive and their sum is $1$. This implies $(1-t)x + tr(x) \in
(b(\sigma_0),\ldots,b(\sigma_n))$.
As we have seen before, $r(x) = x$ for $x \in |\cc{\beta(K)}|$.
Altogether, this proves claim~(ii).
\end{proof}

\section{The case of curves}
\label{sec:curves}

Let $k \subset \C$ be a subfield which is algebraic over $k_0 = k\cap\R$.
For simplicity, we assume that $k$ is algebraically closed.
In this section we will show that
naive exponential periods of the form $\int_G \e^{-f}\omega$
where $G$ is $1$-dimensional
are the same as cohomological exponential periods of smooth marked curves.
This comparison is a special case
of the general result in \xref{thm:compare_all},
but we include it to illustrate the key ideas of the general proof,
while avoiding several technical problems.

This section is organised as follows:
first we give some elementary examples of cohomological exponential periods
and explain why they are naive exponential periods.
This is followed by an intermezzo
in which we describe the oriented real blow-up of a marked curve,
because it features several times in the remainder of the section.
Finally, we prove the equality of the the two notions of periods.

\subsection{Examples of cohomological exponential periods}
In \cref{eg-naive} we saw explicit examples of naive exponential periods.
We will now look at some examples of cohomological exponential periods,
before considering the case for general curves.
\begin{ex}
 \label{eg-coh}
 We start with the simplest non-trivial case: $X=\A^1$, $Y=\{0\}$, $f=\id$.
 Then 
\[ H_1^\rd(\A^1,\{0\},\id)=
 H_1(\Bcirc(\A^1,\id),\{0\}\cup \partial \Bcirc(\A^1,\id);\Q).\]
 Both $\Bcirc=\Bcirc(\A^1,\id)$ and its boundary are contractible,
 hence $H_1^\rd(\A^1,\{0\},\id)$ is of dimension $1$.
 The generator is the path from $0$ to a point on $\partial \Bcirc$,
 i.e., one of the $G_s$ of Example~\ref{ex:3}.
 We use $G_1=[0,\infty)$ because it is in the subspace $\Bsharp$ as
 defined in \cref{sec:not_B}.
 \def\figrad{1}
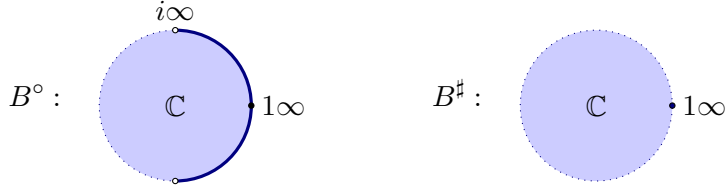
 \[
  \Bcirc:
  \quad
  \begin{tikzpicture}[baseline]
   \filldraw[fill=blue!20!white, draw=blue!50!black, dotted] (0,0) circle (\figrad);
   \draw[draw=blue!50!black, very thick] (-90:\figrad) arc (-90:90:\figrad);
   \node at (0,0) {$\C$};
   \draw[fill] (0:\figrad) circle (1pt) node[anchor=west] {$1\infty$};
   \draw[fill=white, draw=black, thin] (-90:\figrad) circle (1pt);
   \draw[fill=white, draw=black, thin] (90:\figrad) circle (1pt)
        node[anchor=south] {$i\infty$};
  \end{tikzpicture}
  \qquad
  \qquad
  \Bsharp:
  \quad
  \begin{tikzpicture}[baseline]
   \filldraw[fill=blue!20!white, draw=blue!50!black, dotted] (0,0) circle (\figrad);
   \node at (0,0) {$\C$};
   \draw[fill=blue!50!black] (0:\figrad) circle (1pt)
        node[anchor=west] {$1\infty$};
  \end{tikzpicture}
 \]
 The boundary in singular homology
 maps it to minus the class of the point $0$.

 The relative de Rham complex has the shape
 \[ k[z]\xrightarrow{P\mapsto (dP-P\md z,P(0))}k[z]\md z\oplus k .\]
 As in \Cref{ex:compute_period}, the periods of $(Q\md z,a)$ are computed as
 \[ \int_{G_1} \e^{-z}Q\md z-a.\]
 The general theory tells us that $H_\dR^1(\A^1,\{0\},\id)$ also has dimension~$1$.
 It is easy to see that $(\!\md z,0)$ is  not in the image of the differential:
 Indeed $P\mapsto dP-P\md z$ is injective,
 and the preimage of $\md z$ under this injection
 is the constant polynomial $-1$,
 which does not have constant coefficient~$0$.
 Hence $(\!\md z, 0)$ generates our cohomology.
 The periods of $(\A^1,\{0\},\id,1)$ are precisely the elements of
 $k$ as
 \[ \int_{G_1} \e^{-z}\md z=1.\]
 Unsurprisingly, these elements are
  naive exponential periods as explained in
 Example~\ref{ex:3}. 

 We now turn to  $X=\A^1, Y= \{0\}$ and  $f=z^n$ for $n\geq 2$. In this case the
 boundary of $\Bcirc(\A^1, f)$ has $n$ components, hence
 $H_1^\rd(\A^1,\{0\},f)$ is of dimension $n$. As generators for homology
 we can use the $n$ different preimages of $[0,\infty)$ under $z\mapsto z^n$.
 They are of the form $G_{s^m}$ for $m=0,\dots,n-1$ with $s=\e^{2\pi i/n}$.
 The boundary map in singular homology maps each of them to
 minus the class of the point $0$.

 In this case the de Rham complex has the shape
 \[ k[z]\xrightarrow{P\mapsto (dP-nz^{n-1}P\md z,P(0))} k[z]\md z\oplus k.\]
 All elements in $H^1_\dR(\A^1,\{0\},f)$
 are represented by pairs $(Qdz,a)$.
 Their periods are computed as
 \[ \int_{G_{s^m}} \e^{-z^n}Q\md z-a.\]
 These are naive exponential  periods.
It is easy to see that the classes 
of $(z^j\md z,0)$ for $j=0,\dots,n-1$ generate $H^1_\dR(\A^1,\{0\},f)$. 
Hence we get $k$-linear combinations of the numbers
\begin{align*}
 \int_{G_{s^m}}\e^{-z^n}z^j\md z
&=\int_0^\infty \e^{- t^n}(\e^{2\pi mi/n}t)^j \e^{2\pi im/n}\md t\\
&=\e^{2\pi i(j+1)m/n}\int_0^\infty \e^{-s}s^{\frac{j}{n}}\frac{1}{n}s^{\frac{1}{n}-1}\md s\\
&=\frac{\e^{2\pi i(j+1)m/n}}{n}\int_0^\infty e^{-s}s^{\frac{j+1}{n}-1}\md s\\
&\in\Qbar^*\cdot \Gamma\left(\frac{j+1}{n}\right).
\end{align*}
\end{ex}

\begin{rem}
 \label{bcirc-vs-bsharp}
 The preceding example provides
  an explicit instance of \xref{prop:comp_homol}
 which is an important ingredient in the final comparison theorem:
	rapid decay homology  is not only computed by $\Bcirc(\A^1,f)$,
 but also by $\Bsharp(\A^1,f)=\C\cup \tilde{f}^{-1}(1\infty)$
 so we can choose intervals with endpoints $E$ in~$\{0\}\cup\tilde{f}^{-1}(1\infty)$.
\end{rem}

\subsection{The oriented real blow-up of a marked curve}
\label{ssec:orbmc}
Let $\bar C$ be a smooth projective complex curve,
or in other words, a compact Riemann surface.
Let $\bar f \colon \bar C \to \Pe^1$ be a non-constant
meromorphic function. 
Let $Q_1, \dots, Q_n \in \bar C$ denote the poles of~$\bar f$,
let $P_1, \dots, P_m$ be some points on~$\bar C$ distinct from the~$Q_i$,
and denote by $C \subset \bar C$ the complement of
$\{P_1, \dots, P_m, Q_1 \dots Q_n\}$.
Denote by $f \colon C \to \A^1$ the restriction of~$\bar f$ to~$C$.

We now consider the real oriented blow-up $B(C)=B_{\bar
  C}(C)$
and the map $\tilde f \colon B(C) \to \Ptilde$ induced by~$f$.
It adds a circle to $C^\an$ in each of the points $P_i$ and~$Q_j$.
The algebraic map $f \colon C \to \A^1$
induces a semi-algebraic map of manifolds with boundary
$\tilde f \colon B(C) \to \Ptilde$.
The circles around the $P_i$ are mapped to $f(P_i) \in \C$.
The circles around the $Q_i$ are mapped to the circle at infinity of~$\Ptilde$.
As in Definition~\ref{defn:our_Bcirc} let $\Bcirc(C,f) \subset B(C)$
be the open subset of points either in~$C^\an$ or
mapping to $\Re(s\infty) > 0$ on the boundary of~$\Ptilde$.
So it removes the circles around the~$P_i$'s
and some circle segments from the circles around the the~$Q_j$'s.

The following figure illustrates the case $\bar C = \Pe^1$.

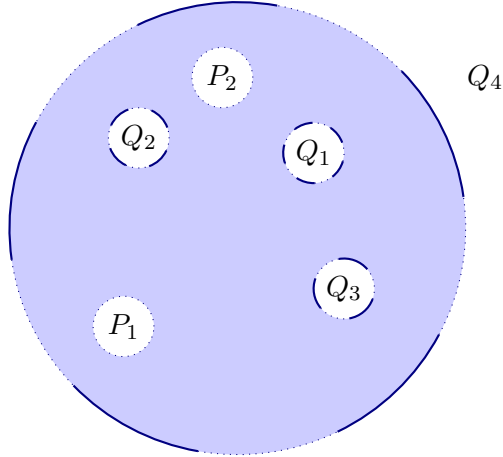
\def\parc#1#2#3#4#5#6{
  \draw[#6,blue!50!black] ([shift=(#4:#3)]#1,#2) arc (#4:#4+#5:#3);
}

\def\segmentedcircle#1#2#3#4#5#6#7{
 \fill[fill=#7] (#1,#2) circle (#3);
 \
 \foreach \a in {0,#5,...,179} {
  \parc{#1}{#2}{#3}{2*\a+#4}{#5}{thick};
  \parc{#1}{#2}{#3}{2*\a+#5+#4}{#5}{dotted};
 }
 \node at (#1,#2) {#6};
}

\def\dashedcircle#1#2#3#4{
  \fill[fill=white] (#1,#2) circle (#3);
  \draw[draw=blue!50!black,dotted] (#1,#2) circle (#3);
  \node at (#1,#2) {#4};
}

\begin{center}
 \begin{tikzpicture}
  \segmentedcircle{0}{0}{3}{80}{36}{}{blue!20!white}
  \segmentedcircle{1}{1}{.4}{20}{36}{$Q_1$}{white}
  \segmentedcircle{-1.3}{1.2}{.4}{110}{45}{$Q_2$}{white}
  \segmentedcircle{1.4}{-.8}{.4}{40}{60}{$Q_3$}{white}
  \node[anchor=west] at (2.5,2) {\rlap{\quad$Q_4$}};

  \dashedcircle{-1.5}{-1.3}{.4}{$P_1$}
  \dashedcircle{-.2}{2}{.4}{$P_2$}
 \end{tikzpicture}
\end{center}
Here the shaded area should be understood as a subset of the $2$-sphere, so that the outside (around $Q_4$) is one more disk.

\subsection{A $1$-dimensional comparison}
We now show that generalised naive exponential periods
are cohomological exponential periods.

\begin{proposition}
 \label{curve-nv-sub-log}
 Let $\alpha = \int_G \e^{-f} \omega$ be
 a generalised naive exponential period over~$k_0$
 as in \cref{defn:gnaive}.
 Assume that $\dim(G) = 1$.
 Then $\alpha$ is a cohomological exponential period
 for a tuple $(C, Y, f, 1)$,
 where $C$ is a smooth curve defined over~$k$,
 $Y \subset C$ is a finite set of points,
 and $f \colon C \to \A^1_k$
 is a regular function.
\end{proposition}
This is a special case of \xref{naive_is_effective}.

\begin{proof}
 By \cref{cor:nv_converges},
 every generalised naive exponential period is absolutely convergent.
 Hence we may apply \cref{prop:geometry} to obtain
 a smooth affine curve $C$ over~$k_0$,
 a finite set of points $Y \subset C(k_0)$,
	a closed pseudo-oriented $1$-dimensional $k$-semi-algebraic
 subset $G$ of~$C(\R)$ 
 with endpoints in~$Y$,
 a function $f \colon C_k \to \A^1_k$ that is proper on~$G$
 and such that $\overline{f(G)} \subset \Bcirc$,
 and a regular $1$-form~$\omega$ on~$C_k$,
 such that $\alpha = \int_G \e^{-f} \omega$.
 By abuse of notation we replace
 $C$ and $Y$ by $C_k$ and $Y_k$ from now on.

 Certainly, the form $\omega$ defines a class $[\omega] \in H^1_\dR(C,Y,f)$.

 The semi-algebraic set $\reg_1(G)$ is semi-algebraically homeomorphic
 to a  finite union of open intervals and
 circles.  We may consider connected components separately.
 Thus, without loss of generality,
 $\reg_1(G)$ is homeomorphic to an open interval and $G$ its closure in $C^\an$.
 The semi-algebraic set $G$ is homeomorphic to either a circle, 
or an interval. Its endpoints are in $Y$ or at infinity.
 By assumption, we are given an orientation on the complement of finitely many points of $G$.
By enlarging $Y$ and cutting $G$ into subintervals, we may assume that $G$ is oriented.

Let $\bar{C}$ be a smooth compactification of $C$,
 and $\bar G$ the closure of $G$ in $B_{\bar C}(C)$.
 It is compact because $B(C)$ is.
 \Cref{lem:GB} implies $\bar{G}\subset \Bcirc_{\bar{C}}(C,f)$.

 By construction, the boundary of $\bar G$
 is contained in $Y \cup \partial \Bcirc_{\bar C}(C,f)$. 
 It defines a class
 $[G] \in H_1^\rd(C,Y,f) = H_1(\Bcirc_{\bar C}(C,g),Y\cup\partial \Bcirc_{\bar C}(C,f);\Q)$.
 Finally, as in \Cref{ex:compute_period},
 the period pairing of these classes is computed as
 \[ \langle [\omega],[G]\rangle = \int_{\bar{G}}\e^{-f}\omega = \alpha.\]
 This proves the result:
 $\alpha$ is indeed a cohomological exponential period.
\end{proof}

\subsection{Converse direction}
We now want to express cohomological exponential periods
as naive exponential periods.
This means that we start with a marked curve $(C,Y)$, a regular function $f\in\Oh(Y)$
and cohomology classes $\gamma \in H_1^\rd(C,Y,f)$
and $\omega \in H_\dR^1(C,Y,f)$.
We want to show that the period pairing $\langle \omega, \gamma \rangle$
is a naive exponential period.
Let us sketch the ingredients of the proof:
\begin{enumerate}[label=(\textit{\roman*})]
 \item The first step is the observation that rapid decay homology~$H_1^\rd(C,Y,f)$
   is computed  via the ordinary homology of the space~$\Bcirc(C,f)$.
 \item We then note that $\Bcirc(C,f)$ is homotopic to a certain subset~$\Bsharp(C,f)$.
  We will give an ad hoc definition of this subset here;
   for the general definition see \xref{defn:sharp}.

  This step is crucial,
  because in the next step
  it will allow us to obtain semi-algebraic sets~$G$
  whose image is contained in a suitable strip: $f(G) \subset S_{r,s}$.
  See also \cref{bcirc-vs-bsharp}.
 \item Finally, we use semi-algebraic triangulation results
  and the delicate \cref{lem:cc}
  to realise $\gamma$ as a linear combination
  of homology classes of semi-algebraic sets.
  This will allow us to realise $\langle \omega, \gamma \rangle$
  as a naive exponential period.
\end{enumerate}

\begin{proposition}[Special case of \xref{prop:coh_implies_naive}]
 \label{curve-log-sub-nv}
 Let $C\subset\A^n$ be a smooth affine curve over $k$,
 $f\in\Oh(C)$, and $Y\subset C$ a proper closed subvariety.
 Then every cohomological exponential period of $(C,Y,f,1)$
 is a naive exponential period.
\end{proposition}

\begin{proof}
 By definition, $f\in k[C]$.
 We also write $f$ for a polynomial in $k[z_1,\dots,z_n]$ representing it.
 As $C$ is affine,
 the twisted de Rham cohomology $H^1_\dR(C,Y,f)$
	is a quotient of $\Omega^1(C) \oplus \bigoplus_{y \in Y} k$, 
 hence every element is represented by a tuple $(\omega,a_y)$.
As in Example~\ref{ex:compute_period}, the $a_y$ contribute
	$\sum_{\sigma_Y}a_Y\in k$ to the cohomological exponential period. This is a naive exponential period. Hence it suffices to consider the case $a_y=0$ for all $y\in Y$ from now on.
We also write $\omega$
 for the element of $\Omega^1(\A^n)$ representing $\omega\in\Omega^1(C)$.

 \textbf{Step 1.}
 Let $\bar{C}$ be a smooth compactification of $C$
 and let $Z=\bar{C}\ssm C$ be the points at infinity.
 By \cref{prop:fj_rd} and \cref{rem:comment}
 \[ H_1^\rd(C,Y,f)=H_1(\Bcirc(C,f),Y^\an\cup \partial \Bcirc(C,f);\Q).\]
We decompose $Z=Z_f\cup Z_\infty$ such that
 $f$ is regular at the points of $Z_f$
 and has a pole at the points of $Z_\infty$.
 Let $d_z\ge 1$ be the multiplicity
 of $\bar{f}$ at $z\in Z$.
 The oriented real blow-up of $\bar{C}$ in $Z$
 replaces each point $z\in Z^\an$ by a circle $S_z$.
 It is compact.
 The boundary is a disjoint union of circles.
 The map $\tilde{f} \colon B(C)\to \Ptilde$
   maps these circles either to $\C$ (the case $z\in Z_f$)
 or to the circle at infinity (the case $z\in Z_\infty$).
 In the latter case, the map on the circle is a $d_z$ to $1$ cover.

 By definition the subset $B^{\circ}(C,f)$
 is the union of the preimage of~$C^\an$ and those points $P$ in the circles $S_z$ above $z\in Z_\infty^\an$
 that are
 in the preimage of the half circle $\{w\infty \mid \Re(w)>0\}$.
 Hence the boundary  of $\Bcirc(C,f)$ consists of $d_z$ many circle segments
 for every $z\in Z_\infty^\an$.

  \textbf{Step 2.}
 Now consider the smaller subset $\Bsharp(C,f)$
 defined as the union of the preimage of~$C^\an$ and
the points $P$ in the circles $S_z$ above $z\in Z_\infty^\an$
 that are
 in the preimage of $1\infty$.  Hence the boundary $\partial \Bsharp(C,f)$ of $\Bsharp(C,f)$ consists of $d_z$ many disjoint points
 for every $z\in Z_\infty^\an$.
 In particular the boundaries of $\Bcirc(C,f)$ and $\Bsharp(C,f)$
 are homotopy equivalent.
 Both $\Bcirc(C,f)$ and $\Bsharp(C,f)$ are homotopy equivalent to $C^\an$.
 Thus
 \[ H_1^\rd(C,Y,f)=H_1(\Bsharp(C,f),Y^\an\cup \partial \Bsharp(C,f);\Q).\] 
 Note that $\Bsharp(C,f)$ is not a manifold with corners,
 hence we are not able to interpret the right hand side
 in the sense of $C^1$-homology as defined in \cref{sec:homology}.
 However, it is a topological space
 so ordinary singular homology is perfectly well-defined
 and this is how we interpret the right-hand side.

\textbf{Step 3.}
 The space $\Bcirc(C,f)$ is a $k_0$-semi-algebraic $C^\infty$-manifold with boundary. The subset $\Bsharp(C,f)$ is $k_0$-semi-algebraic.
 By \cref{prop:triangle_manifold},
	$\Bcirc(C,f)$ has a $k_0$-semi-algebraic triangulation
 compatible with $\Bsharp(C,f)$, $Y$ and $\partial\Bsharp(C,f)$
 that is
 globally of class $C^1$. 
 In particular, the points in $Y^\an\cup\partial\Bsharp(C,f)$
 are vertices. By \Cref{lem:cc}, the closed core of its barycentric
 subdivision is a strong deformation retract of $\Bsharp(C,f)$. 
 We denote the closed core by $A$. It contains all vertices of the triangulation, in particular $Y^\an\cup \partial\Bsharp(C,f)\subset A$.
 Hence
\[ H_1(\Bsharp(C,f),Y^\an\cup\partial \Bsharp(C,f);\Q)=H_1(A,Y^\an\cup\partial\Bsharp(C,f);\Q).\]

 The subcomplex $A$ is compact, hence simplicial and singular homology of $A$ agree.
 Therefore every homology class is represented by a linear combination
 of closed semi-algebraic $1$-simplices in~$A$.
 The triangulation is $C^1$,
 hence the closed $1$-simplices in the triangulation of $C^\an$ define elements of $S_1(C,f)$.
 In all, each homology class in $H_1^\rd(C,Y;\Q)$
 is represented by a linear combination of $C^1$-paths in
 $\Bsharp(C,f)$ (in fact even in $A$) with boundary in $Y^\an\cup \partial \Bsharp(C,f)$.
 The period integral is defined by integrating $\e^{-f}\omega$ on
 these paths, see Definition~\ref{defn:period_pairing_smooth} and  Lemma~\ref{lem:period_map_smooth}.

 Let $\gamma \colon [0,1]\to A$ be one these $C^1$-paths.
 We put $G=\gamma([0,1])\cap C^\an$.
 We need to check that it satisfies the conditions needed for naive exponential periods. The subset $G$ is $k_0$-semi-algebraic.
 The closure $\bar{G}=\gamma([0,1])$ differs from $G$ by at most two points,
 the endpoints.
 The image $\tilde{f}(\bar{G})$ in $\Ptilde$ is compact
 and contained in $\Bsharp(C,f)$,
 hence $f(G)$ is contained in a suitable strip $S_{r,s}$ for $r\in\R,s>0$.
 The map $\tilde{f} \colon \bar{G}\to \Ptilde$ is proper
 because $\bar{G}$ is compact.
 By definition,
 the preimage $\tilde{f}^{-1}(1\infty)$ does not contain any points of $G$.
 Hence $f \colon G\to \C$ is also proper.
 We conclude that $\int_G\e^{-f}\omega$ is a naive
 exponential period.

 Our cohomological period was a linear combination of such.
 By Lemma~\ref{lem:change_field} a linear combination of naive exponential periods
 is a naive exponential period.
\end{proof}

\section{Exponential periods: the general case}
\label{sec:exp_gen}

Throughout this section let $k\subset\C$ be a subfield
such that $k$ is algebraic over $k_0=k\cap\R$.
All varieties are defined over $k$.

We turn to the definition of exponential periods for general $(X,Y,f)$
again following Fres\'an and Jossen in \cite{fresan-jossen}.
Notation for the smooth affine case was set up in \xref{sec:exp}.
\subsection{Complexes of varieties and their compactifications}
We denote by $\SmAff/\A^1$ the category of smooth affine varieties $X$
together with a structure map $f \colon X \to \A^1$.
Note that we do not require $f$ to be smooth. We often write
$X$ instead of $(X,f)$.
Let $\Z[\SmAff/\A^1]$ be the additive hull of $\SmAff/\A^1$:
\begin{itemize}
\item
the objects are the objects of $\SmAff/\A^1$;
\item the morphisms are formal $\Z$-linear combinations of morphisms in $\SmAff/\A^1$, more precisely 
for connected $X$ we have
\[ \Hom_{\Z[\SmAff/\A^1]}((X,Y) = \Z[\Mor_{\SmAff/\A^1}(X,Y)]; \]
\item the disjoint union is the direct sum.
\end{itemize}
We denote by $C_+(\SmAff/\A^1)$
the category of bounded below homological complexes over $\Z[\SmAff/\A^1]$.

We denote by $\Sm/\Pe^1$ and  $\SmProj/\Pe^1$
the category of smooth and smooth projective varieties $X$, respectively, together with
a structure map $f \colon X \to \Pe^1$.
As in the affine case we define $\Z[\Sm/\Pe^1]$,
$\Z[\SmProj/\Pe^1]$, as well as $C_+(\Z[\Sm/\Pe^1])$ and $C_+(\Z[\SmProj/\Pe^1])$.

Let $X$ be a smooth variety, $f \colon X\to\A^1$.
Recall from \Cref{ssec:gc2} that a \emph{good compactification}
of $X$ relative to $f$
is a pair $(\bar{X},\bar{f})$ where $\bar{X}$ is smooth and projective,
$\bar{f} \colon \bar{X}\to\Pe^1$ a morphism and
$X\to \bar{X}$ is a dense open immersion such that the complement
$X_\infty$ is a divisor with simple normal crossings and $\bar{f}$ extends $f$.

\begin{defn}\label{defn:good_complex}
Let $(X_\bullet,f_\bullet)\in C_+(\SmAff/\A^1)$. A \emph{good compactification} of
$X_\bullet$ relative to $f_\bullet$ is an object $(\bar{X}_\bullet,\bar{f}_\bullet)\in C_+(\SmProj/\Pe^1)$ together with a morphism $\iota_\bullet:X_\bullet\to\bar{X}_\bullet$ in
	$C_+(\Sm/\Pe^1)$  such that for all $n$ the morphism $\iota_n:X_n\to \bar{X}_n$ is a good compactification relative to $f_n$.
\end{defn}

In particular, each $\iota_n$ is assumed to be a morphism of varieties, not a general morphism in $\Z[\Sm/\Pe^1]$.

\begin{lemma}
 Let $X$ be a smooth variety and $f \colon X\to \A^1$ a morphism.
 \begin{enumerate}
  \item The system of good compactifications of $(X,f)$ is filtered.
  \item Given a morphism $g \colon Y\to X$ of smooth varieties
   and a good compactification of $X$ relative to $f$
   there is a good compactification $\bar{Y}$ of $Y$ relative to $f\circ g$
   and a morphism $\bar{Y}\to\bar{X}$ over~$g$.
 \end{enumerate}
\end{lemma}
\begin{proof}
 Let $X\to X_1$ and $X\to X_2$ be good compactifications of $X$ relative to $f$.
 Let $X_3'$ be the closure of $X$ in $X_1\times_{\Pe^1}X_2$.
 Let $X_3\to X_3'$ be a desingularisation such that the complement $X_3\ssm X$ is
 a simple normal crossings divisor.
	A morphism $h:X_1\to X_2$ of good compactifications of $X$ relative to $f$ is uniquely determined if it exists because
 $X$ is dense in~$X_1$.

 Let $g \colon Y\to X$ be a morphism of smooth varieties.
	Let $\bar{X}$ be a good compactification of $X$ relative to $f$.
 Choose any compactification $Y'$ of $Y$.
 Possibly after replacing $Y'$ by a blow-up, the map $g$ extends to $Y'$.
 Picking a desingularisation~$\bar Y$ of~$Y'$
  finishes the proof of this lemma.
\end{proof}

\begin{cor}\label{cor:filtering}
 Let $(X_\bullet,f_\bullet)$ be a bounded below (homological) complex
 in $\Z[\SmAff/\A^1]$. Then the system of good compactifications
 of $(X_\bullet,f_\bullet)$ is non-empty, filtering and functorial.
\end{cor}
\begin{proof}
 We construct $\bar{X}_n$ by induction on $n$.
 For small enough $n$ there is nothing to show.
 Suppose we have constructed good compactifications for $n< N$.
 Let $X_N=\bigcup X_N^j$ be the decomposition into connected components.
 The differential $d \colon X_N\to X_{N-1}$
 is of the form $d=\sum_{i=1}^m a_ig_i$
 for morphisms $g_i \colon X^{j(i)}_{N}\to X_{N-1}$ and $a_i\in\Z$.
 Let $Y_i$ be a good compactification of $X_N$ such that $g_i$ lifts.
 Let $\bar{X}_N$ be a common refinement of $Y_1,\dots,Y_m$.
 By construction $d$ lifts to $\bar{X}_N$.
 We need to check that the composition
 $\bar{X}_N\to \bar{X}_{N-1}\to \bar{X}_{N-2}$ vanishes.
 This is a combinatorial identity on the coefficients of the $g_i$.
 It can be checked on the dense open subsets $X_N\to X_{N-1}\to X_{N-2}$,
 where it holds because $X_\bullet$ is a complex.
 This finishes the proof of existence.

 The same method also produces common refinements of two good compactifications
 and lifts of morphisms of complexes.
\end{proof}

\subsection{Rapid decay homology for complexes}
Recall from \xref{defn:S^rd} the description of rapid decay homology for
$(X,f)\in\SmAff/\A^1$. We put
 \[ S_*^\rd(X,f)=S_*(\Bcirc_{\bar{X}}(X,f))/S_*(\partial \Bcirc_{\bar{X}}(X,f))  \]
where $S_*(-)$ is as in Section~\ref{sec:homology} the complex of $C^1$-simplices and $\Bcirc_{\bar{X}}(X,f)$ the semi-algebraic partial compactification of $X^\an$ of Definition~\ref{defn:our_Bcirc}.  By
\xref{thm:compare} it computes singular homology.

Note that the complex $S_*^\rd(X,f)$
depends on the choice of a good compactification~$\bar{X}$
relative to~$f$. Viewed in the derived category it becomes independent of the choice up to canonical isomorphism. It is essential for our constructions to work with actual complexes.
We want to extend the construction to complexes of varieties.

\begin{defn}\label{defn:rd_general}
Let $(X_\bullet,f_\bullet)$ be in $C_+(\SmAff/\A^1)$. We define
\[ S^\rd_*(X_\bullet,f_\bullet)\]
as the total complex of the double complex $(S_m^\rd(X_n,f_n))_{n,m}$
for some choice of good compactification $(\bar{X}_\bullet,\bar{f}_\bullet)$ of
$(X_\bullet,f_\bullet)$.
\end{defn}

\begin{rem}
 By Corollary~\ref{cor:filtering},
 this is well-defined up to canonical isomorphism in the derived category.
\end{rem}

\subsection{Twisted de Rham cohomology and periods for complexes}

Recall from \cite{fresan-jossen}, see also \xref{ssec:twist},
that the twisted de Rham cohomology of $(X,f)\in\SmAff/\A^1$
is defined as the cohomology of the complex $\Omega^*(X)$ with differential
$\Omega^p(X)\to\Omega^{p+1}(X)$ given by $d\omega-df\wedge \omega$. 
As before we denote the complex of sheaves by $\DR(\Eh^f)$.

\begin{defn}\label{defn:deRham_general}
 Let $(X_\bullet,f_\bullet)\in C_+(\SmAff/\A^1)$. We define
$H_\dR^n(X_\bullet,f_\bullet)$ to be the cohomology of the total complex
$\RdR(X_\bullet,f_\bullet)$ of the double complex   $\DR(\Eh^{f_\bullet})(X_\bullet)$.
\end{defn}

\begin{lemma}Let $(X_\bullet,f_\bullet)\in C_+(\SmAff/\A^1)$. Then the period
map of \xref{defn:period_pairing_smooth} extends to a pairing of complexes
\[ \RdR(X_\bullet,f_\bullet)\times S_*^\rd(X_\bullet,f_\bullet)\to \C, \]
i.e., a morphism of complexes
\[ \RdR(X_\bullet,f_\bullet)\to\Hom(S_*^\rd(X_\bullet,f_\bullet),\C).\]
\end{lemma}
\begin{proof}
We apply \xref{lem:period_map_smooth} to each $X_n$, then take total complexes.
\end{proof}

\begin{defn}
 \label{defn:smaff}
Let $(X_\bullet,f_\bullet)\in C_+(\SmAff/\A^1)$, $n\in\Na$.
The \emph{period pairing} for $(X_\bullet,f_\bullet,n)$ is the induced map
\[ H^n_\dR(X_\bullet,f_\bullet)\times H^\rd_n(X_\bullet,f_\bullet)\to \C.\]
The elements in the image of this pairing are
called the \emph{exponential periods} of $(X_\bullet,f_\bullet,n)$.
We denote the set of these numbers for varying $(X_\bullet,f_\bullet,n)$
by $\Psmaff(k)$.
\end{defn}
\begin{rem}
Fres\'an--Jossen interpret these periods
as periods for a suitable category of effective exponential motives.
We consider them in \cref{sec:concl}.
As usual the category of all exponential motives is defined as the localisation of the category of effective exponential motives with respect to the Tate object. On the level of periods this amounts to inverting $\pi$. We do not consider the non-effective case in our paper.
\end{rem}

\subsection{Rapid decay homology via hypercovers}
Let $X$ be a variety over $k$ and $f\in\Oh(X)$ a regular function.
We describe $H_n^\rd(X,f)$ via $C_+(\SmAff/\A^1)$. 

A simplicial or bisimplicial variety $X_\bullet\to X$ is called
a \emph{hypercover} of~$X$
if it is a hypercover for the h-topology.
We do not go into details about this topology, which is introduced and studied in \cite{voesing}. For our purposes it suffices to remark that  hypercovers are universal homological isomorphisms: 
 $H_n(X^\an_\bullet,\Z) \to H_n(X^\an,\Z)$ is an isomorphism and the same is true after passing to any base change $T\to X^\an$. The only examples that we are going to need are open and closed covers, \cref{ssec:hypercover}.
We say that a hypercover is smooth, proper and/or affine, respectively,
if all $X_n$ are smooth, proper and/or affine.
By resolution of singularities, every variety has an h-cover by a smooth affine variety. It is obtained as an open affine cover of a desingularisation.
By applying the procedure of \cite[Section~6.2.5]{hodge3}, this implies that
every hypercover can be refined by a smooth affine hypercover.

\begin{lemma}\label{lem:rd_absolute}
	Let $X$ be a variety over $k$, and $f\in\Oh(X)$ a regular function. 
Let $X_\bullet\to X$ be a smooth affine hypercover and
$C(X)$ the total complex of $X_\bullet$ in $C_+(\SmAff/\Z)$. 
Then there is a natural isomorphism
\[ H_n^\rd(X,f)\isom H_n^\rd(C(X)).\]
\end{lemma}
\begin{proof}
	Let $r\in\R$. We put $T_r(X)=f^{-1}(S_r)$ with $S_r=\{z\in\C\mid \Re(z)\geq r\}$. By Definition~\ref{defn:rd_original} and the comment below it, rapid decay homology of $X$ is computed by
\[ S_*(X^\an)/S_*(T_r(X))\]
for $r$ big enough. We pass to the projective limit with respect to $r$.
Note that the projective limit is exact in this case because the homology groups are finite dimensional. It preserves quasi-isomorphisms, in particular the limit complex 
\[ \projlim_r S_*(X^\an)/S_*(T_r(X))\]
still computes rapid decay homology. Our task is to link it to
$S^\rd_*(X_\bullet,f_\bullet)$ via a chain of quasi-isomorphisms.

Let again $r\in\R$. We put
$T_r(X_n) = f_n^{-1}(S_r) \subset X_n^\an$
where $f_n \colon X_n\to X\to \A^1$ is the structure map of $X_n$.
The hypercover $X_\bullet\to X$ is a universal homological isomorphism, hence
the base change $T_r(X_\bullet)\to T_r(X)$ is also a universal homological isomorphism. This implies that $S_*(X^\an)/T_r(X)$ is quasi-isomorphic to
(the total complex of) $S_*(X^\an_\bullet)/S_*(T_r(X_\bullet))$. By passing to the limit, we deduce that
\[ \projlim_r S_*(X^\an)/S_*(T_r(X))\leftarrow \projlim_r S_*(X^\an_\bullet)/S_*(T_r(X_\bullet))\]
is a quasi-isomorphism.

By \cite[Proposition 3.5.2]{fresan-jossen} (see also \xref{prop:fj_rd})
and the fact that $S_*(-)$ computes singular homology (see \xref{thm:compare}),
we have for each~$n$ and sufficiently large~$r$, a quasi-isomorphism
\[ S_*(X^\an_n)/S_*(T_r(X_n))\to
 S_*(B_{\bar{X}_n}(X_n))/S_*(T_r(X_n))\leftarrow S_*^\rd(X_n,f_n).\]
We pass to the projective limit and then take total complexes. (Not the other way around; $r$ ``big enough'' may depend on $n$.) This induces  quasi-isomorphisms
\[
	\begin{tikzpicture}
		\node at (0,0) {$\projlim_r S_*(B_{\bar{X}_\bullet}(X_\bullet))/S_*(T_r(X_\bullet))$};
		\node at (-4,-1.5) {$\projlim_r S_*(X_\bullet^\an)/S_*(T_r(X_\bullet))$};
		\node at (4,-1.5) {$S^\rd_*(X_\bullet,f_\bullet)$};
		\draw[->] (-3,-1.0) -- (-1,-.5);
		\draw[->] (3,-1.0) -- (1,-.5);
	\end{tikzpicture}
\]
We have completed our task.
\end{proof}

\begin{rem}\label{rem:chain}
\begin{enumerate}
\item
The proof accomplishes more than was claimed: We find functorial chains of quasi-isomorphisms computing rapid decay homology rather than simply an isomorphism on homology.
\item The proof was written for the simplicial case, but the same argument works for bisimplical hypercovers as used for $Y$ in Section~\ref{ssec:hypercover}.
\end{enumerate}
\end{rem}

\subsection{The relative case}\label{sec:relative}
Let $X$ be a variety over $k$, $Y\subset X$ a closed subvariety and $f\in\Oh(X)$ a regular function.
We want to define exponential periods for $H_n^\rd(X,Y,f)$ by reduction to
the case $C_+(\SmAff/\A^1)$.

\begin{lemma}\label{lem:rd_relative}
	Let $X$ be a variety over $k$, $f\in\Oh(X)$ a regular function and $Y\subset X$ a closed subvariety.
Let $X_\bullet\to X$ be a smooth affine hypercover, $Y_\bullet\to Y$ a smooth affine hypercover with a morphism $Y_\bullet\to X_\bullet$ of simplicial schemes compatible with the inclusion. Let
\[ C(X,Y)=\cone( Y_\bullet\to X_\bullet)\]
be the cone of the associated map of total complexes in $C_+(\SmAff/\Z)$. Then
there is a natural isomorphism
\[ H_n^\rd(X,Y,f)\isom H_n^\rd(C(X,Y)).\]
\end{lemma}
\begin{proof}
We apply Lemma~\ref{lem:rd_absolute} to $X$ and $Y$. As pointed out in
Remark~\ref{rem:chain} we get a compatible chain of quasi-isomorphisms, hence
also a quasi-isomorphism for the cone.

The latter computes
relative rapid decay homology, which sits in the sequence
\[ \to H_n^\rd(Y,f)\to H_n^\rd(X,f)\to H_n^\rd(X,Y,f)\to H_{n-1}^\rd(Y,f)\to \qedhere \]
\end{proof}

Given this lemma, we are led to define:

\begin{defn}[{\cite[Definition~7.1.6]{fresan-jossen}}]\label{defn:relative_deRham} Let $X$ be a variety over $k$, $f\in\Oh(X)$, $Y\subset X$ a closed
subvariety. Choose $C(X,Y)\in C_+(\SmAff/\A^1)$ 
as in Lemma~\ref{lem:rd_relative}.
\begin{enumerate}
\item
We define $H^n_\dR(X,Y,f)$ as the cohomology of
\[ \RdR(X_\bullet,Y_\bullet,f)=\RdR(C(X,Y)),\]
see Definition~\ref{defn:deRham_general}.
\item We define the period pairing for $(X,Y,f,n)$ as the period pairing
\[ H_n^\rd(X,Y,f)\times H^n_\dR(X,Y,f)\to\C\]
 for $C(X,Y)$, see Definition~\ref{defn:smaff}.
\end{enumerate}
\end{defn}
This generalises Definition~\ref{defn:dR-snc}, where $X$ was assumed smooth and $Y$ a simple normal crossing divisor, to arbitrary $X$ and $Y$.
Now all details have been filled in for the definition of
cohomological exponential periods, see Definition~\ref{defn:coh_period}. 

For later use we record a tautological property of our construction.
\begin{rem}\label{rem:is_smaff}
Every exponential period is the period of a complex in $C_+(\SmAff)$:
\[ \Pcoh\subset\Psmaff.\]
\end{rem}

\section{Cohomological exponential periods are naive exponential periods}
\label{sec:coh_is_naive}

The aim of this section is to prove the key comparison in \cref{prop:coh_implies_naive}.
See \xref{curve-log-sub-nv} for the corresponding statement
in the special case where $X$ is a curve.
In that case, the main ideas of the proof are present,
but several technicalities are avoided.

\begin{proposition}\label{prop:coh_implies_naive} Let $k\subset \C$ be as in \Cref{ssec:field2}.
Let $(X,Y)$ be a log pair, i.e.,
$X$ a smooth variety, $Y\subset X$ a simple normal crossings divisor.
Let $f\in\Oh(X)$, and let $\alpha$ be a cohomological exponential period of
$(X,Y,f,n)$ (see \Cref{defn:coh_period}).
Then $\alpha$ is a naive exponential period:
\[ \Plog(k)\subset\Pnaive(k).\]
\end{proposition}

\begin{rem}
 This justifies that our fairly restrictive definition
 of a naive exponential period was a reasonable choice.
\end{rem}

The proof is technical and will take the rest of the section.

\subsection{Notation}
Throughout, let $k$ be as in \Cref{ssec:field2}, $k_0=k\cap \R$.

If $X$ is a smooth variety, $f\in\Oh(X)$ a regular function, and
$\bar{X}$ a good compactification relative to $f$,
then we put
$X_\infty = \bar{X} \ohne X$. 
We decompose $X_\infty=D_0\cup D_\infty$
where $D_0$ consists of the horizontal components
and $D_\infty$ of the vertical components mapping to $\infty$ in $\Pe^1$.

As before,
we denote by $\tilde{f} \colon B_{\bar{X}}(X) \to \Ptilde$ the induced map
on the oriented real blow-up of $\bar{X}^\an$ in $X_\infty^\an$.

Recall from \Cref{defn:our_Bcirc} that
\begin{align*}
 \Bcirc_{\bar{X}}(X,f) &=
 B_{\bar X}(X) \ohne
 \{ x \in \partial(B_{\bar{X}}(X)) \mid
  \pi(x)\in D_0^\an
  \text{ or } \Re(\tilde{f}(x)) \le 0 \},\\
\partial \Bcirc_{\bar{X}}(X,f) &=\Bcirc_{\bar{X}}(X,f)\ssm X^\an.
\end{align*}
We introduce a variant.
\begin{defn}
 \label{defn:sharp}
We put
\begin{align*}
 \Bsharp_{\bar{X}}(X,f) &=
 B_{\bar X}(X) \ohne
 \{ x \in \partial(B_{\bar{X}}(X)) \mid
  \pi(x) \in  D_0^\an \text{ or } \tilde{f}(x) \neq 1\infty \},\\
 \partial \Bsharp_{\bar{X}}(X,f)&=\Bsharp_{\bar{X}}(X)\ssm X^\an.
\end{align*}
\end{defn}

The spaces $B_{\bar{X}}(X)$ and $\Bcirc_{\bar{X}}(X,f)$ are  $k_0$-semi-algebraic manifolds with
corners by \xref{prop:is_sa} and $\Bsharp_{\bar{X}}(X,f)$ is a $k_0$-semi-algebraic subset.

\subsection{A comparison of homology}
The first step in the argument is an alternative description
of rapid decay homology using $\Bsharp(X,f)$ rather than $\Bcirc(X,f)$. 
 Let us motivate why this is needed. We are going to represent homology classes
by $k_0$-semi-algebraic sets $G$ such that
 $\bar{G} \subset \Bcirc_{\bar{X}}(X,f)$.
 Hence $\overline{f(G)}\subset \Bcirc$ as in the definition of a generalised naive exponential period. The proposition will allow us to even choose $\bar{G}\subset \Bsharp_{\bar{X}}(X,f)$. Hence
$\overline{f(G)}\subset\Bsharp$ and the data defines a naive exponential period.
Indeed, the closure of the strip
\[
 S_{r,s}=\{z\in\C \mid \Re(z)>r, |\Im(z)|<s\}
\]
inside $\tilde\Pe^1$ is contained in~$\Bsharp$ (see Section~\ref{sec:not_B}).
Actually, we can only apply this argument in the case of smooth  $X$, but see \cref{ssec:hypercover} for the reduction.

\begin{proposition}\label{prop:comp_homol}
Let $V$ be a smooth variety, $f\in\Oh(V)$,
$\bar{V}$ a good compactification relative $f$, $n\geq 0$.
Then the natural map
\[ H_n(\Bsharp_{\bar{V}}(V,f),\partial \Bsharp_{\bar{V}}(V,f);\Z)\to
 H_n(\Bcirc_{\bar{V}}(V,f),\partial \Bcirc_{\bar{V}}(V,f);\Z)\]
is an isomorphism. 
\end{proposition}

\begin{proof}
 We are going to show the equivalent statement on cohomology.
The spaces are paracompact Haussdorff and locally contractible,
hence we may compute singular cohomology as sheaf cohomology.
We abbreviate
$\Bcirc(V)=\Bcirc_{\bar{V}}(V,f)$ and $\Bsharp(V)=\Bsharp_{\bar{V}}(V,f)$.
Let $j^\circ \colon V^\an\to \Bcirc(V)$
and $j^\sharp \colon V^\an\to \Bsharp(V)$ be the open immersions.
Our relative cohomology groups are computed by applying
$R\Gamma$ to $j^\circ_!\Z$ and $j^\sharp_!\Z$, respectively.

We compare their higher direct images on a subset of
${\bar{V}}^\an$. As in the definition of $\Bcirc(V)$ let  $\bar{V}\ohne V=D_0\cup D_\infty$ such
	that $D_\infty\subset \bar{f}^{-1}(\infty)$ and $\bar{f}$ is rational on $D_0$,
	i.e., it is not the constant function $\infty$ on any of the components of $D_0$.
Furthermore let
$p^\circ:\Bcirc(V)\to \bar{V}^\an\ohne D_0^\an$ and $p^\sharp:\Bsharp(V)\to \bar{V}^\an\ohne D_0^\an$ be the projections. We consider the natural map
\[ Rp^\circ_*j^\circ_!\Z\to Rp^\sharp_* j^\sharp_!\Z\]
and claim that it is a quasi-isomorphism.

We compute its stalks. For $x\in V^\an$, both sides are simply equal to $\Z$ concentrated in degree $0$.
Let $x\in D_\infty^\an\ssm D_0^\an$.
The stalk of $R^ip^\sharp_*j^\sharp_!\Z$ at $x$ is given by the limit of
$H^i( p^{\sharp-1}(U),p^{\sharp-1}(U)\cap \partial \Bsharp(V);\Z)$ for $U$ running through the system
of neighbourhoods of $x$. The analogous formula hold for $p^\circ$. Hence
it suffices to show that
\[ H^i( p^{\sharp-1}(U),p^{\sharp-1}(U)\cap \partial \Bsharp(V);\Z)\to  H^i( p^{\circ-1}(U),p^{\circ-1}(U)\cap \partial \Bcirc(V);\Z)\]
is an isomorphism for all $U$ sufficiently small.
This is a local question on~$\bar{V}$.
We choose local coordinates $z_1,\dots,z_n$ on $\bar{V}$ centered at $x$
such that $\bar{f}(z_1,\dots,z_n)=z_1^{-d_1}\cdots z_m^{-d_m}$,
where $m$ is the number of components of~$D_0$ passing through~$x$.
Let $U_\epsilon$ be the polydisc of radius~$\epsilon$ around the origin.
On $U_\epsilon$ the real oriented blow-up is given by
\[ \{(z_1,\dots,z_n,w_1,\dots,w_m)\in B_\epsilon(0)^n\times (S^1)^m \mid z_iw^{-1}_i\in\R_{\geq 0}\}.\]
We make a change of coordinates by writing $z_i=r_iw_i$ with $r_i\in [0,\epsilon)$.
Hence over $U_\epsilon$ the real oriented blow-up is given by
\[ (r_1,\dots,r_m,w_1,\dots,w_m,z_{m+1},\dots,z_n)\in [0,\epsilon)^m\times (S^1)^m\times B_\epsilon(0)^{n-m}.\]
In it $\partial\Bsharp(V)$ is the subset of points with $r_1\cdots r_m=0$, $w_1^{d_1}\cdots w_m^{d_m}=1$
and $\partial \Bcirc(V)$ is the subset  of points with $r_1\cdots r_m=0$ and $\Re(w_1^{d_1}\cdots w_m^{d_m})>0$.

We apply the long exact sequence for relative cohomology.
Hence it suffices to compare the cohomology of $p^{\sharp-1}(U_\epsilon)$ and $p^{\circ -1}(U_\epsilon)$ and
their boundaries separately. Both $p^{\sharp-1}(U_\epsilon)$ and $p^{\circ -1}(U_\epsilon)$ are homotopy equivalent to their intersection with $V^\an$, hence they have the same cohomology.

We now concentrate on the boundary. In both cases they are fibre bundles
over 
\begin{multline*}
 \{ (r_1,\dots,r_m,w_1,\dots,w_{m-1},z_{m+1},\dots,z_{n})
\in \\
[0,\epsilon)^m\times (S^1)^{m-1}\times B_\epsilon(0)^{n-m} \mid r_1\cdots r_m=0\}.
\end{multline*}
In the case of $\partial \Bsharp(V)$, the fibre consists of $d_m$ points, the solutions of $w_m^{d_m}=(w_1^{d_1}\dots w_{m-1}^{d_{m-1}})^{-1}$. In the case
of $\partial \Bcirc(V)$, the fibre consist of
$d_m$ open circle arcs centered around these points. In particular, the
inclusion $p^{\sharp-1}(U_\epsilon)\cap\partial \Bsharp\to p^{\circ-1}(U_\epsilon)\cap \partial \Bcirc$ is fibrewise a homotopy equivalence, hence it induces an isomorphism on cohomology.
\end{proof}

The goal of this whole section is to express $\alpha$ as a naive exponential period.
In order to find the set~$G$ as in the definition of a naive exponential period,
we are going to choose a $k_0$-semi-algebraic triangulation of $\Bsharp(V,f)$
that is globally of class~$C^1$ (see \xref{defn:triangulation}),
and with $V$ as in the setting of the preceding proposition.
Our next goal is therefore to construct
a suitable smooth~$V$ from the log-pair $(X,Y)$.

\subsection{Hypercovers} \label{ssec:hypercover}
By definition of cohomological exponential periods, we need to fix a smooth affine hypercover of our log-pair $(X,Y)$.
We do this explicitly.

Let $p \colon S \to T$ be a morphism of schemes.
Its \v{C}ech-nerve is the simplicial scheme $S_\bullet \to T$ with
\[ S_n = S \times_T \dots \times_T S \quad \text{($n+1$ factors)} \]
and the usual face and degeneracy maps.
It is a hypercover if $p$ is a cover for the $h$-topology.
\begin{ex}\label{ex:simplex}
If $T=\Spec(k)$ and $S=\{s_1,\dots,s_n\}$, then the \v{C}ech-nerve is a simplicial set, in fact the simplicial $n$-simplex. It is contractible and has trivial homology.
\end{ex}

We need two easy cases.

Let $X$ be a smooth variety, 
$U^1,\dots,U^M$ an affine open cover.
We put
\[ U_0 = U^1 \amalg \dots \amalg U^M \to U. \]
Let $U_\bullet$ be its \v{C}ech-nerve.
Explicitly, we have
\[ U_n = \coprod_{J \in \{1,\dots,M\}^{n+1}} U^J \]
with
\[ U^{(j_0,\dots,j_n)} = \bigcap_{i=0}^n U^{j_i}. \]
Singular homology satisfies descent for open covers (the Mayer--Vietoris property),
hence $U_\bullet \to X$ is a smooth affine hypercover,
the \v{C}ech-complex defined by the open cover.

For the second special case, let $X$ be a smooth variety,
$Y \subset X$ a simple normal crossings divisor
with irreducible components $Y^1,\dots,Y^N$.
By assumption they are smooth.
We put
\[ Y_0 = Y^1 \amalg \dots \amalg Y^N \to Y. \]
Let $Y_\bullet$ be its \v{C}ech-nerve.
Explicitly, we have
\[ Y_n = \coprod_{J\in\{1,\dots,N\}^{n+1}} Y^J \]
with
\[ Y^{(j_0,\dots,j_n)} = \bigcap_{i=0}^n Y^{j_i}. \]
Singular homology satisfies proper base change, 
hence the stalks of the higher direct images of the constant sheaf $\Q$
under $Y_\bullet^\an\to Y^\an$ are of the shape described in Example~\ref{ex:simplex}. In particular their cohomology is concentrated in degree $0$ and equal to $\Q$ there.
This makes
$Y_\bullet \to Y$ a smooth hypercover,
the \v{C}ech-complex defined by the closed cover.

We can combine the two constructions.
The bisimplicial scheme
\[ Y_\bullet \cap U_\bullet \to Y \]
is a smooth affine hypercover. 

Let $f\in\Oh(X)$ be a regular function. We denote by $f^J$ its restriction to $Y^J$ and by $f_n:Y_n\to\A^1$ the map defined by the $f^J$.
In the notation of Lemma~\ref{lem:rd_relative}, we have the equality
\[ C(X,Y)=\cone(Y_\bullet\cap U_\bullet\to U_\bullet)\]
in the category of complexes $C_+(\SmAff/\A^1)$ (and not in the category of simplicial varieties!).
We write $Y_{-1}=X$; then all terms of $C(X,Y)$ are direct sums
of objects of the form $Y_n\cap U_m$ for $n\geq -1$ and $m\geq 0$.

The definition of the period pairing also requires the choice of a good compactification of $C(X,Y)$ in the sense of Definition~\ref{defn:good_complex}.
 We proceed as follows.
Let $\bar{X}$ be a good compactification of $(X,Y)$ relative to $f$. This implies that for each component of $Y$ the closure in $\bar{X}$ is again a good compactification relative to the restriction of $f$. 
We choose an open cover $U^1,\dots,U^M$ by affine subvarieties of~$X$
such that $\bar{Y}+\sum_{i=1}^MU^i_\infty$ is still a simple normal crossings divisor.
This can be achieved by choosing $U^i$ as the intersection of $X$ with the complement of a generic hyperplane in $\bar{X}$. It is affine as the preimage of an affine variety (the complement) under an affine map (the inclusion $X\subset\bar{X}$).
Note that $\bar{X}$ is a good compactification of each of the  $U^J$.
Hence the \v{C}ech-nerve of the map
\[ \coprod_{i=1}^M\bar{X}\to\bar{X}\]
is a good compactification of $U_\bullet$. We denote it $\bar{U}_\bullet$.
For each $I\subset \{1,\dots,N\}$ let $\bar{Y}^I$ be the closure of
$Y^I$ in $\bar{X}$. By the transversality assumption it is smooth and a good compactification. Hence
\[ \bar{Y}_n=\coprod_{J\in\{1,\dots,N\}^{n+1}} \bar{Y}^J \]
defines a good compactification of $Y_n$ and of $Y_n\cap U_m$ for all $m$. The complex
\[ \cone( \bar{U}_\bullet\cap \bar{Y}_\bullet\to \bar{U}_\bullet)\in C_+(\SmProj/\A^1)\]
is a good compactification of $C(X,Y)$.

\begin{cor}
 Let $(X,Y)$ be a log pair, $f \colon X\to\A^1$.
 With the notation above
 \[ \RdR(X,Y,f)=\RdR(C(X,Y))=\Omega^*(C(X,Y))\]
 and rapid decay homology of $(X,Y,f)$ is computed by
 \[ S^\rd_*(X,Y,f) :=
  \cone(S^\rd_*(U_\bullet\cap Y_\bullet,f_\bullet)\to
  S^\rd(U_\bullet,f_\bullet))\]
with respect to the good compactification $\cone(\bar{U}_\bullet\cap\bar{Y}_\bullet\to\bar{U}_\bullet)$.
\end{cor}
\begin{proof}
 The statement for de Rham cohomology is simply \Cref{defn:relative_deRham}.
 The claim for rapid decay homology is Lemma~\ref{lem:rd_relative}. 
 Note that $U_\bullet\cap Y_\bullet$ is a bisimplical smooth affine hypercover rather than a simplicial one, but the results are still valid by Remark~\ref{rem:chain}.
\end{proof}

Our next aim is to get a clearer understanding of $\Bcirc(-,f)$ and $\Bsharp(-,f)$ applied to $C(X,Y)$ and its good compactification $C(\bar{X},\bar{Y})$.

\subsection{Real oriented blow-up and closed \v{C}ech complexes}

Let $X$ be  a smooth variety, $Y \subset X$ a simple normal crossings divisor,
$f \in \Oh(X)$.
Let $\bar{X}$ be a good compactification of $(X,Y)$ relative to $f$, i.e., such that $Y + X_\infty$ is a simple normal crossing divisor and $f$ extends to $\bar{X}$. Let $\bar{Y}$ be the closure of $Y$ in $\bar{X}$.
Denote by $B_{\bar X}(Y)$, $B_{\bar X}^\circ(Y,f)$ and $\Bsharp_{\bar{X}}(Y,f)$ the closure of $Y^\an$
in $B_{\bar X}(X)$, $B_{\bar X}^\circ(X,f)$ and $\Bsharp_{\bar{X}}(X,f)$, respectively.
As in the last section let $Y_\bullet\to Y$ and $\bar{Y}_\bullet\to \bar{Y}$ be the \v{C}ech-complexes
for the closed covers of $Y$ and $\bar{Y}$ by their irreducible components.

Applying our oriented blow-ups,
we get simplicial $k_0$-semi-algebraic manifolds with corners $B_{\bar{Y}_\bullet}(Y_\bullet)$
and $ B_{\bar{Y}_m}^\circ(Y_\bullet, f_\bullet)$
and  $k_0$-semi-algebraic subsets $\Bsharp_{\bar{Y}_\bullet}(Y_\bullet)$.
Note that 
\begin{align*}
 B_{\bar{Y}_m}(Y_m,f_m)&=\coprod_{J\in\{1,\dots,N\}^{n+1}} B_{\bar{Y}^J}(Y^J,f^J),\\
 \Bcirc_{\bar{Y}_m}(Y_m,f_m)&=\coprod_{J\in\{1,\dots,N\}^{n+1}} \Bcirc_{\bar{Y}^J}(Y^J,f^J),\\
 \Bsharp_{\bar{Y}_m}(Y_m,f_m)&=\coprod_{J\in\{1,\dots,N\}^{n+1}} \Bsharp_{\bar{Y}^J}(Y^J,f^J).
\end{align*}

\begin{proposition}\label{prop:hypercover_B}
 The simplicial $k_0$-semi-algebraic sets
 \[
  B_{\bar{Y}_\bullet}(Y_\bullet)\to B_{\bar{X}}(Y),\quad
  \Bcirc_{\bar{Y}_\bullet}(Y_\bullet,f_\bullet)\to \Bcirc_{\bar{X}}(Y,f),\quad
  \Bsharp_{\bar{Y}_\bullet}(Y_\bullet,f_\bullet)\to \Bsharp_{\bar{X}}(Y,f)\
 \]
 are the \v{C}ech-nerves for the corresponding closed covers
\[ B_{\bar{Y}_0}(Y_0)\to B_{\bar{X}}(Y), \quad
  \Bcirc_{\bar{Y}_0}(Y_0,f_0)\to \Bcirc_{\bar{X}}(Y,f),\quad
  \Bsharp_{\bar{Y}_0}(Y_0,f_0)\to \Bsharp_{\bar{X}}(Y,f).\
 \]
\end{proposition}
 \begin{proof}
Fix $m\geq 0$ and  $J\subset \{1,\dots,N\}^{m+1}$. Let $Z=Y^J$. 
Then $\bar{Z}$ is transverse to $X_\infty$.
In suitable local coordinates $z_1,\dots,z_d$ of $\bar{X}$ centered at a point in $\bar{Z}\ssm Z$, the divisor
$X_\infty$ is defined by $z_1\cdots z_a=0$ and $\bar{Z}$ by $z_{a+1}=\dots=z_b=0$ with $1\leq a<b\leq d$. The real oriented blow-up of $\bar{X}$ in $X_\infty$ has
local coordinates 
\[ (r_1,\dots,r_a,\omega_1,\dots,\omega_a,z_{a+1},\dots,z_{d})\in\R_{\geq 0}^a\times (S^1)^a\times \C^{d-a}.\] 
The real oriented blow-up of $\bar{Z}$ in $Z_\infty$ has local coordinates 
\[ (r_1,\dots,r_a,\omega_1,\dots,\omega_a,0,\dots,0,z_{b+1},\dots,z_{d})\in\R_{\geq 0}^a\times (S^1)^a\times \C^{d-a}.\] 

This gives
 \begin{align*}
  B_{\bar Z}(Z) = \bar Z^\an \times_{\bar X} B_{\bar X}(X),\quad
  B_{\bar Z}^\circ(Z,f) &= \bar Z^\an \times_{\bar X^\an} B_{\bar X}^\circ(X,f),\\
 \quad
 \Bsharp_{\bar{Z}}(Z,f)&=\bar Z^\an \times_{\bar X^\an} \Bsharp_{\bar X}(X,f).
 \end{align*}
In total we have
  \begin{align*}
B_{\bar{Y}_\bullet}(Y_\bullet)&=\bar{Y}_\bullet\times_{\bar{X}}B(X),\\
\Bcirc_{\bar{Y}_\bullet}(Y_\bullet,f_\bullet)&= \bar{Y}_\bullet \times_{\bar{X}}\Bcirc(X,f),\\
\Bsharp_{\bar{Y}_\bullet}(Y_\bullet,f_\bullet)&= \bar{Y}_\bullet \times_{\bar{X}}\Bsharp(X,f).
\end{align*}
This gives the claim on \v{C}ech-nerves.
\end{proof}

\begin{cor}\label{cor:alter}
	Let $X$ be a smooth variety, $f\in\Oh(X)$ a regular function, $Y\subset X$ a simple normal crossings divisor.
	Choose a good compactification $\bar{X}$ of $(X,Y)$ relative to $f$.
	Let $\Bcirc_{\bar{X}}(Y,f)$ and $\Bsharp_{\bar{X}}(Y,f)$ be the closure of $Y^\an$ in $\Bcirc_{\bar{X}}(X,f)$ and $\Bsharp_{\bar{X}}(X,f)$, respectively.
	Then
	\[ H_n^\rd(Y,f)\isom H_n(\Bcirc_{\bar{X}}(Y,f),\partial \Bcirc_{\bar{X}}(Y,f);\Q)
	\isom H_n(\Bsharp_{\bar{X}}(Y,f),\partial\Bsharp_{\bar{X}}(Y,f);\Q)\]
	and 
	\begin{align*}
		H_n^\rd(X,Y,f)&\isom
		H_n(\Bcirc_{\bar{X}}(X,f), \Bcirc_{\bar{X}}(Y,f)\cup \partial \Bcirc_{\bar{X}}(X,f); \Q)\\
		&\isom H_n(\Bsharp_{\bar{X}}(X,f),\Bsharp_{\bar{X}}(Y,f)\cup\partial\Bsharp_{\bar{X}}(X,f); \Q).
	\end{align*}
\end{cor}
\begin{proof}
	Recall that the case of rapid decay homology of smooth $X$ (not relative to $Y$) was settled in Proposition~\ref{prop:comp_homol}.
 
We now address the case of $Y$, which is not smooth but a divisor with normal crossings. 
Let $Y_\bullet\to Y$ be the \v{C}ech-nerve of the closed cover of
$Y$ by the disjoint union of its irreducible components.
By Lemma~\ref{lem:rd_absolute}, the rapid decay homology $H_r^\rd(Y,f)$ is computed via the hypercover $Y_\bullet$, more precisely by singular homology of $\Bcirc_{\bar{Y}_\bullet}(Y_\bullet,f_\bullet)$ relative to
$\partial\Bcirc_{\bar{Y}_\bullet}(Y_\bullet,f_\bullet)$. By Proposition~\ref{prop:comp_homol} we may replace $\Bcirc$ by $\Bsharp$ everywhere.

By \Cref{prop:hypercover_B}, the natural map
\[ \Bcirc_{\bar{Y}_\bullet}(Y_\bullet,f_\bullet)\to \Bcirc_{\bar{X}}(Y,f)\]
	is the \v{C}ech-nerve of a closed cover,  hence it induces isomorphisms on singular homology.
The same argument also works for $\Bsharp$.
In all, we have 
\[ H_n^\rd(Y,f)\isom H_n(\Bcirc_{\bar{Y}_\bullet}(Y_\bullet,f),\partial B_{\bar{Y}_\bullet}(Y_\bullet,f))\isom H_n(\partial \Bcirc_{\bar{X}}(Y,f)\]
(and the same with $\Bsharp$), as claimed.

Finally, the claims for relative rapid decay homology follow from the absolute case via the comparison of long exact sequences.
\end{proof}

\subsection{Semi-algebraic triangulations of hypercovers}
We use the notation of Section~\ref{ssec:hypercover}. In particular, $U_\bullet$ and $Y_\bullet$ are the \v{C}ech-nerves of certain open and closed covers introduced there.

Note that the natural map $B_{\bar{X}}(U^J)\to B_{\bar{X}}(X)$ induces
an inclusion $\Bcirc_{\bar{X}}(U^J,f)\subset \Bcirc_{\bar{X}}(X,f)$.

\begin{proposition} \label{prop:rd_simplicial}
 There is a finite dimensional subcomplex
 \[ S^\Delta_*(X,Y,f)\subset S^\rd_*(X,Y,f)\]
 such that the inclusion is a quasi-isomorphism and every $S^\Delta_n(X,Y,f)$
 has a finite basis consisting of $k_0$-semi-algebraic $C^1$-simplices of the form
 \[ \sigma:\bar{\Delta}_a\to \Bsharp(U_b) \quad a+b=n\]
 or
 \[ \sigma:\bar{\Delta}_a\to \Bsharp(U_b\cap Y_c)\quad a+b+c=n-1\]
 such that $\sigma$ is a homeomorphism onto its image. 
\end{proposition}
\begin{proof}
 By definition, $S^\rd_*(X,Y,f)$ is the total complex of
 \[ \cone (S_*(\Bcirc(U_\bullet\cap Y_\bullet,f),\partial B)\to S_*(\Bcirc(U_\bullet,f),\partial B)\]
	(where $\partial B$ is an abbreviation for $\partial \Bcirc(U_\bullet\cap Y_\bullet,f)$ and $\partial \Bcirc(U_\bullet,f)$, respectively).
 
 In order to unify notation, we write $Y_{-1}=X$.
 The non-sensical term $Y^I$ for $|I|=-1$ is interpreted as $X$.
 We now want to choose compatible $k_0$-semi-algebraic triangulations
 in the sense of \xref{sec:triangle}. We first triangulate the
 base $B_{\bar{X}}(X)$ by applying \xref{prop:triangle_manifold}.
 We obtain a $k_0$-semi-algebraic triangulation of the 
 compact $k_0$-semi-algebraic manifold with corners $B_{\bar{X}}(X)$
 compatible with the finitely many $k_0$-semi-algebraic subsets $\Bsharp_{\bar{X}}(U^J\cap Y^I,f)$ and their boundaries.

In the next step, we want to triangulate the (bi)simplicial $k_0$-semi-algebraic sets
$\Bsharp_{\bar{X}}(U_\bullet\cap Y_\bullet,f)$ and their boundaries such that all structure maps and the maps between them are simplicial. We obtain this simply by pull-back of the triangulation of the base.
 By loc.~cit.\ the simplices can be chosen to be $C^1$.

 We apply \xref{lem:cc} to these (bi)simplicial complexes
 and replace them by the closed core of their barycentric subdivisons.
 The simplicial complexes  
\[ K_{ab}:=\cc{\beta(\Bsharp(U_a\cap Y_b,f))}\] 
and
\[ K_{ab}\cap\partial B=\cc{\beta(\partial\Bsharp(U_a\cap Y_b,f)}\]
 are deformation retracts by Lemma~\ref{lem:cc},
 hence $(|K_{ab}|,|K_{ab}\cap\partial B|)$
 has the same homology as  $(\Bsharp(U_a\cap Y_b,f),\partial B)$.
 By \cref{prop:comp_homol} their homology also agrees with the homology of
 $(\Bcirc(U_a\cap Y_b,f),\partial B)$.
 As closed subsets of the compact $B_{\bar{X}}(X)$, all $K_{ab}$ are compact, hence finite.
Recall that their homology can be computed by simplicial homology, defined by finite dimensional complexes with the simplices as a degreewise basis. We  put
\[ S_*^\Delta(U_a\cap Y_b,\partial B)=S_*^\Delta(K_{ab})/S_*^\Delta(K_{ab}\cap \partial B).\]

The construction is compatible with the differentials in the $a$ and $b$-direction.
By taking total complexes, we thus define the subcomplexes 
 of 
\begin{align*}
S_*^\Delta(U_\bullet\cap Y_\bullet,\partial B)\subset&\ S_*(\Bcirc(U_\bullet \cap Y_\bullet,f),\partial B)\quad\text{and}\\
S_*^\Delta(U_\bullet,\partial B)\subset&\ S_*(\Bcirc(U_\bullet,f),\partial B)
\end{align*}
 By what we argued above,
 the inclusion of subcomplexes into the ambient complexes
 are quasi-iso\-mor\-phisms.
 Let 
\[ S^\Delta_*(X,Y,f)=\cone\left(S_*^\Delta(U_\bullet\cap Y_\bullet,\partial B)\to S_*^\Delta(U_\bullet,\partial B)\right)[-1].\]
In degree $n$, it is given by 
the finitely many simplices of dimension $a$ in the chosen triangulation of
$|\cc{\beta\Bsharp(U_b\cap Y_c,f)}|$ with $a+b+c=n-1$ where $a,b\geq 0, c\geq -1$.
\end{proof}

\subsection{Proof of Proposition~\ref{prop:coh_implies_naive}}

\begin{proof}
Let $\alpha=\langle \Omega,\Sigma\rangle$ be an exponential period for the log-pair $(X,Y)$ over $f$. We want to express it as a naive exponential period.

We work with the hypercovers $U_\bullet$ and $Y_\bullet$ as in Section~\ref{ssec:hypercover}.
By definition, $\Sigma\in H_n^\rd(X,Y,f)$. We compute rapid decay homology via the complex $S^\Delta_*(X,Y,f)$ of Proposition~\ref{prop:rd_simplicial}.
By definition this means that the cohomology class $\Sigma$ is represented by
a tuple $(\sigma_{bc})\in \bigoplus S^\Delta_a(U_b\cap Y_c)$ with $a+b+c=n-1$, $b\geq 0$, $c\geq -1$.

Also by definition, $\Omega$ is represented by a cycle in
$\Omega^*(\cone(U_\bullet\cap Y_\bullet\to U_\bullet))$,
i.e., a tuple
$(\omega_{bc})\in \bigoplus\Omega^a(U_b\cap Y_c)$ with $a+b+c=n-1$, $b\geq 0$, $c\geq -1$
(again we use the convention that $Y_{-1}=X$).
By definition of the period pairing, the value $\langle\Omega,\Sigma\rangle$
is obtained by taking a linear combination of the integrals
\[ \int_{\sigma_{bc}}\e^{-f}\omega_{bc}.\]
Each of the $\sigma_{bc}$ is a linear combination of $k_0$-semi-algebraic simplices globally of class $C^1$ with values in $\Bsharp(U_b\cap Y_c)\subset \Bcirc(U_b\cap Y_c)$.

Recall that naive exponential periods form an algebra, hence it suffices to show that the integrals for the individual simplices define naive exponential periods.

Let $U=U_b\cap Y_c\subset\A^N$, $\omega=\omega_{bc}\in\Omega^a(U)$.
Let $T \colon \bar{\Delta}_a\to \Bsharp(U,f)$ be
a   $k_0$-semi-algebraic $C^1$-simplex.
Let $G=T(\bar{\Delta}_a)\cap U^\an$.
We equip it with the pseudo-orientation induced from $\Delta_a$.
It is a closed $k_0$-semi-algebraic subset of $\C^N$
because $U$ is affine and
the inclusion $U^\an\to \Bsharp_{\bar{U}}(U,f)$ is $k_0$-semi-algebraic.
Moreover, as $U$ is affine,
$f|_G$ is the restriction of a polynomial in $k[X_1,\dots,X_N]$ to~$G$
and $\omega|_G$ the restriction of an algebraic differential form. 

We need to check the condition on $f(G)$.
The closure $\bar{G}=T(\bar{\Delta}_a)\subset\Bsharp(U,f)$ is compact, hence
so is its image in $\Bsharp=\C\cup \{1\infty\}$. 
This implies that
$f(G)\subset \C$ is contained in a strip $S_{r,s}$ as we want.
Compactness of $\bar{G}$ also implies that the map $\bar{G}\to \Ptilde$ is proper.
The preimage of the circle at infinity is precisely $\bar{G}\ssm G$,
hence $f \colon G\to \C$ is also proper. 

Therefore our $\alpha$ is a linear combination of numbers of the form
\[ \int_{\bar{\Delta}_a}\e^{-f\circ T}T^*\omega=\int_G\e^{-f}\omega, \]
which are naive exponential periods.
\end{proof}

\section{Generalised naive exponential periods are cohomological}
\label{sec:gnaive_is_coh}

Let $k\subset\C$ be a subfield,
$k_0=k\cap\R$ and assume that $k$ is algebraic over $k_0$,
see \Cref{ssec:field2}.
Recall from \cref{defn:gnaive} the notion of a generalised naive exponential period.
We denote by $\Pgnaive(k)$ the set of generalised naive exponential periods.
Recall from \Cref{defn:coh_period}
the notion of an exponential period of a log pair
and the set $\Plog(k)$ of all such numbers.

The aim of this section is the proof of the following converse of
Proposition~\ref{prop:coh_implies_naive}:

\begin{proposition}
 \label{naive_is_effective}
 Every generalised naive exponential period over $k$
 is an exponential period of a log-pair over $k$:
 \[ \Pgnaive(k)\subset\Plog(k).\]
  More precisely, given 
\begin{itemize}
\item a pseudo-oriented $k$-semi-algebraic subset $G\subset\C^n$ of real dimension $d$,
\item  a rational function $f$
 and 
\item a rational algebraic $d$-form $\omega$
\end{itemize} 
as in the definition of a generalised naive exponential period,
 there are
\begin{itemize}
\item 
a smooth affine variety $X$ of dimension $d$,
\item a simple normal crossings divisors $Y$ on $X$,
\item a function $f\in\Oh(X)$ induced from the original $f$,
\item a homology class $[G]\in H_d^\rd(X,Y,f)$, 
 and 
\item a cohomology class $[\omega]\in H^d_\dR(X,Y,f)$
\end{itemize}
 such that
 \[\langle [\omega],[G]\rangle =\int_G\e^{-f}\omega.\]
\end{proposition}

\subsection{Horizontal divisors}

We will need to make the closure of $G$ disjoint
from the components of the divisor that are horizontal relative to~$f$.
We start with a local criterion.

\begin{lemma}
 \label{very-good-cmpt-aux}
 Let $k \subset \R$ be a real closed field, so that $\bar k = k(i)$.
 Let $E \subset \A^n_k$ be a union of coordinate hyperplanes, i.e.,
 $E = \left\{ \prod_{j \in J} x_j = 0 \right\}$,
 with $J \subset \{ 1, \dots, n \}$.
 Let $G$ be a semi-algebraic subset of~$\A^n_k(\R) = \R^n$
 such that $\bar G$ contains the origin.
 Let $\partial \bar{G}=\bar{G}\ssm \bar{G}^\interior$ be its boundary in $\R^n$.
 Let $U \subset \A^n_k(\R)$ be an open neighbourhood of the origin,
 and assume that $G \cap U$ is open in $\R^n$
 and $\partial \bar{G} \cap U \subset E(\R)$.
 Then $G$ meets all coordinate hyperplanes $D_i = \{x_i=0\}$ for $i\notin J$.
\end{lemma}

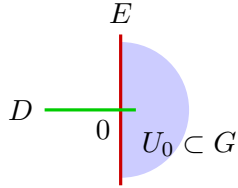
\begin{center}
 \begin{tikzpicture}
  \fill [fill=blue!20!white] (0, -.9) arc (-90:90:.9);
  \draw [red!80!black, very thick]
  (0, -1) -- (0, 1) node[above, black] {$E$};
  \draw [green!80!black, very thick]
  (.2, 0) -- (-1, 0) node[left, black] {$D_i$};
  \node [anchor=north east] at (0,0) {$0$};
  \node at (0.9, -.45) {$U_0 \subset G$};
 \end{tikzpicture}
\end{center}

\begin{proof}
 Without loss of generality,
 we may assume that $U$ is an open ball.
 Note that $U \setminus E(\R)$ has $2^{|J|}$ connected components.
 Since $G \cap U$ is open, and $\bar G$ contains the origin,
 we see that $G$ intersects at least one of these components, say $U_0$.
 Since $\partial \bar{G} \cap U \subset E(\R)$, we find that $U_0 \subset G$.
 On the other hand, for every $i \notin J$,
 it is clear that $D_i = \{x_i = 0\}$ intersects~$U_0$.
\end{proof}

\begin{setting}\label{setting:final}
For the actual proof of \Cref{naive_is_effective},
we are going to use the following data:
\begin{itemize}
\item a real closed field $k\subset\R$, hence $k(i)=\bar{k}$,
\item a smooth affine variety $X$ over $k$ of dimension $d$, 
\item a simple normal crossings divisor $Y\subset X$, 
\item a closed $k$-semi-algebraic subset $G\subset X(\R)$ of dimension $d$ such that $\partial G\subset Y(\R)$
 (where $\partial G=G\ssm G^\interior$ inside $X(\R)$),
\item a pseudo-orientation on $G$,
\item a morphism $f \colon X_{\bar{k}}\to \A^1_{\bar{k}}$
 such that $f \colon G\to\C$ is proper and such that the closure
$\overline{f(G)}\subset\Ptilde$ is contained in $\Bcirc$,
\item a regular algebraic $d$-form $\omega$ on $X_{\bar{k}}$,
\item a good compactification $\bar{X}$ of $X$ relative to $f$,
\item and finally, we denote by $D \subset \bar{X}$ 
 the smallest subvariety of $\bar{X}$ containing
 all components of $(X_\infty)_{\bar k} = \bar X_{\bar k} \ohne X_{\bar k}$
		on which $\bar{f}$ is rational, i.e., not the constant function $\infty$.
\end{itemize}
\end{setting}

\begin{lemma}
 \label{very-good-cmpt}
 In this setting, we may choose $\bar{X}$ such that,
 in addition to being a good compactification,
 the closure of~$G$ in $\bar{X}^\an$ is disjoint from~$D^\an$.
\end{lemma}

\begin{proof}
 Without loss of generality, we may assume that $X$ is connected.
 If $D$ is empty, we are done.
 Hence assume that $D$ is not empty.
 By the properness assumption on~$f$,
 we see that $\bar G \cap D^\an$ lies in the preimage of~$\infty \in \Pe^1$.

 Let $\bar{Y}$ be the Zariski closure of~$\partial \bar{G} \cup Y$ in $\bar X$,
 where $\partial \bar G = \bar G \ssm (\bar G)^\interior$
 viewed as subset of $\bar X$.
 It contains the closure of $Y$ in $\bar{X}$,
 but possibly also additional components mapping to~$\infty$. 
 By resolution of singularities, we may find a modification
 $\pi \colon \tilde X \to \bar X$ such that $\pi^{-1}(X)\to X$ is an isomorphism,
 with $\tilde{X}$ again smooth
 and such that $\tilde D \cup E \cup \tilde Y$
 is a simple normal crossings divisor in~$\tilde X$,
 where $\tilde D$ and~$\tilde Y$ denote the strict transforms
 of~$D$ and~$\bar Y$ respectively,
 and where $E$ denotes the exceptional locus of~$\pi$.
 In addition, we may assume that $\tilde D$ and~$\tilde Y$ are disjoint.

 Let $\tilde G$ denote the strict transform of~$\bar G$ under~$\pi$,
 i.e., the closure of $G\isom \pi^{-1}(G)$ in $\tilde{X}^\an$.
 It is contained in $\tilde{X}(\R)$ and, by continuity, in the preimage of
$\bar{G}$.
 This means that $\partial \tilde G \subset E \cup \tilde Y$.
	We identify $G$ and~$X$ with their preimages under~$\pi$.

 \begin{center}
  \begin{tikzpicture}
   \begin{scope}[xshift=0]
    \draw [help lines, thick] (0, -0.2) -- (0, 3.2)
    node [above, black] {$\bar f^{-1}(\infty)$};
    \draw [green!80!black, very thick]
    (-0.2, 0) -- (3.2, 0) node[right, black] {$D$};
    \draw [draw=blue, fill=blue!20!white, smooth, domain=0:3, samples=20]
    (0.5, 3) -- plot (\x, \x^2 / 3) node [right] {$\partial G \subset Y$};
    \node at (1.2, 1.8) {$G$};
   \end{scope}
   \begin{scope}[xshift=6cm]
    \draw [draw=blue, thick, fill=blue!20!white, smooth, domain=0:3, samples=20]
    (0.5, 3) -- plot (\x, \x^2 / 3)
    node [right] {$\tilde Y$};
    \fill [fill=white]
    (0.6, 0) circle (.8)
    (0, 0.6) circle (.8);
    \draw [help lines, thick] (0, 1.2) -- (0, 3.2);
    \draw [green!80!black, very thick]
    (1.2, 0) -- (3.2, 0) node[right, black] {$\tilde D$};
    \draw [red!80!black, very thick]
    ([shift=(-10:.8)]0.6, 0) arc (-10:100:.8)
    ([shift=(-10:.8)]0, 0.6) arc (-10:100:.8)
    node[left, black] {$E$};
    \node at (1.2, 1.8) {$\tilde G$};
   \end{scope}
  \end{tikzpicture}
 \end{center}

 We will now show that $\tilde D(\R)$ is disjoint from
  $\tilde G$ in~$\tilde X(\R)$.
 Suppose that $x$ is contained in their intersection.
 Since $\tilde Y$ is disjoint from $\tilde D$,
 we conclude that $x \in E(\R)$.
 As $\tilde Y$ is closed,
 there is even an open neighbourhood~$U$ of~$x$ in~$\tilde X(\R)$
 such that $U$ is disjoint from $\tilde Y(\R)$. This implies that $\partial \tilde G\cap U\subset E(\R)$, and that
$G\cap U=(G\ohne Y)\cap U
=(\tilde{G}\ohne \partial \tilde G)\cap U$
is open in $\tilde X(\R)$.

 We choose suitable continuous semi-algebraic local coordinates around~$x$,
 such that $\tilde D$ and $E$ are unions of coordinate hyperplanes.
 Then \Cref{very-good-cmpt-aux} tells us that $G\cap U$ intersects $\tilde D(\R)$,
 which contradicts the fact that $G\cap U \subset X(\R)$ and $\tilde D \subset \tilde X \ohne X$.
 As desired $\tilde G$  is disjoint from~$\tilde D(\R)$.

We replace $\bar{X}$ by $\tilde{X}$ in Setting~\ref{setting:final}. 
 Since $\tilde f = \bar f \circ \pi$ is not rational on~$E$, the divisor $\tilde{D}$ takes the role of $D$. 
 We conclude that $\tilde X$ satisfies the conditions of the statement.
\end{proof}

\subsection{Proof of \Cref{naive_is_effective}}

\begin{proof}
 Let $\alpha$ be a generalised naive exponential period.
By \xref{lem:change_field} and \cref{lem:change_field_coh}, we may assume
 without loss of generality that $k\subset\R$ and that $k$ is real closed and hence $k(i)=\bar{k}$.
Generalised naive exponential periods are absolutely convergent, so we can use the 
characterisation of \xref{prop:geometry}. 
This brings us into \Cref{setting:final} with
\[ \alpha=\int_G\e^{-f}\omega.\]
By \Cref{very-good-cmpt},
 we may improve the good compactification~$\bar X$
 in such a way that the closure of $G$ in $\bar{X}^\an$ is disjoint from the components of~$X_\infty$
 on which $\bar f$ has a pole. 
 This implies that the closure $\bar{G}$ of $G$ in the real oriented blow-up $B_{\bar{X}}(X)$ is contained in $\Bcirc_{\bar{X}}(X,f)$. Note that $\bar{G}$ is compact because $B_{\bar{X}}(X)$ is so. All constructions in the rest of the argument take place on 
$X_{\bar{k}}$, e.g. $f$ is defined there. To simplify notation we write $X$ instead of 
$X_{\bar{k}}$.

Let $Y_\bullet $ and their compactifications be as in
Section~\ref{ssec:hypercover}. Note that we do not need to pass to an open \v{C}ech-cover because $X$ is affine. By definition
\[
 \RdR(X,Y)^d= \Omega^d(X)\oplus \Omega^{d-1}(Y_0)\oplus\dots\oplus\Omega^0(Y_{d-1}).
\]
The tuple $(\omega,0,\dots,0)$ is a cocycle because $d\omega=0$ and
$\omega|_{Y_0}=0$, both for dimension reasons.
We denote the induced cohomology class by
\[ [\omega]\in H^d_\dR(X,Y,f).\]

 Recall that $G$ is equipped with a pseudo-orientation.
  Let $G' \subset G$ be an oriented semi-algebraic subset
  with $\dim(G \ssm G') < d$
  that represents the pseudo-orientation.
We apply \xref{prop:triangle_manifold} to the semi-algebraic manifold with corners
 $\Bcirc_{\bar{X}}(X,f)$.
 Hence we may choose a semi-algebraic triangulation of~$\bar{G}$
 that is globally of class~$C^1$
 and that is
 compatible with
 the oriented subset~$G'$
  and also compatible with
 the subsets $\Bcirc_{\bar{X}}(Y^J) \cap \bar{G}$
 and
 $\partial \Bcirc_{\bar{X}}(Y^J,f)\cap\bar{G}$ for all~$J$.
 Here $Y^J$ is the intersection of irreducible components
 of $Y$ as in Section~\ref{ssec:hypercover}.
 The top dimensional simplices inherit an orientation from $G'$.
 We use the triangulation of $\bar{G}$
 to define a cycle $(\sigma,\sigma_0,\dots,\sigma_{d-1})$ in
\begin{multline*}
 S_d^\rd(\cone(Y_\bullet\to X))=\\
S_d(\Bcirc(X,f),\partial)\oplus S_{d-1}(\Bcirc(Y_0,f_0),\partial)\oplus\dots\oplus S_0(\Bcirc(Y_{d-1},f_{d-1},\partial)
\end{multline*}
where we abbreviate
\[
S_{n}(\Bcirc(Y_i,f_i),\partial) =
S_{n}(\Bcirc_{\bar{Y}_i}(Y_i,f_i))/S_{n}(\partial \Bcirc_{\bar{Y}_i}(Y_i,f_i)).
\]

 In detail: We are given  a simplicial complex $K$ and  a homeomorphism $h:|K|\to \bar{G}$ which extends to a $C^1$-map on a neighbourhood of $|K|$. For each closed top-dimensional simplex $a=[a_0,\dots,a_d]\in K$, we choose a linear isomorphism $\bar{\Delta}_d\to [a_0,\dots,a_d]$. By composition we obtain a
$C^1$-map
\[ T_a:\bar{\Delta}_d\to [a_0,\dots,a_d]\xrightarrow{h|_{[a_0,\dots,a_d]}}  \Bcirc_{\bar{X}}(X,f).\]
It is a homeomorphism onto its image.
The image of $T_a$ is oriented by the orientation on $G'$.
We can arrange for $T_a$ to respect this orientation.
The formal linear combination 
\[ \sigma=\sum_{a\in K_d} T_a\]
 is a chain on $\Bcirc_{\bar{X}}(X,f)$.
Its boundary $\check{\partial} \sigma$ is a linear
 combination of $(d-1)$-simplices with image contained in one of the components $Y^i$ or in $\partial \Bcirc_{\bar{X}}(X,f)$.
 Let $\sigma_0\in S_{d-1}(Y_0)$ be the chain defined by the simplices in the $Y^i$, ignoring the ones with image contained in $\partial\Bcirc_{\bar{X}}(X,f)$.
 By construction, the simplices appearing in $\check{\partial} \sigma_0$
 are contained in one of the $Y^{ij}$, hence they define $\sigma_1\in S_{d-2}(Y_1)$.
 Recursively, we find all $\sigma_a$. By construction,
$(\sigma,\sigma_0,\dots,\sigma_{d-1})$ is a cycle. 
Let 
\[ [G]\in H_d^\rd(X,Y,f)\]
 be its homology class. Because of the special shape of $[\omega]$, we have
\[ \langle [\omega],[\sigma]\rangle =\sum_{a\in K_d} \int_{\bar{\Delta}_d}T_a^*\omega
  =\sum_{a\in K_d}\int_{T_a(\bar{\Delta}_d)}\omega=\int_G\omega.\]
We have written $\alpha$ as a cohomological period over $\bar{k}$. 
\end{proof}

\section{Conclusion}
\label{sec:concl}

Fres\'an and Jossen develop a fully fledged theory of exponential motives
in~\cite{fresan-jossen}.
It behaves very much like the theory of ordinary Nori motives.
In particular, there is a so-called ``basic lemma'' for
affine pairs $(X,Y,f)$. We refer to their book for further details.
We denote by $\Pmot(k)$ the set of periods of effective exponential motives.

\begin{proposition}
 \label{mot_sub_coh}
 The periods of effective exponential motives
 are exponential periods in the sense of \Cref{defn:coh_period}
 for a tuple $(X,Y,f,n)$ with $X$ smooth, $Y$ a strict normal crossings divisor and $n=\dim X$. In other words,
\[ \Pmot(k)\subset\Plog(k).\]
\end{proposition}
\begin{proof}
 By \cite[Theorem~4.3.2]{fresan-jossen} (and its proof, in particular \cite[Proposition~4.3.7]{fresan-jossen}) 
 every effective exponential motive $M$
 is a subquotient of some exponential motive of the form $H^n(X,Y,f)$
 for an affine $k$-variety $X$, a closed subvariety $Y\subset X$ and a regular function
 $f\in\Oh(X)$. Without loss of generality, $Y$ can be chosen to contain all singularities of $X$, so that  $X\ssm Y$ is smooth.
	Hence the periods of $M$ are also periods of $H^n(X,Y,f)$. 

There is a blow-up
 $\pi \colon \tilde{X}\to X$ such that $\tilde{X}$ is smooth
 and $\tilde{Y}=\pi^{-1}(Y)$ is a simple normal crossings divisor.
		Note that excision holds for rapid decay homology because it holds for singular homology. Hence
	we obtain an isomorphism
 \[ H^\rd_n(\tilde{X},\tilde{Y},f)\isom H^\rd_n(X,Y,f).\]
This isomorphism lifts to an isomorphism of motives.
 Hence they have the same periods.
\end{proof}
\begin{rem}
	We proved an even stronger result in \cref{naive_is_effective}: all exponential periods are even realised as cohomological exponential periods of log-pairs with $X$ smooth and \emph{affine}. 
\end{rem}

\begin{proposition}
 \label{smaff_sub_mot}
Periods of complexes of smooth affine varieties are periods of effective
exponential Nori motives, i.e.,
\[ \Psmaff(k)\subset\Pmot(k).\]
\end{proposition}
\begin{proof}The argument is the same as in the case of ordinary Nori motives,
 see \cite[Theorem~11.4.2]{period-buch}. We give a sketch of the proof.

 By \cite[Corollary~3.3.3]{fresan-jossen}, we may choose a good filtration
 $F_0X\subset F_1X\subset\dots F_nX=X$ of an affine variety $X$, i.e., one where in every step the relative homology is concentrated in a single degree equal to the dimension.
 By \cite[Lemma~4.3.10]{fresan-jossen} the exponential motives 
 $H_n(X,f)$ for $n\geq 0$ of
$X$ are computed as the homology of the complex of exponential Nori motives
\[ \cdots H_{i+1}(F_{i+1}X,F_{i}X,f)\to H_i(F_iX,F_{i-1}X,f)\to\cdots.\]
Given a complex $X_\bullet$ of affine varieties, we may choose compatible good filtrations
 on all entries of the complex. The exponential motives of $X_\bullet$ are defined as the homology of the total complex of the double complex
$H_i(F_{i}X_j,F_{i-1}X_j,f_j)$. This is compatible with the period computation, hence we have identified the periods of $X_\bullet$ with the periods of exponential motives.
\end{proof}

\begin{thm}
 \label{thm:compare_all}
  Let $k\subset\C$ be a field, $k_0 = k \cap \R$,
 and assume that $k/k_0$ is algebraic.
 Then the following subsets of $\C$ agree:
 \begin{enumerate}
  \item $\Pnaive(k)$, i.e., naive exponential periods over~$k$;
  \item $\Pgnaive(k)$, i.e., generalised naive exponential periods over~$k$;
  \item $\Pabs(k)$, i.e., absolutely convergent exponential periods over~$k$;  
  \item $\Pmot(k)$, i.e., periods of all effective exponential motives over $k$;
  \item $\Pcoh(k)$, i.e., the set of periods of all $(X,Y,f,n)$
   with $X$ a $k$-variety, $Y\subset X$ a subvariety, $f\in\Oh(X)$, and $n\in\Na_0$;
  \item $\Plog(k)$, i.e., periods of all tuples $(X,Y,f,n)$ with $(X,Y)$ a log pair,
   $f\in\Oh(X)$, and $n\in\Na_0$;
  \item $\Psmaff(k)$, i.e., periods of all tuples $(X_\bullet,f_\bullet,n)$
   for complexes $(X_\bullet,f_\bullet)\in C_-(\SmAff/\A^1)$ and $n\in\Na_0$.
 \end{enumerate}
 Moreover, the real and imaginary part of these numbers are up to sign
 volumes of bounded definable sets
 for the o-minimal structure $\R_{\sin,\exp,k_0}$
 generated by $\exp$, ${\sin}|_{[0,1]}$ and with paramaters in $k_0$,
 see \xref{defn:Rexpsin}.
\end{thm}
\begin{proof}
 The statement on volumes of definable sets is \xref{thm:naive_is_volume}.

 The following diagram shows all the inclusions that we have proved
 between the sets listed above.
 \[
 \begin{tikzcd}
  \text{Vol}
  \\ \\
  \partAorB{\hyperref[defn:naive]{\Pnaive}}{\hyperref[defn:naive2]{\Pnaive}}
  \ar[uu, hook, "\text{\Aref{thm:naive_is_volume}}"]
  \ar[rr, hook, "\text{\Aref{naive-is-generalised}}"]
  &
  & \partAorB{\hyperref[defn:gnaive]{\Pgnaive}}{\hyperref[defn:naive2]{\Pgnaive}}
  \ar[ddl, hook, "\text{\Bref{naive_is_effective}}"]
  \ar[rr, equal, "\text{\Aref{cor:abs=naive}}"]
  &
  & \partAorB{\hyperref[defn:abs]{\Pabs}}{\hyperref[defn:naive2]{\Pabs}}
  \\ \\
  & \partAorB{\hyperref[defn:coh_period]{\Plog}}{\hyperref[defn:coh_period]{\Plog}}
  \ar[uul, hook, "\text{\Bref{prop:coh_implies_naive}}"]
  \ar[r, hook, "\text{triv}"]
  & \partAorB{\hyperref[defn:coh_period]{\Pcoh}}{\hyperref[defn:coh_period]{\Pcoh}}
	 \ar[r, hook, "\text{\hyperref[rem:is_smaff]{Rmk}}", "\labelcref{rem:is_smaff}"']
  & \partAorB{\Psmaff}{\hyperref[defn:smaff]{\Psmaff}}
  \ar[ddl, hook, "\text{\Bref{smaff_sub_mot}}"]
  \\ \\
  &
  & \partAorB{\Pmot}{\hyperref[sec:concl]{\Pmot}}
  \ar[uul, hook, "\text{\Bref{mot_sub_coh}}"]
 \end{tikzcd}
\]

 Therefore we have equality everywhere.
\end{proof}

 \bibliographystyle{alpha}
\bibliography{periods}

\end{document}